\documentclass[12pt,reqno]{amsart}

\usepackage{amsfonts,amsmath,amsthm}
\usepackage{amssymb,epsfig}
\usepackage{enumerate} 

\usepackage{color} 
\definecolor{vert}{rgb}{0,0.6,0}

\usepackage[text={425pt,650pt},centering]{geometry}
\usepackage{graphicx}
\usepackage{epsfig}
\usepackage{tikz,esvect}
\usepackage{tikz-3dplot}
\usetikzlibrary{3d,calc,shapes}
\usetikzlibrary{shapes.geometric}
\usepackage{adjustbox}
\usepackage{subfig}

\usepackage{caption}
\usepackage{color} 
\definecolor{vert}{rgb}{0,0.6,0}

\numberwithin{figure}{section}

\theoremstyle{plain}
\newtheorem{thm}{Theorem}[section]

\newtheorem{defn}{Definition}

\newtheorem{lem}[thm]{Lemma}
\newtheorem{cor}[thm]{Corollary}
\newtheorem{prop}[thm]{Proposition}
\theoremstyle{remark}
\newtheorem{rem}{\bf{Remark}}
\numberwithin{equation}{section}

\newcommand{\N}{\mathbb{N}}

\newcommand{\R}{\mathbb{R}}

\newcommand{\T}{\mathbb{T}}
\newcommand{\Z}{\mathbb{Z}}

\newcommand{\cA}{\mathcal{A}}

\newcommand{\cH}{\mathcal{H}}

\newcommand{\Lip}{{\rm Lip\,}}

\newcommand{\gam}{\gamma}

\newcommand{\ep}{\varepsilon}

\newcommand{\lam}{\lambda}

\newcommand{\om}{\omega}
\newcommand{\mdef}{m^{\rm d}}
\newcommand{\Omext}{\mathcal{O}}

\newcommand{\Om}{\Omega}

\newcommand{\ol}{\overline}

\newcommand{\dist}{{\rm dist}\,}

\DeclareMathOperator*{\esssup}{ess\,sup}

\begin{document}

\title[Homogenization on perforated domains and applications]
{Quantitative homogenization of state-constraint Hamilton--Jacobi equations on perforated domains and applications}

\author[Y. Han, W. Jing, H. Mitake, H. V. Tran] 
{Yuxi Han, Wenjia Jing, Hiroyoshi Mitake, and Hung V. Tran}

\thanks{
The work of YH is partially supported by NSF CAREER grant DMS-1843320 and the Bung-Fung Lee Torng Fellowship.
The work of WJ is partially supported by the China Ministry of Science and Technology under grant No.\,2023YFA1008900.
The work of HM was partly supported by JSPS through the grants Kakenhi: 22K03382, 21H04431, and 24K00531.
The work of HVT is partially supported by NSF CAREER grant DMS-1843320 and a Vilas Faculty Early-Career Investigator Award.
}

\address[Y. Han]
{
Department of Mathematics, 
University of Wisconsin Madison, Van Vleck Hall, 480 Lincoln Drive, Madison, Wisconsin 53706, USA}
\email{yuxi.han@math.wisc.edu}

\address[W. Jing]
{
Yau Mathematical Sciences Center, Tsinghua University, No.1 Tsinghua Yuan, Beijing 100084, China}
\email{wjjing@tsinghua.edu.cn}

\address[H. Mitake]
{Graduate School of Mathematical Sciences, University of Tokyo 3-8-1
Komaba, Meguro-ku, Tokyo, 153-8914, Japan}
\email{mitake@ms.u-tokyo.ac.jp}

\address[H. V. Tran]
{
Department of Mathematics, 
University of Wisconsin Madison, Van Vleck Hall, 480 Lincoln Drive, Madison, Wisconsin 53706, USA}
\email{hung@math.wisc.edu}

\date{\today}
\keywords{Cell problems; domain defects; first-order convex Hamilton--Jacobi equations; periodic homogenization; perforated domains; optimal convergence rate; state-constraint problems; viscosity solutions}
\subjclass[2010]{
35B10 
35B27 
35B40 
35F21 
49L25 
}

\maketitle

\begin{abstract}
We study the periodic homogenization problem of state-constraint Hamilton--Jacobi equations on perforated domains in the convex setting and obtain the optimal convergence rate.
We then consider a dilute situation in which the holes' diameter is much smaller than the microscopic scale.
Finally, a homogenization problem with domain defects where some holes are missing is analyzed.

\end{abstract}


\section{Introduction}
This paper studies the following homogenization problem on perforated domains and its applications. Consider an open and connected set $\Omega \subset \mathbb{R}^n$ with $C^1$ boundary and assume $\Omega$ is $\mathbb{Z}^n$-periodic, which means it satisfies the condition $\Omega+ \mathbb{Z}^n = \Omega$. Note that for $\Om$ to be connected, we need $n\geq 2$. For $\varepsilon >0$, consider the domain $\Omega_\varepsilon: = \varepsilon \Omega$ and let $u^\varepsilon$ be the unique viscosity solution to the state-constraint problem
\begin{equation}\label{eqn:PDE_epsilon}
    \left\{\begin{aligned}
    u_t^\varepsilon+H\left(\frac{x}{\varepsilon}, Du^\varepsilon\right) & \leq 0 \qquad \quad \text{in } \Omega_\varepsilon \times (0, \infty),\\
    u_t^\varepsilon+H\left(\frac{x}{\varepsilon}, Du^\varepsilon\right) & \geq 0 \qquad \quad \text{on } \overline{\Omega}_\varepsilon \times (0, \infty), \\
    u^\varepsilon(x,0) & =g(x) \,\,\,\, \quad \text{on } \overline{\Omega}_\varepsilon \times \{t=0\}, 
    \end{aligned}
    \right.
\end{equation}
where $H$ is a given Hamiltonian and $g$ is a given initial condition. Under appropriate assumptions, we expect that as $\varepsilon$ tends to $0$, $u^\varepsilon$ converges locally uniformly to $u$ on $\overline{\Omega}_\varepsilon \times [0,\infty)$, where $u$ is the unique viscosity solution to
\begin{equation}\label{eqn:PDE_limit}
    \left\{\begin{aligned}
    u_t+\overline{H} \left(Du\right) & = 0 \qquad \quad \text{in } \mathbb{R}^n \times (0, \infty), \\
    u(x, 0)&=g(x) \, \, \,\quad \text{on } \mathbb{R}^n \times \left\{t=0\right\}.
    \end{aligned}
    \right.
\end{equation}
Here $\overline{H}$ is the effective Hamiltonian determined by the state-constraint cell problem on $\ol \Omega$. More specifically, for any $p \in \mathbb{R}^n$, $\overline{H}(p)$ is the unique constant such that the following state-constraint cell problem has a $\Z^n$-periodic viscosity solution:
\begin{equation}\label{eqn:cell}
    \left\{\begin{aligned}
    H\left(y, p+Dv(y)\right) &\leq \overline{H}(p) \quad \text{in }  \Omega, \\
    H\left(y, p+Dv(y)\right) &\geq \overline{H}(p) \quad \text{on }  \overline{\Omega}.
    \end{aligned}
    \right.
\end{equation}
If needed, we write $\ol H= \ol H_\Om$ to denote the dependence of $\ol H $ on $\Om$.

We are particularly interested in how fast $u^\varepsilon$ converges to $u$. We prove that the convergence rate of $u^\varepsilon$ to $u$ is $O(\varepsilon)$, which is optimal, in the convex setting.
We then study related problems in perforated domains including a dilute situation in which the holes' diameter is much smaller than the microscopic scale $\ep$ and a problem with domain defects.
\medskip

\begin{center}
    \includegraphics[scale=0.15]{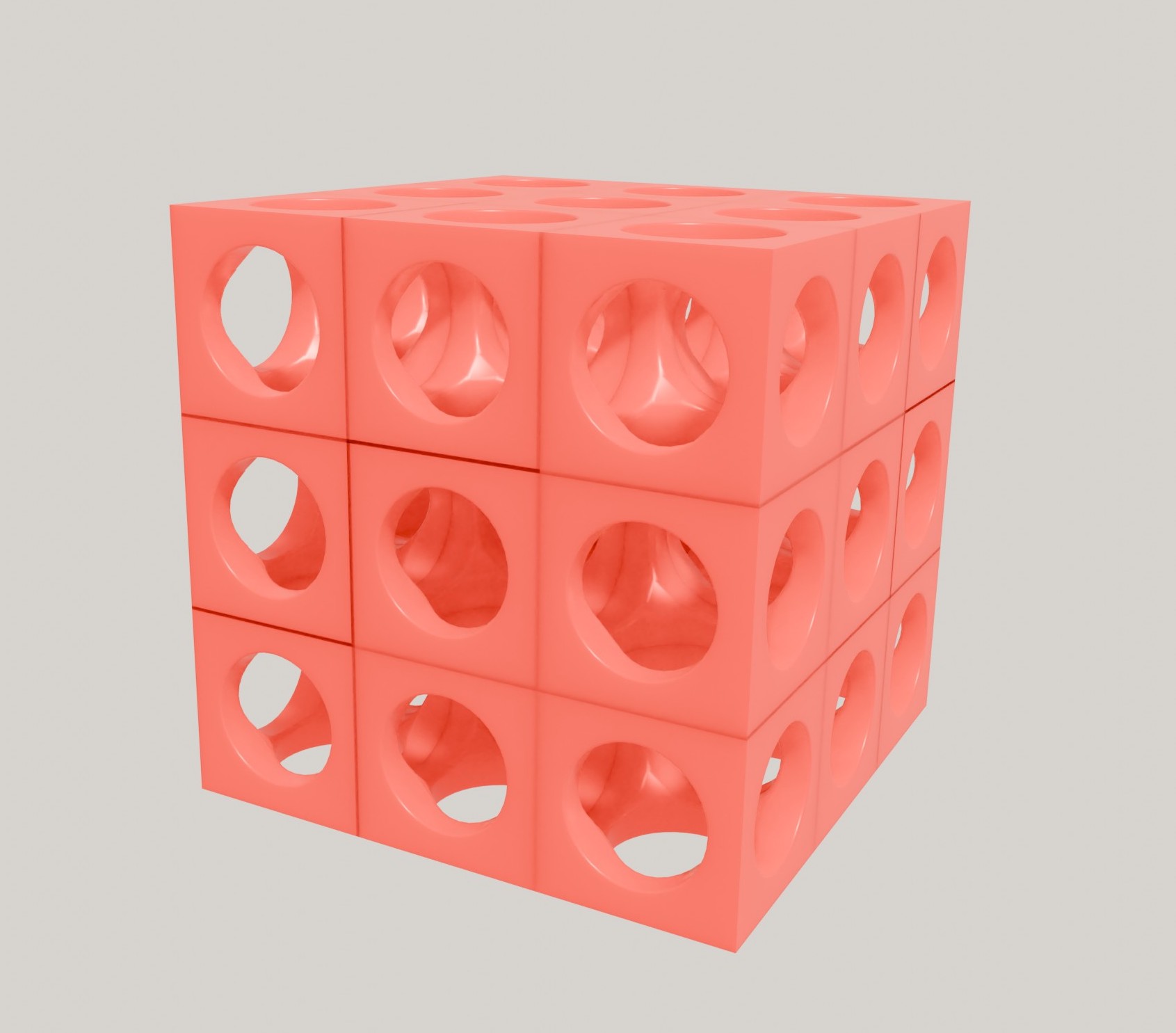}
\captionof{figure}{An example in three dimensions with $\Om, \Om^c$ connected}\label{figVan}
\end{center}

\begin{center}
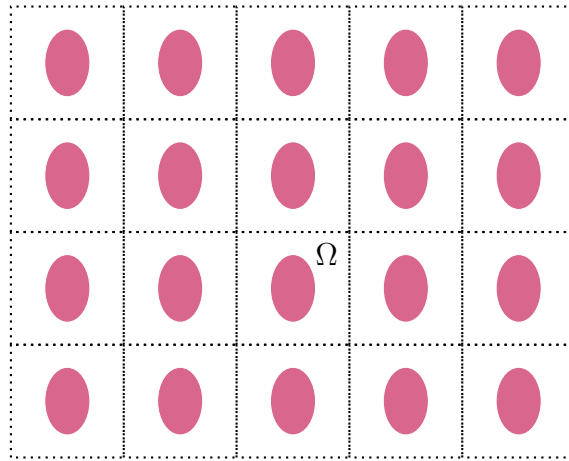

    \begin{tikzpicture}[scale=1.5]
    \foreach \x in {0.5,1.5,...,4.5}
        \foreach \y in {0.5,1.5,...,3.5}
            \filldraw[fill=purple!60,draw=white] (\x,\y) ellipse (0.2 cm and 0.3 cm);
     \foreach \x in {0.5,1.5,...,4.5}
        \foreach \y in {0.5,1.5,...,3.5}
            \draw[thick, dotted] (\x-0.5,\y-0.5) rectangle (\x+0.5,\y+0.5);
    \draw (2.8,1.8) node{$\Omega$};
\end{tikzpicture}
\captionof{figure}{An example of $\Om$ in two dimensions}\label{fig1}
\end{center}

\subsection{Relevant Literature} 
The qualitative homogenization theory of Hamilton--Jacobi equations on perforated domains was first studied in \cite{HorieIshii1998, Alvarez1999, AlvarezIshii}.
For the state-constraint cell problems, see \cite{CDL, HorieIshii1998, Mitake2008, IMT}.
See Figures \ref{figVan}--\ref{fig2} for examples of $\Om$.
Some properties of the effective Hamiltonian $\ol H$ were obtained in \cite{Con}.
However, the convergence rate of $u^\varepsilon$ to $u$ in this setting has not been studied.

Quantitative homogenization for first-order Hamilton--Jacobi equations in the
periodic setting in the whole domain has received much attention since 2000.
The convergence rate $O(\ep^{1/3})$ was obtained for first-order equations in \cite{CDI} for general coercive Hamiltonians.
The optimal rate of convergence $O(\ep)$ in the convex setting was obtained in \cite{TY2022}.
The method in \cite{TY2022} was generalized to handle problems with multiscales in \cite{HanJang} and a spatio-temporal environment in \cite{HNguyen}. 
For earlier progress with nearly optimal convergence rates in the convex setting, we refer the reader to \cite{MTY, JTY, Tu, Tran, TYWKAM} and the references therein.

\subsection{Settings}

Throughout this paper, we always assume the following conditions for the Hamiltonian $H: \R^n \times \mathbb{R}^n \to \mathbb{R}$ and the initial data $g:\R^n \to \R$.
\begin{itemize}
    \item[(A1)] $H\in C(\R^n\times \R^n)$; and for $p\in \R^n$, $y\mapsto H(y,p)$ is $\Z^n$-periodic.
    \item[(A2)] $\lim_{\left|p\right| \to \infty}\left(\inf_{y \in \R^n}H\left(y,p\right)\right) = +\infty$.
    \item[(A3)] For each $y\in \R^n$, the map $p \mapsto H(y, p)$ is convex.
    \item[(A4)] $g \in \mathrm{BUC}(\mathbb{R}^n) \cap \mathrm{Lip} (\mathbb{R}^n)$.
\end{itemize}
Let $\T^n=\R^n/\Z^n$ be the usual flat $n$ dimensional torus. 
Then, we can also write $H \in C(\T^n \times \R^n)$.
Note that (A1) and (A2) imply that there exists a constant $C_1>0$ so that
\begin{equation}\label{eq:H-lower-bd}
H(y, p) \geq -C_1 \qquad \text{ for all } (y, p) \in \R^n \times \mathbb{R}^n.
\end{equation}
Although \eqref{eqn:PDE_epsilon} and \eqref{eqn:cell} only require $H$ to be defined on $\ol{\Omega}\times \R^n$, it is more convenient for us to consider $H$ on $\R^n \times \R^n$ for later usages. 

The well-posedness of \eqref{eqn:PDE_epsilon} has been well studied in the literature (see \cite{Soner1, Soner2, CDL, Mitake2008}). Furthermore, one can show that the solution $u^\varepsilon$ is globally Lipschitz, that is, for $\ep\in (0,1)$,
\begin{equation}\label{eqn:u_prior}
\left\|u^\varepsilon_t\right\|_{L^\infty\left(\overline{\Omega} \times [0, \infty)\right)}+\left\|Du^\varepsilon\right\|_{L^\infty(\overline{\Omega} \times [0, \infty))} \leq C_0,
\end{equation}
where $C_0>0$ is a constant that depends only on $H$ and $\left\|Dg\right\|_{L^\infty(\mathbb{R}^n)}$.
Indeed, in light of \eqref{eq:H-lower-bd}, $g+C_1t$ is a supersolution to \eqref{eqn:PDE_epsilon}.
Besides, (A1) and (A4) yield that, for $C=C(H,\|Dg\|_{L^\infty(\R^n)})>0$ sufficiently large, $g-Ct$ is a subsolution to \eqref{eqn:PDE_epsilon}.
By the usual comparison principle,
\[
g(x)-Ct \leq u^\ep(x,t) \leq g(x)+C_1t \qquad \text{ for all } (x, t) \in \overline{\Omega} \times [0,\infty).
\]
Hence, $\|u^\ep_t(\cdot,0)\|_{L^\infty(\ol \Om)} \leq C$.
By using the comparison principle once more, we get $\|u^\ep_t\|_{L^\infty(\ol \Om\times[0,\infty))} \leq C$.
Finally, we use this bound, \eqref{eqn:PDE_epsilon}, and (A2) to obtain \eqref{eqn:u_prior}.

\smallskip

Based on \eqref{eqn:u_prior}, we can modify $H(y, p)$ for $|p| > 2C_0+1$ without changing the solutions to \eqref{eqn:PDE_epsilon} and ensure that, for all $ (y,p) \in \R^n\times \mathbb{R}^n$,
\begin{equation}\label{eqn:K_0H}
    \frac{|p|^2}{2}-K_0 \leq H(y, p) \leq \frac{|p|^2}{2}+K_0
\end{equation}
for some constant $K_0 >0$ that depends only on $H$ and $\left\|Dg\right\|_{L^\infty(\mathbb{R}^n)}$. Consequently, for all $(y,v)  \in \R^n \times \mathbb{R}^n$,
\begin{equation}\label{eqn:K_0L}
\frac{|v|^2}{2}-K_0 \leq L(y, v) \leq \frac{|v|^2}{2}+K_0,
\end{equation}
where $L: \R^n \times \mathbb{R}^n \to \mathbb{R
}$ is the Legendre transform of $H$.

Moreover, we have the optimal control formulas for $u^\varepsilon$, that is,
\begin{equation}\label{eqn:ocfue}
    \begin{aligned}
        &u^\varepsilon (x, t)\\
        =\ & \inf \left\lbrace \int_0^t L\left( \frac{\xi(s)}{\varepsilon} ,\dot{\xi}(s)\right) \,ds+g\left(\xi(0)\right): \xi\in \mathrm{AC}([0,t];\overline{\Omega}_\varepsilon), \xi(t) = x\right\rbrace\\
        =\ &\inf \left\lbrace \varepsilon \int_0^{\frac{t}{\varepsilon}} L\left(\gamma(s),\dot{\gamma}(s)\right) \,ds+g\left(\varepsilon \gamma(0)\right): \gamma\in \mathrm{AC}\left(\left[0,\frac{t}{\varepsilon}\right];\overline{\Omega} \right), \gamma\left(\frac{t}{\varepsilon}\right) = \frac{x}{\varepsilon}\right\rbrace.
    \end{aligned}
\end{equation}
For $u$, we have the Hopf-Lax formula
\begin{equation}
\label{eq:HLubar}
    \begin{aligned}
        u(x, t) = \inf \left\lbrace t\overline{L} \left(\frac{x-y}{t}\right) +g\left(y\right): y\in \R^n\right\rbrace.
    \end{aligned}
\end{equation}
 Here, $\mathrm{AC}(J,U)$ denotes the set of absolutely continuous curves $\xi:J \to U$ and $\overline{L}$ is the Legendre transform of $\overline{H}: \mathbb{R}^n \to \mathbb{R
}$. Note that the admissible paths in the optimal control formula for $u^\varepsilon$ are restricted on $\overline{\Omega}_\varepsilon$ or $\ol \Om$ after rescaling. 
Also, for all $p,v \in \R^n$,
\begin{equation}\label{eq:Hbar-Lbar}
    \frac{|p|^2}{2}-K_0 \leq \ol H(p) \leq \frac{|p|^2}{2}+K_0,\qquad
    \frac{|v|^2}{2}-K_0 \leq \ol L(v) \leq \frac{|v|^2}{2}+K_0.
\end{equation}
We always assume \eqref{eqn:K_0H}, \eqref{eqn:K_0L}, and \eqref{eq:Hbar-Lbar} in our analysis from now on.

\subsection{Main results}
Our first main result is concerned with the convergence rate of $u^\ep$, the viscosity solution to \eqref{eqn:PDE_epsilon}, to $u$, the viscosity solution to \eqref{eqn:PDE_limit}.
\begin{thm}\label{thm:main1}
    Assume {\rm (A1)--(A4)}. For $\varepsilon>0$, let $u^\varepsilon$ be the viscosity solution to \eqref{eqn:PDE_epsilon} and $u$ be the viscosity solution to \eqref{eqn:PDE_limit}, respectively. Then, there exists a constant $C=C\left(n, \partial\Omega, H, \left\|Dg\right\|_{L^\infty\left(\mathbb{R}^n\right)}\right)>0$ such that, for $\ep \in (0,1)$,
\begin{equation}
    \left\|u^\varepsilon- u\right\|_{L^\infty\left(\overline{\Omega}_\varepsilon \times [0, \infty)\right)} \leq C \varepsilon.
\end{equation}
\end{thm}

It is clear that the convergence rate $O(\ep)$ in the above theorem is optimal.
We refer the reader to \cite[Proposition 4.3]{MTY} for an explicit example confirming this optimality.
It is worth noting that we only require that $\Om$ is open, connected, $\Z^n$-periodic with $C^1$ boundary here.
In particular, $\Om^c$ can be connected (see Figure \ref{figVan}) or can contain unbounded components such as periodic small tubes when $n\geq 3$ (see Figure \ref{fig2} for one example.)
To prove Theorem \ref{thm:main1}, we develop further the method in \cite{TY2022} to the current setting of perforated domains where the admissible paths and the optimal paths (geodesics) are restricted on $\overline{\Omega}_\varepsilon$ or $\ol \Om$ after rescaling. 
Intuitively, one can think of $\ol \Om_\ep^c$ or $\ol \Om^c$ after rescaling as obstacles that the admissible paths cannot travel into.

\medskip
\begin{center}
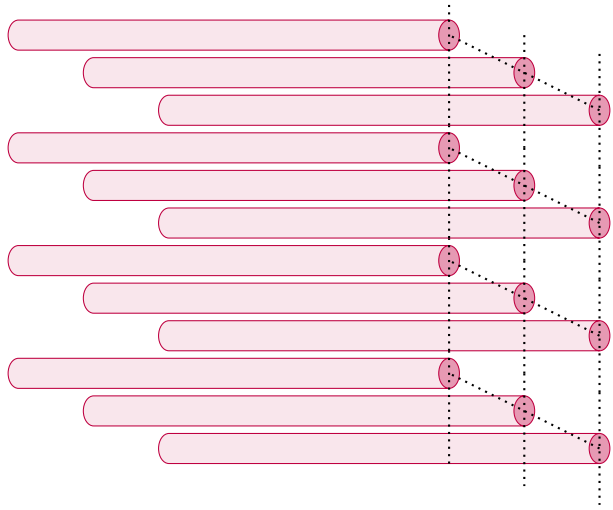

\begin{tikzpicture}[scale=1]
\begin{scope}
  \foreach \i in {0,1.5,3,4.5} {
    \node[cylinder, 
    draw = purple, 
    cylinder uses custom fill, 
    cylinder body fill = purple!10, 
    cylinder end fill = purple!40,
    minimum width = 0.4cm,
    minimum height = 6cm] (c) at (0,\i){};
    \node[cylinder, 
    draw = purple, 
    cylinder uses custom fill, 
    cylinder body fill = purple!10, 
    cylinder end fill = purple!40,
    minimum width = 0.4cm,
    minimum height = 6cm] (c) at (-1,\i+0.5){};
    \node[cylinder, 
    draw = purple, 
    cylinder uses custom fill, 
    cylinder body fill = purple!10, 
    cylinder end fill = purple!40,
    minimum width = 0.4cm,
    minimum height = 6cm] (c) at (-2,\i+1){};
    \draw[thick, dotted] (3,\i)--(1,\i+1);
    \draw[thick, dotted] (3,\i+0.75)--(3,\i-0.75);
    \draw[thick, dotted] (2,\i+1)--(2,\i-0.5);
    \draw[thick, dotted] (1,\i+1.5)--(1,\i-0.25);
    }
    \end{scope}
\end{tikzpicture}
\captionof{figure}{An example of $\Om$ in three dimensions}\label{fig2}
\end{center}

\medskip

Next, for each $\ep>0$, we consider the case that $\Omega_\ep$ consists of $\ep$-periodic copies of a hole whose diameter is $O(\eta(\ep)\ep)$ with $\lim_{\ep \to 0} \eta(\ep)=0$.
As the holes' diameter $O(\eta(\ep)\ep)$ vanishes faster than $\ep$ as $\ep \to 0$, it is natural to expect that we see the whole space in the limit.
This case is called a dilute situation in the homogenization theory, that is, the volume fraction of the obstacles vanishes in the limit. Since $\Om_\ep$ behaves like $\R^n$ as $\ep \to 0$, it is natural to make a connection with the usual homogenization problems in the whole space.
Let $\tilde u^\ep$ solve
\begin{equation}\label{eqn:whole_epsilon}
    \left\{\begin{aligned}
    \tilde u_t^\varepsilon+H\left(\frac{x}{\varepsilon}, D\tilde u^\varepsilon\right) & = 0 \qquad \quad \text{in } \R^n \times (0, \infty),\\
    \tilde u^\varepsilon(x,0) & =g(x) \,\,\,\, \quad \text{on } \R^n \times \{t=0\}.
    \end{aligned}
    \right.
\end{equation}
Let $\ol H_0=\ol H_{\R^n}$ be the effective Hamiltonian corresponding to $H$ in the whole space. For any $p\in \R^n$, $\ol H_0(p)$ is the unique constant such that the following usual cell problem has a $\Z^n$-periodic viscosity solution
\[
H(y,p+Dv_0(y))= \ol H_0(p) \qquad \text{ in } \R^n.
\]
It is important to note that $\ol H_0 \neq \ol H_\Om$ in general.
By the usual comparison principle, we have the one-sided inequality $\ol H_\Om \leq \ol H_0$.

Let $\tilde u$ solve
\begin{equation}\label{eqn:whole_limit}
    \left\{\begin{aligned}
    \tilde u_t+\overline{H}_0 \left(D\tilde u\right) & = 0 \qquad \quad \text{in } \mathbb{R}^n \times (0, \infty), \\
    \tilde u(x, 0)&=g(x) \, \, \,\quad \text{on } \mathbb{R}^n \times \left\{t=0\right\}.
    \end{aligned}
    \right.
\end{equation}

To get quantitative convergence rates, we need to put a further assumption on the Hamiltonian $H$.
\begin{itemize}
    \item[(A5)] For each $y\in \R^n$, $\min_{p\in \R^n} H(y,p)=H(y,0)=0$.
\end{itemize}

\begin{thm}\label{thm:main2}
    Assume {\rm (A1)--(A5)}. 
    Assume that there exists an open connected set $D\Subset \left(-\frac{1}{2},\frac{1}{2}\right)^n$ containing $0$ with connected $C^1$ boundary and $\eta:\left[0,\frac{1}{2}\right)\to \left[0,\frac{1}{2}\right)$ with $\lim_{\ep\to 0} \eta(\ep)=0$. Let 
    $\Omega^{\eta(\ep)} =\R^n\setminus \bigcup_{m\in \Z^n}(m+\eta(\ep) \ol D)$ and $\Omega_\ep = \ep \Omega^{\eta(\ep)}$ for $\ep \in \left(0,\frac{1}{2}\right)$.
    
    For $\varepsilon \in \left(0,\frac{1}{2}\right)$, let $u^\varepsilon$ and $\tilde u^\ep$ be the viscosity solutions to \eqref{eqn:PDE_epsilon} and \eqref{eqn:whole_epsilon}, respectively. 
    Let $\tilde u$ be the viscosity solution to \eqref{eqn:whole_limit}.
    Then, there exists a constant $C=C\left(n, \partial D, H, \left\|Dg\right\|_{L^\infty\left(\mathbb{R}^n\right)}\right)>0$ such that, for $\ep \in \left(0,\frac{1}{2}\right)$ and $(x,t)\in \ol \Om_\ep \times [0,\infty)$,
\begin{equation}\label{eqn:main2}
    \tilde u(x,t) - C \varepsilon \leq \tilde u^\ep(x,t) \leq u^\ep(x,t) \leq \tilde u^\ep(x,t) + C(\ep + \eta(\ep)t) \leq \tilde u(x,t) + C(\ep + \eta(\ep)t).
\end{equation}
In particular, for $\ep \in \left(0,\frac{1}{2}\right)$ and $(x,t)\in \ol \Om_\ep \times [0,\infty)$,
\[
|u^\ep(x,t)-\tilde u(x,t)| \leq C(\ep + \eta(\ep)t).
\]
\end{thm}

We note that the requirement on $\Omega_\ep$ in Theorem \ref{thm:main2} is stricter than that of Theorem \ref{thm:main1}. In particular, there cannot be any unbounded connected component in $\R^n\setminus {\ol \Omega_\ep}$. We think of $D$ as the model hole for the perforated domain. The restriction that $\partial D$ is connected can be lifted as long as $\Omega_\ep$ remains connected, for example, if $D$ consists of several separated parts that have connected boundary (or in other words, if the `hole' is actually a group of several separated holes). 
We say that $m+\eta(\ep)D$ is a hole of $\Om^{\eta(\ep)}$ and $\ep(m+\eta(\ep)D)$ is a hole of $\Om_\ep$ for $m\in \Z^n$, respectively.
For the Hamiltonian $H$, we need to require in addition that it satisfies (A5).
A prototypical example of $H$ in Theorem \ref{thm:main2} is $H(y,p)=a(y)|p|$ for $a\in C(\T^n,(0,\infty))$, which comes from the first-order front propagation framework.
To the best of our knowledge, this is the first time that the homogenization of Hamilton--Jacobi equations in a dilute perforated domain is studied, and the convergence rate obtained is essentially optimal (see Lemma \ref{lem:thm2-optimal}).

\begin{rem} \label{rem:1}
    As we are in the convex setting, we have the inf-sup representation formulas for the effective Hamiltonians (see \cite{Tran}): For $p\in \R^n$,
    \[
    \ol H_0(p) =\ol H_{\R^n}(p)= \inf_{\varphi \in \Lip(\T^n)} \esssup_{y\in \T^n} H(y,p+D\varphi(y)),
    \]
    and
    \begin{equation}\label{formula:effective-Hamiltonian}
    \ol H_\Om(p) = \inf_{\varphi \in \Lip(\T^n)} \esssup_{y\in \Om} H(y,p+D\varphi(y)).
    \end{equation}
    The inf-sup formulas confirm again that $\ol H_\Om \leq \ol H_0$.
    The formula \eqref{formula:effective-Hamiltonian} is well-known, but we give a proof in Lemma \ref{lem:inf-sup} in Appendix \ref{appendix} for the reader's convenience. 

    Let us now consider a sequence of nested domains $\{\Om_k\}$ such that $\Om_k \subset \Om_{k+1}$ for $k\in \N$ and 
    \[
    \bigcup_{k\in \N} \Om_k = \R^n \setminus \Z^n.
    \]
    Then, by the usual stability results for viscosity solutions, we have that, for $p\in \R^n$,
    \[
    \lim_{k\to \infty} \ol H_{\Om_k}(p) = \ol H_0(p).
    \]
    However, no convergence rate of $\ol H_{\Om_k}$ to $\ol H_0$ in the general setting is known in the literature.

    In the setting of Theorem \ref{thm:main2}, $\{\Omega^{\eta(\ep)}\}$ is also a sequence of nested domains with $\Omega^{\eta(\ep_1)} \subset \Omega^{\eta(\ep_2)}$ for $\ep_1>\ep_2>0$ and
     \[
    \bigcup_{\ep>0} \Om^{\eta(\ep)} = \R^n \setminus \Z^n.
    \]
    We have qualitatively that
    \[
    \lim_{\ep \to 0} \ol H_{\Om^{\eta(\ep)}} = \ol H_0.
    \]
    Let $\tilde u^{\eta(\ep)}$ be the unique viscosity solution to \eqref{eqn:PDE_limit} with $\ol H_{\Om^{\eta(\ep)}}$ in place of $\ol H=\ol H_\Om$.
    By repeating the proof of Theorem \ref{thm:main1}, we have readily that, for $\ep\in (0,1)$, 
    \[
    \|u^\ep - \tilde u^{\eta(\ep)}\|_{L^\infty(\Om_\ep\times [0,\infty))} \leq C\ep.
    \]
    Note that $C$ does not depend on $\eta(\ep)$.
    However, we cannot control $\tilde u^{\eta(\ep)} -\tilde u$ quantitatively yet because we do not have any quantitative bound on $\ol H_{\Om^{\eta(\ep)}} - \ol H_0$.
    This shows that the quantitative convergence result in Theorem \ref{thm:main2} is rather hard to obtain through the above roadmap.

\end{rem}

Once Theorem \ref{thm:main2} is obtained, then we can use it to deduce back a quantitative bound for $\ol H_{\Om^{\eta(\ep)}} - \ol H_0$ in Corollary \ref{cor:main2}.
For some related state-constraint problems in changing domains, we refer the reader to \cite{KTT, Tu2, TuZhang} and the references therein for quantitative convergence results.

\smallskip

Finally, we study a homogenization problem with domain defects.
Here, domain defects mean some holes in $\Om_\ep$ are missing.
Let us describe the setting.

Let $D\Subset (-\frac{1}{2},\frac{1}{2})^n$ be an open connected set  with $C^1$ boundary containing $0$.
Let $\Omega = \R^n\setminus \bigcup_{m\in \Z^n}(m+\ol D)$ and $\Omega_\ep = \ep \Omega$ for $\ep>0$.
We say that $m+D$ is a hole of $\Om$ and $\ep(m+D)$ is a hole of $\Om_\ep$ for $m\in \Z^n$, respectively.
Note that the holes' diameter of $\Om_\ep$ is $O(\ep)$, and we are in the usual setting, not the dilute one.
We remark that while $\Omega$ in Figure \ref{fig1} satisfies this requirement, the ones in Figures \ref{figVan} and \ref{fig2} do not.
Let $I \subsetneq \Z^n$ with $I \ne \emptyset$ be an index set that denotes the places where the holes are missing. Define
\[
W=\Omega \cup \bigcup_{m\in I} (m+ \ol D)= \R^n\setminus \bigcup_{m\in \Z^n\setminus I}(m+\ol D), \qquad W_\ep = \ep W.
\]
For $k\in \N$, let
\[
I_k = I \cap [-k,k]^n.
\]
We assume
\begin{equation}\label{eqn:assumpI}
    \frac{\left|I_k\right|}{k} = \omega_0\left(\frac{1}{k}\right) 
\end{equation}
where $\left|I_k\right|$ is the cardinality of the set $I_k$ and $\omega_0$ is a modulus of continuity. 
Note that $I_k$ collects all indices $m \in \mathbb{Z}^n \cap [-k, k]^n$ such that $Y_m: = m+[-\frac{1}{2},\frac{1}{2}]^n$ does not have a hole. 
For instance, if $\left| I_k\right|$ is of the size $ k^\theta$ for some $\theta \in (0, 1)$ as $k\to \infty$, then $\omega_0\left(s\right)=s^{1-\theta}$ for $s>0$. 

Interestingly, $W$ is not $\Z^n$-periodic anymore (see Figure \ref{figW}).
For $\varepsilon >0$,  let $w^\varepsilon$ be the unique viscosity solution to the state-constraint problem
\begin{equation}\label{eqn:w_epsilon}
    \left\{\begin{aligned}
    w_t^\varepsilon+H\left(\frac{x}{\varepsilon}, Dw^\varepsilon\right) & \leq 0 \qquad \quad \text{in } W_\varepsilon \times (0, \infty),\\
    w_t^\varepsilon+H\left(\frac{x}{\varepsilon}, Dw^\varepsilon\right) & \geq 0 \qquad \quad \text{on } \overline{W}_\varepsilon \times (0, \infty), \\
    w^\varepsilon(x,0) & =g(x) \,\,\,\, \quad \text{on } \overline{W}_\varepsilon \times \{t=0\}. 
    \end{aligned}
    \right.
\end{equation}
Note that the only difference between \eqref{eqn:w_epsilon} and \eqref{eqn:PDE_epsilon} is the domain of consideration $\Om_\ep \subsetneq W_\ep$. 
We now show that if we can control the size of $I$, which is allowed to be infinite, the domain defects do not affect the limiting behavior.
Specifically, we prove that $w^\ep$ behaves essentially like $u^\ep$, the solution of \eqref{eqn:PDE_epsilon}, as $\ep \to 0$.
Moreover, we also get a convergence rate of $w^\ep$ to $u$, the solution of \eqref{eqn:PDE_limit}.

\begin{center}
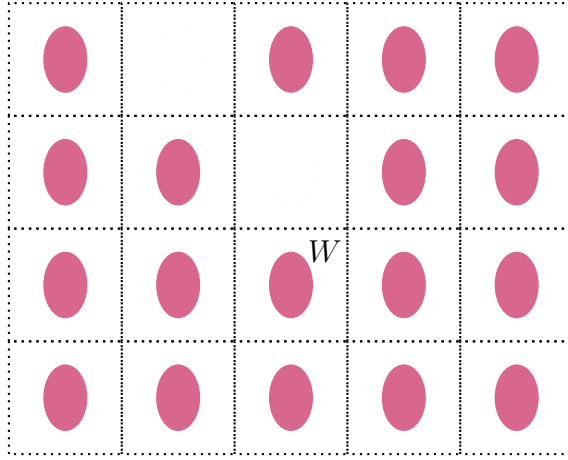

    \begin{tikzpicture}[scale=1.5]
    \foreach \x in {0.5,1.5,...,4.5}
        \foreach \y in {0.5,1.5,...,3.5}
            \filldraw[fill=purple!60,draw=white] (\x,\y) ellipse (0.2 cm and 0.3 cm);
     \foreach \x in {0.5,1.5,...,4.5}
        \foreach \y in {0.5,1.5,...,3.5}
            \draw[thick, dotted] (\x-0.5,\y-0.5) rectangle (\x+0.5,\y+0.5);
     \filldraw[fill=white,draw=white] (2.5,2.5) ellipse (0.2 cm and 0.3 cm);
      \filldraw[fill=white,draw=white] (1.5,3.5) ellipse (0.2 cm and 0.3 cm);
    \draw (2.8,1.8) node{$W$};
    
\end{tikzpicture}
\captionof{figure}{$W$ with some missing holes in two dimensions}\label{figW}
\end{center}

\begin{thm}\label{thm:main3}
    Assume {\rm (A1)--(A5)}, \eqref{eqn:assumpI} and the above setting.
    For $\ep>0$, let $u^\ep$ and $w^\ep$ be the unique viscosity solutions to \eqref{eqn:PDE_epsilon} and \eqref{eqn:w_epsilon}, respectively. Then, there exists a constant $C=C\left(n, \partial D, H, \|Dg\|_{L^\infty(\mathbb{R}^n)} \right) >0$ and $\tilde{C}=\tilde{C}\left(n, \partial D, H, \|Dg\|_{L^\infty(\mathbb{R}^n)} \right) >0$ such that
    \begin{equation}\label{eqn:main3}
    \begin{aligned}
    &u(x, t) - C\left(M_0t+|x| +1 \right) \omega_0 \left(\frac{\varepsilon}{ M_0t +|x|}\right) -\tilde{C} \varepsilon \\
    \leq \ & u^\varepsilon(x, t) -  C\left(M_0t+|x| +1 \right) \omega_0 \left(\frac{\varepsilon}{ M_0t +|x|}\right) - C\varepsilon \\
    \leq \ & w^\varepsilon(x, t) \leq u^\varepsilon(x, t) \leq u(x, t) +C \varepsilon.
    \end{aligned}
    \end{equation}
for all $\varepsilon \in (0, 1)$ and $\left(x, t\right) \in \overline W_\varepsilon \times [0, \infty)$. In particular, for $\varepsilon \in (0, 1)$ and $\left(x, t\right) \in \overline W_\varepsilon \times [0, \infty)$,
\begin{equation}\label{eq:rate-thm3}
\left|w^\varepsilon (x, t) - u(x, t)\right| \leq C\left(M_0t+|x| +1 \right) \omega_0 \left(\frac{\varepsilon}{ M_0t +|x|}\right) + C\varepsilon.
\end{equation}
\end{thm}

The optimality of the convergence rate in Theorem \ref{thm:main3} is confirmed by Lemma \ref{lem:W-op-2} with $\om_0(s)=s^{\frac{1}{2}}$ for $s>0$.
It is worth emphasizing that the control of the size of $I$ is optimal, that is, if we allow the size of $I$ to be bigger, then we will not observe the same limiting behavior (see Lemma \ref{lem:W1}).
One intriguing aspect of \eqref{eq:rate-thm3} is that it eliminates the necessity of pinpointing the exact locations of domain defects; instead, it relies solely on discerning the relative frequency of their occurrence.
It also does not hold if (A5) is not assumed (see Lemma \ref{lem:no-A5}).
To the best of our knowledge, this type of homogenization problem with domain defects has not been studied in the literature.

As $I$ can be infinite, it does not seem clear if the cell problem \eqref{eqn:cell} with $W$ in place of $\Om$ has sublinear solutions.
This is fine as we do not use the corrector approach in the proofs of Theorems \ref{thm:main1}--\ref{thm:main3}, and we use the optimal control formula together with the metric distance approach.
While the corrector approach works for general settings including the nonconvex cases, it does not give the optimal convergence rates.

We also note that the results of Theorems \ref{thm:main1}--\ref{thm:main3} hold if we only require $g\in \Lip(\R^n)$, a weaker assumption than (A4).
This is clear as the constant $C=C\left(n, \partial D, H, \left\|Dg\right\|_{L^\infty\left(\mathbb{R}^n\right)}\right)$ in the bounds of Theorems \ref{thm:main1}--\ref{thm:main3} depends only on $\left\|Dg\right\|_{L^\infty\left(\mathbb{R}^n\right)}$ but not $\left\|g\right\|_{L^\infty\left(\mathbb{R}^n\right)}$.

\subsection*{Notations} 
Let $\{e_1,e_2,\ldots,e_n\}$ be the canonical basis of $\R^n$.
We write $Y=[-\frac{1}{2},\frac12]^n$ 
as the unit cube in $\R^n$.
For $m\in \Z^n$, set $Y_m=m+[-\frac{1}{2},\frac{1}{2}]^n$, which is the cube of unit size centered at $m$.
For $\emptyset \neq U, V \subset \R^n$, denote by $\dist(U,V)=\inf_{x\in U, y\in V} |x-y|$.
For $x,y\in \R^n$, denote by $[x,y]$ the line segment connecting $x$ and $y$.
Let $\mathrm{AC}(J,U)$ be the set of absolutely continuous curves $\xi:J \to U$.
To avoid confusion, we recall that $\overline{H}=\ol H_\Om$ is the effective Hamiltonian corresponding to $H$ of the state-constraint problem on $\ol \Omega$, and $\ol H_0=\ol H_{\R^n}$ is the effective Hamiltonian corresponding to $H$ in the whole space.
If a function $h:\R^n\to \R$ is $\Z^n$-periodic, we can think of $h$ as a function from $\T^n$ to $\R$ and vice versa.

\subsection*{Organization of this paper}
In Section \ref{sec:prelim}, we give some preliminaries and extend the cost function (the metric distance) to the whole space.
We prove the subadditivity and superadditivity of the extended cost function in Section \ref{sec:metric}.
The proof of Theorem \ref{thm:main1} is given in Section \ref{sec:main1}.
Section \ref{sec:main2} is devoted to the study of homogenization in a dilute setting, which includes the proof of Theorem \ref{thm:main2}.
The problem of domain defects and the proof of Theorem \ref{thm:main3} are given in Section \ref{sec:main3}.
In Appendix \ref{appendix}, we give the proofs of some auxiliary results.

\section{Preliminaries and extension of the cost function to $\mathbb{R}^n$} \label{sec:prelim}

We introduce the following definition of a metric, which measures the cost to go from one point to another on $\overline{\Omega}$ in a given time.

\begin{defn}\label{def:m}
Let $x, y\in \overline{\Omega}$ and $t>0$. 
Define 
\begin{equation}
    m(t,x,y)= \inf \left\lbrace \int_0^{t} L\left(\gamma(s),\dot{\gamma}(s)\right) \,ds: \gamma\in \mathrm{AC}\left(\left[0,t\right];\overline{\Omega}\right), \gamma(0)=x, \gamma\left(t\right)=y \right\rbrace.
\end{equation}
\end{defn}

Figure \ref{fig3} gives an example of an admissible path $\gamma\in \mathrm{AC}\left(\left[0,t\right];\overline{\Omega}\right)$.
Note that $\gamma$ might touch and run on $\partial \Omega$ for some time.

\begin{center}
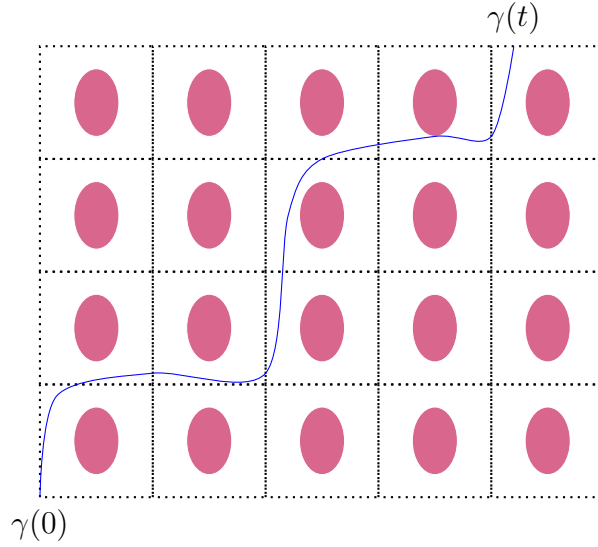

    \begin{tikzpicture}[scale=1.5]
    \foreach \x in {0.5,1.5,...,4.5}
        \foreach \y in {0.5,1.5,...,3.5}
            \filldraw[fill=purple!60,draw=white] (\x,\y) ellipse (0.2 cm and 0.3 cm);
     \foreach \x in {0.5,1.5,...,4.5}
        \foreach \y in {0.5,1.5,...,3.5}
            \draw[thick, dotted] (\x-0.5,\y-0.5) rectangle (\x+0.5,\y+0.5);
\draw [blue] plot [thick, smooth] coordinates { (0,0) (0.15,0.9) (1,1.1) (2,1.1) (2.2,2.5) (2.5,3) (3.5,3.2) (4,3.2) (4.2,4)};
\draw (0,-0.25) node{$\gam(0)$} (4.2,4.25) node{$\gam(t)$};

\end{tikzpicture}
\captionof{figure}{An admissible path}\label{fig3}
\end{center}

\medskip

The following lemma, the proof of which is given in Appendix \ref{appendix}, tells us that all the optimal paths for $u^\varepsilon$ and $u$ have a uniform velocity bound, which will help us simplify the optimal control formulas for $u^\varepsilon$ and $u$.

\begin{lem}\label{lem:velocity bound}
Assume {\rm(A1)--(A4)}. Let $\varepsilon, t>0$ and $x\in \mathbb{R}^n$. Suppose that $\gamma:\left[0, \frac{t}{\varepsilon}\right] \to \overline{\Omega}$ is a minimizing curve of $u^\varepsilon(x, t)$ in the sense that $\gamma$ is absolutely continuous, and
\begin{equation}
u^\varepsilon (x, t) = \varepsilon \int_0^{\frac{t}{\varepsilon}} L\left(\gamma(s),\dot{\gamma}(s)\right) \,ds+g\left(\varepsilon \gamma(0)\right)  
\end{equation}
with $\gamma\left(\frac{t}{\varepsilon}\right)=\frac{x}{\varepsilon}$. 
Then there exists a constant $M_0=M_0\left(n,\partial \Omega, H, \|Dg\|_{L^\infty(\mathbb{R}^n)}\right)>0$ such that 
\[
\left\|\dot{\gamma}\right\|_{L^\infty([0,\frac{t}{\varepsilon}])}  \leq M_0.
\]
\end{lem}

Using Definition \ref{def:m} and Lemma \ref{lem:velocity bound}, we can rewrite the optimal control formula for $u^\varepsilon$ as the following:

\begin{equation}
\begin{aligned}
    u^\varepsilon\left(x, t\right) &= \inf\left\{\varepsilon m\left(\frac{t}{\varepsilon}, \frac{y}{\varepsilon}, \frac{x}{\varepsilon}\right)+g\left(y\right): y \in \varepsilon \overline{\Omega}\right\}\\
    &= \inf\left\{\varepsilon m\left(\frac{t}{\varepsilon}, \frac{y}{\varepsilon}, \frac{x}{\varepsilon}\right)+g\left(y\right): \left|x-y\right|\leq M_0t, y \in \varepsilon \overline{\Omega}\right\}.
\end{aligned}
\end{equation}

Note that in Definition \ref{def:m}, $m$ is only defined for $x, y \in \overline{\Omega}$. We extend this metric to the whole $\mathbb{R}^n \times \R^n$ as follows.
\begin{defn}
    Let $x, y\in \mathbb{R}^n$ and $t>0$. Define 
\begin{equation}
     m^{\ast}\left(t,x,y\right)= \inf \left\lbrace m\left(t, \tilde{x}, \tilde{y}\right): \tilde{x}, \tilde{y} \in \partial \Omega, \tilde{x} - x \in Y, \tilde{y} - y \in Y \right\rbrace. 
\end{equation} 
\end{defn}
As we prove in Lemmas \ref{lem:subadd}--\ref{lem:superadd}, 
$m^\ast$ satisfies the subadditivity and superadditivity properties, 
which guarantee the limit of the scaling (large time average) $\frac{1}{k}m^\ast(kt,kx,ky)$ as $k\to\infty$. 
The subadditivity and superadditivity further allow us to quantify the convergence rate of this limit optimally in Theorem \ref{thm:mbarmstarrate}.
Moreover, as we prove in Proposition \ref{prp:mstarm} below, 
$m-m^\ast$ is bounded, which plays a key role in characterizing the limit of $u^\ep$ as $\ep\to0$ (see Lemma \ref{lem:ubaru}.)
Combining these results, we obtain Theorem \ref{thm:main1}. 


We now explore some properties of the metric $m^\ast$ and its connection with $m$ in the following two propositions. First, we provide the existence of a particular admissible path for $m^\ast$ with a velocity bound and get an upper bound for $m^\ast$. 
See \cite[Lemma 2.6]{HorieIshii1998} for a related result.

\begin{prop} \label{prop:mstarbound}
    Let $ t\geq \delta$ for some $\delta >0$, and $x, y \in \mathbb{R}^n$ with $|x-y|\leq M_0t$. Then, there exists an absolutely continuous curve $\xi:[0,t] \to \overline{\Omega}$ such that $\xi(0)=\tilde{x}$ and $\xi(t)=\tilde{y}$, for some $\tilde{x}, \tilde{y} \in \partial \Omega$ with $\tilde{x} - x \in Y, \tilde{y} - y \in Y$. Moreover, for some constant $C_b>0$ that only depends on $\partial \Omega$ and $n$, we have 
    \begin{equation}
    \label{eq:curvepO}
    \left\|\dot{\xi}\right\|_{L^\infty([0,t])} \leq C_b \left(M_0+\frac{2\sqrt{n}}{\delta}\right), \qquad \displaystyle m ^\ast \left(t, x, y\right) \leq \left( \frac{C_b^2 \left(M_0+\frac{2\sqrt{n}}{\delta}\right)^2}{2} +K_0 \right) t.
    \end{equation}
\end{prop}

\begin{proof}
Choose any $\tilde{x}, \tilde{y} \in \partial \Omega$ such that $\tilde{x} - x \in Y, \tilde{y} - y \in Y$. Consider the path $\xi(s): = \displaystyle \frac{s}{t}\left( \tilde{y}-\tilde{x}\right) + \tilde{x}$ for $s\in [0,t]$, which is a straight line segment connecting $\tilde{x}$ and $\tilde{y}$. Note that for any $s \in (0, t)$,
\begin{equation}
\begin{aligned} \label{eqn:bv}
  \left|\dot{\xi}(s) \right|= \frac{\left|\tilde{y}-\tilde{x}\right|}{t} &\leq \frac{\left|\tilde{y}-y\right|}{t} + \frac{\left|y-x\right|}{t} + \frac{|x-\tilde{x}|}{t} \\
  &\leq  M_0 + \frac{2\sqrt{n}}{t}
   \leq M_0 + \frac{2\sqrt{n}}{\delta}.
\end{aligned}
\end{equation}
If the whole segment $[\tilde x,\tilde y]$ is contained in $\ol\Omega$, $\xi$ is a curve satisfying the required velocity bound. Otherwise, $\xi([0,t])$ has nonempty intersections with $\ol\Omega^c$, and we modify $\xi$ to obtain a desired path. In the latter case, there are finitely many time segments $[t_1,t_2]\subset [0,t]$ such that $\xi(s) \in \overline{\Omega}^c$ for $s \in \left(t_1, t_2\right)$ and $\xi(t_1), \xi(t_2) \in \partial \Omega$.
\medskip

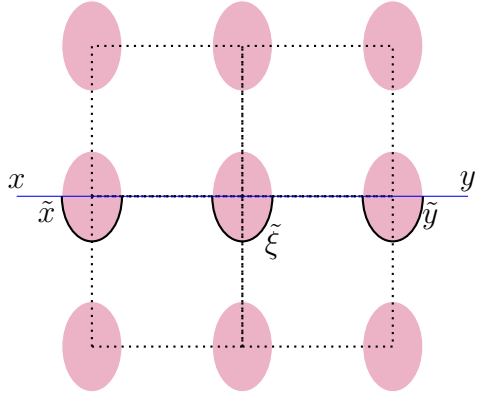
\begin{figure}
\begin{center}
    \begin{tikzpicture}
    \foreach \x in {1,3,5}
        \foreach \y in {1,3,5}
            \filldraw[fill=purple!30,draw=white] (\x,\y) ellipse (0.4 cm and 0.6 cm);
     \foreach \x in {2,4}
        \foreach \y in {2,4}
            \draw[thick, dotted] (\x-1,\y-1) rectangle (\x+1,\y+1);
    \draw[blue] (0,3)--(6,3);
    \draw[black, thick] (0.6,3) arc(180:360:0.4cm and 0.6cm);
    \draw[black, thick] (2.6,3) arc(180:360:0.4cm and 0.6cm);
    \draw[black, thick] (4.6,3) arc(180:360:0.4cm and 0.6cm);
    \draw[black, dotted, thick] (1.4,3)--(2.6,3);
    \draw[black, dotted, thick] (3.4,3)--(4.6,3);
    \draw (0,3.2) node{$x$} (0.4, 2.8) node{$\tilde x$} (6,3.2) node{$y$} (5.5,2.75) node{$\tilde y$} (3.4,2.4) node{$\tilde \xi$};
\end{tikzpicture}
\caption{A modified admissible path $\tilde \xi$}\label{fig4}
\end{center}
\end{figure}
\medskip

For each such segment, the end points $p= \xi(t_1)$ and $q=\xi(t_2)$ must belong to a same connected component $M$ of $\partial \Omega$, and by Lemma \ref{lem:pOmetric} in Appendix \ref{appendix}, there is a curve $\gamma \in C^1([0,1];M)$ joining $p$ to $q$ with arc length $\ell$, and, moreover, $\ell \le C_b|p-q|$. Let $\tilde \gamma: [0,\ell] \to M$ be the arc length (re-)parametrization of $\gamma$. We modify $\xi(s)$ for $s\in (t_1,t_2)$ to
$$
\tilde\xi(s) = \tilde\gamma\left(\frac{(s-t_1)\ell}{t_2-t_1}\right).
$$
Since $|\dot{\tilde\gamma}| =1$, we check that for all $s\in (t_1,t_2)$,
\begin{equation*}
\left|\dot{\tilde \xi}(s)\right| = \frac{\ell}{t_2-t_1} \le \frac{C_b|p-q|}{t_2-t_1} = C_b|\dot\xi(s)| \le C_b\left(M_0+\frac{2\sqrt{n}}{\delta}\right).
\end{equation*}
Modify each segment of $\xi([0,t])$ that is in $\ol\Omega^c$ in the above manner, and call the whole modified curve as $\tilde\xi$.
See Figure \ref{fig4} for one representative example of $\tilde \xi$.
It is a path on $\ol\Omega$ satisfying the desired velocity bound in \eqref{eq:curvepO} and hence an admissible path for the definition of $m^*(t,x,y)$. Finally, using the bound \eqref{eqn:K_0L}, we check
\[
\begin{aligned}
m^\ast\left(t, x, y\right) &\leq \int_0^t L \left( \tilde{\xi}(s), \dot{\tilde{\xi}}(s) \right) \,ds \leq \left(\frac12 \left\|\dot{\tilde{\xi}}\right\|_{L^\infty \left([0,t]\right)}^2 + K_0 \right)t\\
& \leq \left( \frac{C_b^2 \left(M_0+\frac{2\sqrt{n}}{\delta}\right)^2}{2} +K_0 \right) t.
\end{aligned}
\]
This is the desired bound for $m^*$ in \eqref{eq:curvepO} and completes the proof.
\end{proof}

\begin{prop}\label{prop:mstar}

Let $t, \tau>0$ with $t \geq 1$ or $\tau \geq 1$. Consider $x, y, z \in \mathbb{R}^n$ with $|x-y|\leq M_0t$ and $|y-z| \leq M_0\tau$. Then, there exists a constant $C=C\left(n, \partial \Omega, M_0, K_0\right)>0$ such that 
\[
m^\ast\left(t+\tau, x, z\right) \leq m^\ast \left(t, x, y\right)+ m^\ast \left(\tau, y, z\right)+C.
\]
    
\end{prop}
\begin{proof}

Without loss of generality, we can assume $t \geq 1$. Suppose $\xi: [0,t] \to \overline{\Omega}$ is an optimal path for $m^\ast(t, x, y)$, that is,
$\xi\left(0\right)=\tilde{x}, \xi\left(t\right) = \tilde {y}$,
for some $\tilde{x}, \tilde{y} \in \partial \Omega$ with $\tilde{x} - x \in Y, \tilde{y} - y \in Y$ and
\[
m^\ast\left(t, x, y\right) = \int_0^t L\left(\xi(s), \dot{\xi}(s)\right)\,ds.
\]
We claim that there exists $\displaystyle d\in \left\{0, \frac{1}{4}, \frac{1}{2}, \frac{3}{4}, \cdots, \lfloor t \rfloor -\frac{1}{4} \right\}$ such that 
\begin{equation}\label{eqn:intb}
    \int_d^{d+\frac{1}{4}} L\left(\xi(s), \dot{\xi} (s)\right) \,ds\leq \frac{C_b^2 \left(M_0+8\sqrt{n}\right)^2}{2} +K_0.
\end{equation} 
If not, 
\begin{equation}\label{eqn:cdt}
    \int_0^{\lfloor t \rfloor} L\left(\xi(s), \dot{\xi} (s)\right)\,ds \geq 4 \lfloor t \rfloor \left(\frac{C_b^2 \left(M_0+8\sqrt{n}\right)^2}{2}+K_0\right).
\end{equation}
Then, by Proposition \ref{prop:mstarbound},
\begin{equation}
\begin{aligned}
    \frac{C_b^2 \left(M_0+8\sqrt{n}\right)^2}{2} t+K_0 t &\geq \int_0^{\lfloor t \rfloor} L\left(\xi(s), \dot{\xi} (s)\right)\,ds + \int_{\lfloor t \rfloor}^t L\left(\xi(s), \dot{\xi} (s)\right)\,ds\\
    &\geq 4 \lfloor t \rfloor \left(\frac{C_b^2 \left(M_0+8\sqrt{n}\right)^2}{2}+K_0\right) -K_0,\\
\end{aligned}
\end{equation}
where the last inequality comes from \eqref{eqn:K_0L} and \eqref{eqn:cdt}.
On the other hand, since $t \geq 1$, 
\[
4 \lfloor t \rfloor \left(\frac{C_b^2 \left(M_0+8\sqrt{n}\right)^2}{2}+K_0\right) -K_0 > \frac{C_b^2 \left(M_0+8\sqrt{n}\right)^2}{2} t+K_0 t,
\]
which is a contradiction.

Let $\eta$ be an optimal path for $m^\ast \left(\tau, y, z\right)$, that is,
$\eta(0)=\tilde{y}^\prime, \, \eta(\tau) = \tilde{z}$ for some $\tilde{y}^\prime, \tilde{z} \in \partial \Omega$ with $\tilde{y}^\prime-y \in Y, \tilde{z}-z \in Y$, and
\[
m^\ast \left(\tau, y, z\right) = \int_0^\tau L\left(\eta(s), \dot{\eta}(s)\right)\,ds.
\]
Similar to the proof of Proposition \ref{prop:mstarbound}, we can find a path $\gamma:\left[0,1\right] \to \overline{\Omega}$ such that $\gamma(0)=\tilde{y}$ and $\gamma(1)=\tilde{y} ^\prime$ with
\begin{equation} \label{eqn:gammavb}
\left\|\dot{\gamma}\right\|_{L^\infty([0, 1])} \leq C_b\left(2\sqrt{n}+2\sqrt{n}\right)=4\sqrt{n}C_b,
\end{equation}
since $|\tilde{y}-\tilde{y}^\prime| \leq 2\sqrt{n}$.
Define $\zeta: [0,t+\tau] \to \overline{\Omega}$ by 
\begin{equation}
  \zeta (s) :=  \left\{\begin{aligned} &\xi(s), \qquad \qquad \qquad \qquad \qquad \quad \text{if } 0 \leq s \leq d,\\
  &\xi \left(2(s-d)+d\right), \, \, \, \, \, \quad \qquad \qquad \text{if } d \leq s \leq d+\frac{1}{8},\\
  &\xi\left(s+\frac{1}{8}\right),\qquad \qquad \qquad \qquad \text{if } d+\frac{1}{8}\leq s \leq t-\frac{1}{8},\\
  &\gamma\left(8\left(s-t+\frac{1}{8}\right)\right),\,\qquad \qquad \text{if } t-\frac{1}{8}\leq s \leq t,\\
&\eta(s-t),   \qquad \qquad \qquad \qquad \quad   \, \, \text{if } t \leq s \leq t+\tau,\\
    \end{aligned}
    \right.
\end{equation}
which is an admissible path for $m^\ast\left(t+\tau, x, z\right)$. Therefore, 
\[m^\ast\left(t+\tau, x, z\right) \leq \int_0^{t+\tau}  L\left(\zeta(s), \dot{\zeta}(s)\right)\,ds, \]
that is,
\begin{equation}\label{eqn:adp}
    \begin{aligned}
&m^\ast\left(t+\tau, x, z\right) \\\leq &\int_0^d L\left(\xi(s), \dot{\xi}(s)\right)\,ds+\int_d^{d+\frac{1}{8}} L\left(\xi \left(2(s-d)+d\right), 2\dot{\xi}\left(2(s-d)+d\right)\right)\,ds\\
    & + \int_{d+\frac{1}{8}}^{t-\frac{1}{8}} L\left(\xi\left(s+\frac{1}{8}\right), \dot{\xi}\left(s+\frac{1}{8}\right)\right)\,ds\\
    &+\int_{t-\frac{1}{8}}^{t} L\left(\gamma\left(8\left(s-t+\frac{1}{8}\right)\right), 8\dot{\gamma}\left(8\left(s-t+\frac{1}{8}\right)\right)\right)\,ds\\
    &+ \int_t^{t+\tau} L\left(\eta(s-t), \dot{\eta}(s-t)\right)\,ds.
\end{aligned}
\end{equation}
Note that
\begin{equation}\label{eqn:xi-d}
\begin{aligned}
    &\int_0^d L\left(\xi(s), \dot{\xi}(s)\right)\,ds + \int_{d+\frac{1}{8}}^{t-\frac{1}{8}} L\left(\xi\left(s+\frac{1}{8}\right), \dot{\xi}\left(s+\frac{1}{8}\right)\right)\,ds\\
    =& \int_0^d L\left(\xi(s), \dot{\xi}(s)\right)\,ds + \int_{d+\frac{1}{4}}^tL\left(\xi(s), \dot{\xi}(s)\right)\,ds\\
    =& \int_0^t L\left(\xi(s), \dot{\xi}(s)\right)\,ds - \int_d^{d+\frac{1}{4}}L \left(\xi(s), \dot{\xi}(s)\right)\,ds\\
    \leq & \, m^\ast \left(t, x, y\right) + \frac{K_0}{4},
\end{aligned}
\end{equation}
where the last inequality comes from \eqref{eqn:K_0L}. Besides, 
\begin{equation}\label{eqn:xi_d}
\begin{aligned}
    &\int_d^{d+\frac{1}{8}} L\left(\xi \left(2(s-d)+d\right), 2\dot{\xi}\left(2(s-d)+d\right)\right)\,ds \\ 
    \leq\ & \frac{1}{2}\int_d^{d+\frac{1}{4}} L\left(\xi(s), 2\dot{\xi}(s)\right)\,ds
     \leq \frac{1}{2} \int_d^{d+\frac{1}{4}} \frac{\left(2\left|\dot{\xi}(s)\right|\right)^2}{2} + K_0\,ds\\
    \leq\ &\int_d^{d+\frac{1}{4}} \left|\dot{\xi}(s)\right|^2\,ds+ \frac{K_0}{8} \,ds
     \leq  2 \int_d^{d+\frac{1}{4}} L\left(\xi(s), \dot{\xi} (s)\right) \,ds+\frac{17}{8} K_0 \,ds\\
     \leq \ & C_b^2 \left(M_0+8\sqrt{n}\right)^2 +\frac{33}{8}K_0,
\end{aligned}
\end{equation}
where the second and the fourth inequalities come from \eqref{eqn:K_0L} and the last inequality follows from \eqref{eqn:intb}.
Moreover,
\begin{equation}\label{eqn:connecteta}
    \begin{aligned}
        &\int_{t-\frac{1}{8}}^{t} L\left(\gamma\left(8\left(s-t+\frac{1}{8}\right)\right), 8\dot{\gamma}\left(8\left(s-t+\frac{1}{8}\right)\right)\right)\,ds \\
        &+ \int_t^{t+\tau} L\left(\eta(s-t), \dot{\eta}(s-t)\right)\,ds\\ 
        =\ & \frac{1}{8}\int_0^{1} L\left(\gamma(s), 8\dot{\gamma}(s)\right)\,ds + m^\ast \left(\tau, y, z\right)\\
        \leq \ &   4 \int_0^1 \left|\dot{\gamma}(s)\right|^2\,ds + \frac{K_0}{8}+ m^\ast \left(\tau, y, z\right)\\
        \leq\ & 64 n C_b^2 + \frac{K_0}{8}+m^\ast \left(\tau, y, z\right),
    \end{aligned}
\end{equation}
where the last inequality follows from \eqref{eqn:gammavb}.

\medskip

Combining \eqref{eqn:adp}, \eqref{eqn:xi-d}, \eqref{eqn:xi_d}, and \eqref{eqn:connecteta}, we have
\begin{equation}
    m^\ast\left(t+\tau, x, z\right) \leq m^\ast \left(t, x, y\right)+ m^\ast \left(\tau, y, z\right)+C_b^2\left(M_0+8 \sqrt{n}\right)^2+64 n C_b^2+\frac{9}{2}K_0.
\end{equation}

The proof for the case where $\tau \geq 1 $ is similar.
\end{proof}

For the points $x, y \in \overline{\Omega}$, the cost calculated by $m^\ast$ is fairly similar to the one computed by $m$, as demonstrated in the following proposition.

\begin{prop} \label{prp:mstarm}
Let $t \geq 1$ and $x, y \in \overline{\Omega}$ with $\left|x-y\right| \leq M_0t$. Then, there is a constant $C>0$ depending only on $n, \partial \Omega, M_0$ and $K_0$ such that
\begin{equation}
     |m^\ast(t, x, y)-m(t, x, y)|<C. 
\end{equation} 
\end{prop}

\begin{proof}
We proceed to prove
\[
m^\ast(t, x, y) \leq m(t, x, y) +C.
\]
The proof for the other direction follows similarly.

\smallskip

    Let $\xi: [0, t] \to \overline{\Omega}$ be an optimal path for $m(t, x, y)$, that is, $\xi(0)=x$, $\xi(t)=y$, and
    \[
    m(t, x, y)= \int_0^t L\left(\xi(s), \dot{\xi}(s)\right)\,ds.
    \]
    Choose $\tilde{x}, \tilde{y} \in \overline{\Omega}$ such that $\tilde{x} - x \in Y$ and  $\tilde{y} - y \in Y$. Consider the path $\alpha(s): = \displaystyle \frac{s}{t}\left( y-x\right) + x$, which is a straight line segment connecting $x$ and $y$. Therefore, $\left\|\dot{\alpha} \right\|_{L^\infty\left([0,t]\right)} \leq M_0$. We can revise $\alpha$ into a new path $\tilde{\alpha}:[0,t] \to \overline{\Omega}$ which is restricted in $\overline{\Omega}$ with $\left\|\dot{\tilde{\alpha}} \right\|_{L^\infty\left([0,t]\right)} \leq C_b M_0$, similar to the argument in the proof of Proposition \ref{prop:mstar}. Therefore,
    \[
    m(t, x, y) = \int_0^t L\left(\xi(s), \dot{\xi}(s)\right)\,ds \leq \left(\frac{C_b^2M_0^2}{2} + K_0\right)t.
    \]
    Furthermore, we can also prove that there exists $\displaystyle d\in \left\{0, \frac{1}{4}, \frac{1}{2}, \frac{3}{4}, \cdots, \lfloor t \rfloor -\frac{1}{4} \right\}$ such that 
\begin{equation}
    \int_d^{d+\frac{1}{4}} L\left(\xi(s), \dot{\xi} (s)\right) \,ds\leq \frac{C_b^2 M_0^2}{2} +K_0,
\end{equation} 
by a similar argument as in the proof of Proposition \ref{prop:mstar}. Moreover, by the same argument, we can find a path $\gamma:[0,1] \to \overline{\Omega}$ connecting $\tilde{x}$ to $x$, and a path $\eta:[0,1] \to \overline{\Omega}$ connecting $\tilde{y}$ to $y$ with $\left\|\dot{\gamma}\right\|_{L^\infty\left([0,1]\right)}, \left\|\dot{\eta}\right\|_{L^\infty\left([0,1]\right)} \leq 3\sqrt{n} C_b$. Define $\zeta: [0,t] \to \overline{\Omega}$ by 
\begin{equation}
  \zeta (s) :=  \left\{\begin{aligned} &\gamma(16 s), \, \, \qquad \qquad \qquad \qquad \quad \text{if } 0 \leq s \leq \frac{1}{16},\\
  &\xi\left(s-\frac{1}{16}\right), \, \, \, \quad \qquad \quad \qquad \quad \text{if } \frac{1}{16} \leq s \leq d + \frac{1}{16},\\
  &\xi \left(2\left(s-d-\frac{1}{16}\right)+d \right), \, \, \, \, \, \quad \text{if } d + \frac{1}{16}\leq s \leq d+\frac{3}{16},\\
  &\xi\left(s+\frac{1}{16}\right),\, \, \,  \qquad \quad \qquad \qquad \text{if } d+\frac{3}{16}\leq s \leq t-\frac{1}{16},\\
&\eta\left(16\left(s-t+\frac{1}{16}\right)\right),   \qquad \quad   \, \, \text{if } t-\frac{1}{16} \leq s \leq t,\\
    \end{aligned}
    \right.
\end{equation}
which is an admissible path for $m^\ast(t, x, y)$. Therefore,

\begin{equation}
    \begin{aligned}
     &m^\ast(t, x, y)\\
     \leq\ & \int_0^\frac{1}{16} L \left(\gamma(16s), 16\dot{\gamma}(16s)\right)\,ds + \int_\frac{1}{16}^{d+\frac{1}{16}} L \left(\tilde{\xi}\left(s-\frac{1}{16}\right), \dot{\tilde{\xi}}\left(s-\frac{1}{16}\right)\right) \,ds\\ & +\int_{d+\frac{1}{16}}^{d+\frac{3}{16}} L \left(\tilde{\xi} \left(2\left(s-d-\frac{1}{16}\right)+d \right), 2\dot{\tilde{\xi}} \left(2\left(s-d-\frac{1}{16}\right)+d \right)\right)\,ds\\
     &+\int_{d+\frac{3}{16}}^{t-\frac{1}{16}} L \left(\tilde{\xi}\left(s+\frac{1}{16}\right), \dot{\tilde{\xi}}\left(s+\frac{1}{16}\right)\right) \,ds\\ &+ \int_{t-\frac{1}{16}}^t L\left(\eta\left(16\left(s-t+\frac{1}{16}\right)\right),16 \dot{\eta}\left(16\left(s-t+\frac{1}{16}\right)\right)\right)\,ds.
    \end{aligned}.
\end{equation}
Similar to \eqref{eqn:xi-d}, \eqref{eqn:xi_d}, and \eqref{eqn:connecteta} in the proof of Proposition \ref{prop:mstar}, we have

\begin{equation}
\begin{aligned}
  &\int_\frac{1}{16}^{d+\frac{1}{16}} L \left(\tilde{\xi}\left(s-\frac{1}{16}\right), \dot{\tilde{\xi}}\left(s-\frac{1}{16}\right)\right) \,ds \\
  &+\int_{d+\frac{3}{16}}^{t-\frac{1}{16}} L \left(\tilde{\xi}\left(s+\frac{1}{16}\right), \dot{\tilde{\xi}}\left(s+\frac{1}{16}\right)\right) \,ds \\
  \leq \ & m(t, x, y) - \int_d^{d+\frac{1}{4}}L\left(\tilde{\xi}(s), \dot{\xi}(s
  )\right)\,ds
  \leq  m(t, x, y) +\frac{K_0}{4},
\end{aligned}
\end{equation}
and
\begin{equation}
\begin{aligned}
    &\int_{d+\frac{1}{16}}^{d+\frac{3}{16}} L \left(\tilde{\xi} \left(2\left(s-d-\frac{1}{16}\right)+d \right), 2\dot{\tilde{\xi}} \left(2\left(s-d-\frac{1}{16}\right)+d \right)\right)\,ds \\
    \leq \ & C_b^2M_0^2+\frac{21}{8}K_0,
\end{aligned}
\end{equation}
and
\begin{equation}
\begin{aligned}
    &\int_0^\frac{1}{16} L \left(\gamma(16s), 16\dot{\gamma}(16s)\right)\,ds\\
    &+ \int_{t-\frac{1}{16}}^t L\left(\eta\left(16\left(s-t+\frac{1}{16}\right)\right),16 \dot{\eta}\left(16\left(s-t+\frac{1}{16}\right)\right)\right)\,ds\\ 
    \leq \ & 144 nC_b^2+\frac{K_0}{8},    
\end{aligned}
\end{equation}

which gives us
\[
m^\ast(t, x, y) \leq m(t, x, y) + 144 nC_b^2 + C_b^2M_0^2 + 3K_0.
\]
\end{proof}

Next, we show that if $H$ satisfies (A5), then $L$ also satisfies a dual version of (A5).
This property of $L$ is important for our analysis in the proofs of Theorems \ref{thm:main2}--\ref{thm:main3}.

\begin{lem}\label{lem:A5}
    Assume {\rm (A3), (A5)}.
    Then, for $y\in \R^n$,
    \[
    \min_{v\in \R^n} L(y,v)=L(y,0)=0.
    \]
\end{lem}
\begin{proof}
    Fix $y\in \R^n$.
   We compute that
   \[
   L(y,0)=\sup_{p\in \R^n}(-H(y,p))= -\inf_{p\in \R^n}H(y,p)=0.
   \]
   Besides, for any $v\in \R^n$,
   \[
   L(y,v)=\sup_{p\in \R^n}(p\cdot v - H(y,p)) \geq -H(y,0)=0.
   \]
   The proof is complete.
\end{proof}

\section{Superadditivity and subadditivity of the extended cost function} \label{sec:metric}

We first prove the subadditivity of the metric $m^\ast$.

\begin{lem} \label{lem:subadd}
 For $t \geq 1$, $\left|y\right| \leq M_0t$, we have
 \[
 m^\ast\left(2t, 0, 2y\right) \leq  2 m^\ast \left(t, 0, y\right) +C.
 \]
\end{lem}

\begin{proof}
By Proposition \ref{prop:mstar}, we know
\[
m^\ast\left(2t, 0, 2y\right) \leq m^\ast \left(t, 0, y\right) + m^\ast \left(t, y, 2y \right)+C
\]
for some constant $C=C\left(n, \partial \Omega, M_0, K_0\right)$. It suffices to prove that
\[
m^\ast\left(t, y, 2y\right) \leq m^\ast \left(t, 0, y\right) +C,
\]
for some constant $C=C(n, \partial \Omega, M_0, K_0)$. Suppose $\xi:[0, t] \to \overline{\Omega}$ is an optimal path for $m^\ast (t, 0, y)$, that is, $\xi(0)=\tilde{x}$, $\xi(t)=\tilde{y}$ for some $\tilde{x}, \tilde{y} \in \partial \Omega$ with $\tilde{x}-0\in Y$ and $\tilde{y}-y \in Y$. Similarly to the proof of Proposition \ref{prop:mstar}, we know that there exists $\displaystyle d\in \left\{0, \frac{1}{4}, \frac{1}{2}, \frac{3}{4}, \cdots, \lfloor t \rfloor -\frac{1}{4} \right\}$ such that 
\begin{equation}
    \int_d^{d+\frac{1}{4}} L\left(\xi(s), \dot{\xi} (s)\right) \,ds\leq \frac{C_b^2 \left(M_0+8\sqrt{n}\right)^2}{2} +K_0.
\end{equation} 
Choose $k \in \mathbb{Z}^n$ such that $\tilde{x}+k \in y+Y$ and define $\tilde{\xi}(s): = \xi (s)+ k$ for $s \in [0,t]$. Then,
\[
\int_d^{d+\frac{1}{4}} L\left(\tilde{\xi}(s), \dot{\tilde{\xi}} (s)\right) \,ds=\int_d^{d+\frac{1}{4}} L\left(\xi(s), \dot{\xi} (s)\right) \,ds\leq \frac{C_b^2 \left(M_0+8\sqrt{n}\right)^2}{2} +K_0.
\]
Choose $\tilde{z} \in \partial \Omega$ such that $\tilde{z}-2y \in Y$. From Proposition \ref{prop:mstarbound}, we can find a path $\gamma:\left[0,1\right] \to \overline{\Omega}$ such that $\gamma(0)=\tilde{y}+k$ and $\gamma(1)=\tilde{z}$ with

\begin{equation}
\left\|\dot{\gamma}\right\|_{L^\infty([0, 1])} \leq C_b\left(4\sqrt{n}+2\sqrt{n}\right)=6 \sqrt{n} C_b,
\end{equation}
since $|\tilde{y}+k-\tilde{z}| \leq |\tilde{y}+k-2y| + |\tilde{z} -2y| \leq 3\sqrt{n} + \sqrt{n}=4\sqrt{n}$.

\smallskip

Define $\zeta:[0, t] \to \overline{\Omega}$ by 
\begin{equation}
  \zeta (s) :=  \left\{\begin{aligned} &\tilde{\xi}(s), \qquad \qquad \qquad \qquad \qquad \quad \text{if } 0 \leq s \leq d,\\
  &\tilde{\xi} \left(2(s-d)+d\right), \, \, \, \, \, \quad \qquad \qquad \text{if } d \leq s \leq d+\frac{1}{8},\\
  &\tilde{\xi}\left(s+\frac{1}{8}\right),\qquad \qquad \qquad \qquad \text{if } d+\frac{1}{8}\leq s \leq t-\frac{1}{8},\\
  &\gamma\left(8\left(s-t+\frac{1}{8}\right)\right),\,\qquad \qquad \text{if } t-\frac{1}{8}\leq s \leq t,\\
    \end{aligned}
    \right.
\end{equation}
which is an admissible path for $m^\ast\left(t, y, 2y\right)$. Then, $m^\ast\left(t, y, 2y\right) \leq \int_0^t L\left(\zeta(s), \dot{\zeta}(s)\right)\,ds$, that is,
\begin{equation}
    \begin{aligned}
&m^\ast\left(t, y, 2y\right) \\
\leq &\int_0^d L\left(\tilde{\xi}(s), \dot{\tilde{\xi}}(s)\right)\,ds+\int_d^{d+\frac{1}{8}} L\left(\tilde{\xi} \left(2(s-d)+d\right), 2\dot{\tilde{\xi}}\left(2(s-d)+d\right)\right)\,ds\\
    & + \int_{d+\frac{1}{8}}^{t-\frac{1}{8}} L\left(\tilde{\xi}\left(s+\frac{1}{8}\right), \dot{\tilde{\xi}}\left(s+\frac{1}{8}\right)\right)\,ds\\
    &+\int_{t-\frac{1}{8}}^{t} L\left(\gamma\left(8\left(s-t+\frac{1}{8}\right)\right), 8\dot{\gamma}\left(8\left(s-t+\frac{1}{8}\right)\right)\right)\,ds\\
    = &\int_0^d L\left(\xi(s), \dot{\xi}(s)\right)\,ds+\int_d^{d+\frac{1}{8}} L\left(\xi \left(2(s-d)+d\right), 2\dot{\xi}\left(2(s-d)+d\right)\right)\,ds\\
    & + \int_{d+\frac{1}{8}}^{t-\frac{1}{8}} L\left(\xi\left(s+\frac{1}{8}\right), \dot{\xi}\left(s+\frac{1}{8}\right)\right)\,ds\\
    &+\int_{t-\frac{1}{8}}^{t} L\left(\gamma\left(8\left(s-t+\frac{1}{8}\right)\right), 8\dot{\gamma}\left(8\left(s-t+\frac{1}{8}\right)\right)\right)\,ds.
\end{aligned}
\end{equation}
Similar to \eqref{eqn:xi-d}, \eqref{eqn:xi_d}, and \eqref{eqn:connecteta} in the proof of Proposition \ref{prop:mstar}, we have
\begin{equation}
\int_0^d L\left(\xi(s), \dot{\xi}(s)\right)\,ds + \int_{d+\frac{1}{8}}^{t-\frac{1}{8}} L\left(\xi\left(s+\frac{1}{8}\right), \dot{\xi}\left(s+\frac{1}{8}\right)\right)\,ds
    \leq  \, m^\ast \left(t, 0, y\right) + \frac{K_0}{4},
\end{equation}

\begin{equation}
    \int_d^{d+\frac{1}{8}} L\left(\xi \left(2(s-d)+d\right), 2\dot{\xi}\left(2(s-d)+d\right)\right)\,ds \leq  C_b^2 \left(M_0+8\sqrt{n}\right)^2 +\frac{33}{8}K_0,
\end{equation}
and
\begin{equation}
    \begin{aligned}
        &\int_{t-\frac{1}{8}}^{t} L\left(\gamma\left(8\left(s-t+\frac{1}{8}\right)\right), 8\dot{\gamma}\left(8\left(s-t+\frac{1}{8}\right)\right)\right)\,ds \\ 
        =\ &\frac{1}{8}\int_0^{1} L\left(\gamma(s), 8\dot{\gamma}(s)\right)\,ds \\
        \leq \ &  4 \int_0^1 \left|\dot{\gamma}(s)\right|^2\,ds + \frac{K_0}{8}
        \leq  144 n C_b^2 + \frac{K_0}{8},
    \end{aligned}
\end{equation}
which gives
\[
m^\ast(t, y, 2y) \leq m^\ast(t, 0, y)+ 144 n C_b^2 +C_b^2 \left(M_0+8\sqrt{n}\right)^2 +\frac{9}{2}K_0.
\]
\end{proof}
Next, we show the superadditivity of the metric $m^\ast$.
\begin{lem}\label{lem:superadd}
    For $t \geq 1$, $|y| \leq M_0t$, we have
\begin{equation}\label{eqn:superadd}
    2m^\ast\left(t, 0, y\right) \leq m^\ast \left(2t, 0, 2y\right) +C.
\end{equation}
\end{lem}

\begin{proof}
    Let $\xi:[0,2t] \to \overline{\Omega} 
    $ be an optimal path of $m^\ast\left(2t, 0, 2y\right)$, that is, $\xi(0)=\tilde{x}$, $\xi(2t)=\tilde{z}$ for some $\tilde{x}, \tilde{z} \in \partial \Omega$ with $\tilde{x}-0\in Y$ and $\tilde{z}-2y \in Y$. Define $\gamma (s): = \left(\xi(s), s \right)$ for $0 \leq s \leq 2t$. By Burago's lemma (\cite{Burago}, see also \cite{TY2022}), there exists a collection of disjoint time intervals $\left\{[a_i, b_i] \right\}_{1 \leq i \leq k} \subset [0, 2t]$ with $k\leq \frac{n}{2}+1$ such that
\[
\sum_{i=1}^k \left(\gamma\left(b_i\right)-\gamma\left(a_i\right)\right)=\frac{\gamma(2t)-\gamma(0)}{2}=\left(\frac{\tilde{z}-\tilde{x}}{2}, t\right).
\]
Shift $\xi$ on $\left\{[a_i, b_i]\right\}_{i=1}^k$ in a periodic way to define a new path $\tilde{\xi}: [0 , t] \to \overline{\Omega}$ so that
\begin{enumerate}
    \item $t_0:=0$, $t_j:= \sum_{i=1}^j \left(b_i-a_i\right)$ for $1 \leq j \leq k$;
    \item $\tilde{\xi}(0^+) \in Y \cap \overline{\Omega}$;

    \item $\tilde{\xi} \big|_{\left(t_{j-1}, t_j\right)}$ is a periodic shift of $\xi \big|_{\left(t_{j-1}, t_j\right)}$ for $1 \leq j \leq k$;

    \item Define $\tilde{\xi}\left(0^-\right):=0$ and $\tilde{\xi}\left(t_k^+\right):=\frac{\tilde{z}-\tilde{x}}{2}$. Choose $\tilde{y} \in \partial \Omega$ such that $\tilde{y}-y \in Y$. Define $\tilde{\xi}(t_{k+1}) : = \tilde{y}$. Note that 
\begin{equation}\label{eqn:k+1kbound}
    \left|\tilde{\xi}(t_{k+1})-\tilde{\xi}\left(t_k^+\right)\right| \leq 2 \sqrt{n}. 
\end{equation}

    \item  Shift $\xi \big|_{\left(t_{j-1}, t_j\right)}$ so that $1 \leq j \leq k-1$, $\tilde{\xi}\left(t_j^+\right) - \tilde{\xi} \left(t_j^-\right) \in Y$, which gives that
\end{enumerate}
\begin{equation}\label{eqn:1k-1bound}
    \left|\tilde{\xi}\left(t_j^+\right) - \tilde{\xi} \left(t_j^-\right)\right| \leq \sqrt{n}.   
\end{equation}
Note that 
\begin{equation}\label{eqn:k+-bound}
    \left|\tilde{\xi}\left(t_k^+\right) - \tilde{\xi}\left(t_k^-\right)\right|\leq k\sqrt{n}.
\end{equation}
A key point here is $\tilde{\xi}([0,t]) \subset \overline{\Omega}$ as $\overline{\Omega}$ is $\Z^N$-periodic.
In other words, periodic shifts of pieces of $\xi$ still stay in $\ol \Omega$ and hence are admissible.
Moreover, we have 
\[
\sum_{j=1}^k\left(\tilde{\xi}\left(t_j^-\right) - \tilde{\xi} \left(t_j^+\right)\right)=\frac{\tilde{z}-\tilde{x}}{2}.
\]
\begin{figure}
  \centering
  \subfloat{
        \begin{tikzpicture}[scale=1.35]
        \foreach \x in {0.5,1.5,...,3.5}
        \foreach \y in {0.5,1.5,...,2.5}
            \filldraw[fill=purple!60,draw=white] (\x,\y) ellipse (0.2 cm and 0.3 cm);
        \foreach \x in {0.5,1.5,...,3.5}
        \foreach \y in {0.5,1.5,...,2.5}
            \draw[thick, dotted] (\x-0.5,\y-0.5) rectangle (\x+0.5,\y+0.5);
        \draw [blue, ultra thick] plot [smooth] coordinates { (0,0) (0.7,0.1) (1,1) (1,1.5)};
        \draw [blue, ultra thick] (2,3)--(3,3);
        \draw plot[smooth] coordinates {(1,1.5) (1.1,2) (0.9,2.5) (2,3)};
        \draw plot[smooth] coordinates {(3,3) (3.1, 2.1) (3.9, 2.2) (4,3)};
        \draw (0,-0.25) node{$\xi(0)$} (4.2,3.25) node{$\xi(2t)$};
        \end{tikzpicture}
  }
  \subfloat{
        \begin{tikzpicture}[scale=1.35]
        \foreach \x in {0.5,1.5,...,3.5}
        \foreach \y in {0.5,1.5,...,2.5}
            \filldraw[fill=purple!60,draw=white] (\x,\y) ellipse (0.2 cm and 0.3 cm);
        \foreach \x in {0.5,1.5,...,3.5}
        \foreach \y in {0.5,1.5,...,2.5}
            \draw[thick, dotted] (\x-0.5,\y-0.5) rectangle (\x+0.5,\y+0.5);
        \draw [blue, ultra thick] plot [smooth] coordinates { (0,0) (0.7,0.1) (1,1) (1,1.5)};
        \draw [blue, ultra thick] (1,2)--(2,2);
        \draw [->, red, thick] (1,1.5)--(1,2);
        \draw [->, red, thick] (2,2)--(2,1.5);
        \draw (0,-0.25) node{$\zeta(0)$} (2.25,1.1) node{$\zeta(t)$};
        \end{tikzpicture}
  }
\captionof{figure}{An example of $\xi$ and $\zeta$}\label{fig:xi}
\end{figure}
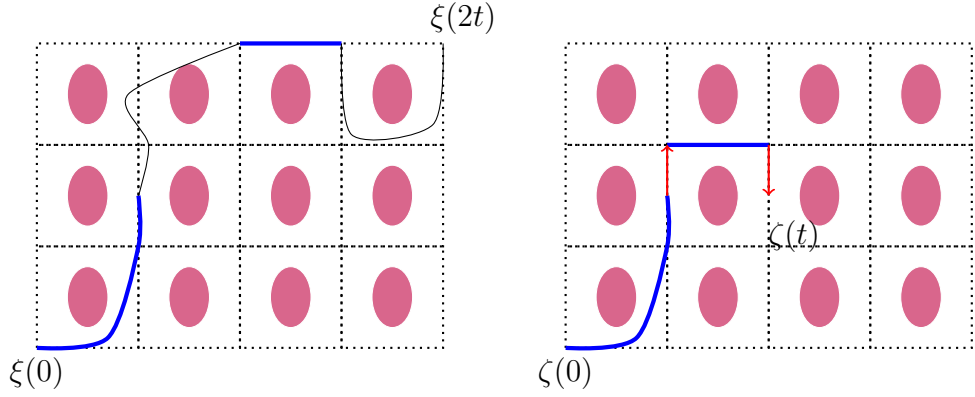
Now we define another curve $\zeta : [0, t] \to \overline{\Omega}$ from $\tilde{\xi}$ such that $\zeta$ is an admissible path for $m^\ast \left(t, 0, y\right)$. 
See Figure \ref{fig:xi}.
From Proposition \ref{prop:mstarbound}, we know

\begin{equation}
    m^\ast\left(2y, 0, 2y\right) = \int_0^{2t} L\left(\xi(s), \dot{\xi}(s)\right)\,ds \leq \left( \frac{C_b^2 \left(M_0 + \sqrt{n}\right)^2}{2}+K_0\right)2t 
\end{equation}
and by using a similar reasoning as that in Proposition \ref{prop:mstar},  there exists $\displaystyle d\in \left\{0, \frac{1}{4}, \frac{1}{2}, \frac{3}{4}, \cdots, \lfloor t \rfloor -\frac{1}{4} \right\}$ such that 
\begin{equation}\label{eqn:xitildebound}
    \int_d^{d+\frac{1}{4}} L\left(\tilde{\xi}(s), \dot{\tilde{\xi}} (s)\right) \,ds\leq C_b^2 \left(M_0+\sqrt{n}\right)^2 + 2K_0.
\end{equation}
Define $\zeta : \left[0, t-\frac{1}{8}\right] \to \mathbb{R}^n$ by
\begin{equation}
  \zeta (s) :=  \left\{\begin{aligned} &\tilde{\xi}(s), \qquad \qquad \qquad \qquad \qquad \quad \text{if } 0 \leq s \leq d,\\
  &\tilde{\xi} \left(2(s-d)+d\right), \, \, \, \, \, \quad \qquad \qquad \text{if } d \leq s \leq d+\frac{1}{8},\\
  &\tilde{\xi}\left(s+\frac{1}{8}\right),\qquad \qquad \qquad \qquad \text{if } d+\frac{1}{8}\leq s \leq t-\frac{1}{8}.\\
    \end{aligned}
    \right.
\end{equation}

Therefore, by \eqref{eqn:K_0L} and \eqref{eqn:xitildebound},
\begin{equation}
\begin{aligned}
    &\left|\int_0^t L\left(\tilde{\xi}(s), \dot{\tilde{\xi}} (s)\right) \,ds - \int_0^{t-\frac{1}{8}} L \left( \zeta(s), \dot{\zeta}(s) \right)\,ds\right|\\
\leq \ & \left|\int_d^{d+\frac{1}{8}} L \left( \tilde{\xi}(2(s-d)+d), \dot{\tilde{\xi}}(2(s-d)+d) \right)\,ds\right| + \left| \int_d^{d+\frac{1}{4}} L\left(\tilde{\xi}(s), \dot{\tilde{\xi}} (s)\right) \,ds\right|\\
\leq \ & \left|\int_d^{d+\frac{1}{4}} L \left( \tilde{\xi}(s), 2\dot{\tilde{\xi}} (s)\right)\,ds\right|+ C_b^2 \left(M_0+\sqrt{n}\right)^2 + \frac{5}{2}K_0\\
\leq\ &\int_d^{d+\frac{1}{4}} \frac{\left|2 \dot{\tilde{\xi}} (s)\right|^2}{2}\,ds+\frac{K_0}{4}+ C_b^2 \left(M_0+\sqrt{n}\right)^2 + \frac{5}{2}K_0\\
\leq\ & 4\int_d^{d+\frac{1}{4}} L\left(\tilde{\xi}(s), \dot{\tilde{\xi}} (s)\right) \,ds+C_b^2 \left(M_0+\sqrt{n}\right)^2 + \frac{27}{4}K_0\\
\leq \ & 5C_b^2 \left(M_0+\sqrt{n}\right)^2 + \frac{59}{4}K_0.
\end{aligned}
\end{equation}
We now create $k+2$ paths connecting $\tilde{\xi}\left(t_j^+\right)$ and $\tilde{\xi} \left(t_j^-\right)$ for $0 \leq j \leq k$, and $\tilde{\xi}\left(t_{k+1}\right)$ and $\tilde{\xi} \left(t_k^+\right)$, each of which takes time $\frac{1}{8(k+2)}$. Similar to the proof of Proposition \ref{prop:mstarbound}, we can find such paths with a velocity bound $8k(k+2)C_b \sqrt{n}$ according to \eqref{eqn:k+1kbound}, \eqref{eqn:1k-1bound}, and \eqref{eqn:k+-bound}. Then we glue these $k+2$ paths to $\zeta$ and get a new admissible path $\tilde{\zeta}$ for $m^\ast\left(t, 0, y\right)$. Then, 
\begin{equation}
    m^\ast\left(t, 0, y\right) \leq \int_0^t L\left(\tilde{\zeta}(s), \dot{\tilde{\zeta}}(s)\right)\,ds\leq \int_0^t L\left( \tilde{\xi}(s), \tilde{\xi}(s)\right)\,ds +C
\end{equation}
for some constant $C$ that is independent of $t, y$. We can show a similar result for $m^\ast\left(t, y, 2y\right)$ and hence
\begin{equation}\label{eqn:split}
   m^\ast\left(t, 0, y\right)+m^\ast\left(t, y, 2y\right) \leq m^\ast\left(2t, 0, 2y\right) + C
\end{equation}
for some constant $C$ that is independent of $t, y$.
Moreover, by a similar argument as in the proof of Lemma \ref{lem:subadd}, we have 
\begin{equation}\label{eqn:movearound}
    m^\ast(t, 0, y) \leq m^\ast(t, y, 2y)+C
\end{equation}
for some constant $C$ that is independent of $t, y$. Combine \eqref{eqn:split} and \eqref{eqn:movearound}, we have
\[
 2m^\ast\left(t, 0, y\right) \leq m^\ast \left(2t, 0, 2y\right) +C.
\]
\end{proof}

\section{Proof of Theorem \ref{thm:main1}}\label{sec:main1}
We proved the subadditivity and the superadditivity of the metric function $m^\ast$ in Section 3. Hence, the following proposition is a quick application of Fekete's lemma and we omit the proof here. (See  \cite[Lemma D.1]{TYWKAM} for a proof.)
\begin{prop}\label{prop:mbar-qual}
    For any $t >0 $ and $x, y \in \mathbb{R}^n$ with $|x-y| \leq M_0t$, the following limit  exists
    \[
    \displaystyle \overline{m}^\ast (t, x, y): = \lim_{k\to \infty} \frac{1}{k} m^\ast \left(kt, kx, ky\right).
    \]
\end{prop}

We refer the reader to \cite[Proposition X.1]{CDL} for some qualitative convergence results which are related to that of Proposition \ref{prop:mbar-qual}.
From the limit in the above proposition, it is clear to see that $\overline{m}^\ast$ is positive homogeneous of degree $1$, that is, for $t,s >0 $ and $x, y \in \mathbb{R}^n$,
\[
\overline{m}^\ast (st, sx, sy) = s\overline{m}^\ast (t, x, y).
\]

The next result is the key point for the main results in this paper.

\begin{thm}\label{thm:mbarmstarrate}
    Let $\varepsilon>0$, $t \geq \varepsilon$ and $x, y \in \mathbb{R}^n$ with $|x-y| \leq M_0 t$. Then, there exists a constant $C=C\left(n, \partial \Omega, M_0, K_0\right)>0$ such that
    \begin{equation}
        \left|\overline{m}^\ast \left(t, x, y\right) - \varepsilon m^\ast \left(\frac{t}{\varepsilon}, \frac{x}{ \varepsilon}, \frac{y}{\varepsilon}\right)\right| \leq C \varepsilon.
    \end{equation}
\end{thm}

\begin{proof}
    Without loss of generality, we can assume $x=0$ and we only prove the direction
    \[\varepsilon m^\ast \left(\frac{t}{\varepsilon}, \frac{0}{ \varepsilon}, \frac{y}{\varepsilon}\right)-\overline{m}^\ast \left(t, 0, y\right) \leq C\varepsilon.\]
    The proof for the other direction is similar.

    Let $\tilde{t} : = \frac{t}{\varepsilon}$ and $\tilde{y}: =\frac{y}{\varepsilon}$. Then $\tilde{t} \geq 1$ and $|\tilde{y}| \leq M_0\tilde{t}$. From Lemma \ref{lem:superadd}, we can iterate \eqref{eqn:superadd} for $k$ times and get
    \begin{equation}\label{eqn:iterate}
    2^k \left(m^\ast\left(\tilde{t}, 0, \tilde{y}\right) -C \right)\leq m^\ast \left(2^k\tilde{t}, 0, 2^k\tilde{y}\right) -C.        
    \end{equation}
Dividing both sides of \eqref{eqn:iterate} by $2^k$ and sending $k$ to infinity, we obtain
\begin{equation}
    m^\ast \left(\tilde{t}, 0, \tilde{y}\right) - C\leq \overline{m}^\ast \left(\tilde{t}, 0, \tilde{y}\right)
\end{equation}
which implies
\[
\varepsilon m^\ast \left(\frac{t}{\varepsilon}, 0, \frac{y}{\varepsilon}\right) - \varepsilon C\leq \varepsilon \overline{m}^\ast \left(\frac{t}{\varepsilon}, 0, \frac{y}{\varepsilon}\right) = \overline{m}^\ast\left(t, 0, y\right).
\]
    
\end{proof}

We have not shown any connection between the limit $\overline{m}^\ast$ and $u$. Next, we show that after replacing the running cost in the original optimal control formula \eqref{eq:HLubar} of $u$ with $\overline{m}^\ast$, we still obtain the limiting solution $u$.

\begin{defn} Let $t > 0$ and $x \in \mathbb{R}^n$. Define
\[\overline{u}(x, t): = \inf \left\{ \overline{m}^\ast \left(t, y, x \right) +g(y): |x -y | \leq M{_0} t, y \in 
\mathbb{R}^n \right\}.\]
\end{defn}

Recall that for all $x\in \R^n$, 
\[
 u(x, t) = \inf \left\lbrace t\overline{L} \left(\frac{x-y}{t}\right) \,ds+g\left(y\right): y\in \R^n\right\rbrace.
\]
by \eqref{eq:HLubar}, and for all $x\in \ep \ol\Omega$,
\[\displaystyle u^\varepsilon (x, t)=\inf \left\{ \varepsilon m\left(\frac{t}{\varepsilon}, \frac{y}{\varepsilon}, \frac{x}{\varepsilon}\right)+g\left(y\right): |x -y | \leq M{_0} t, y \in \varepsilon \overline{\Omega}\right\}.\]

\begin{lem}\label{lem:ubaru}
    Let $t > 0 $ and $x \in \mathbb{R}^n$. Then $u(x, t) = \overline{u}(x, t)$.
\end{lem}
\begin{proof}
Let $\delta>0$. By the qualitative homogenization result, for $\varepsilon$ small enough with $x\in \varepsilon \overline{\Omega}$, we have

\begin{equation}\label{eqn:ueps-u}
    \left|u^\varepsilon (x, t)-u (x, t)\right|< \delta.
\end{equation}
By Theorem \ref{thm:mbarmstarrate}, for any $t \geq \varepsilon$, we have

\begin{equation}\label{eqn:mbarstar-mstar}
    \left|\overline{m}^\ast(t, y, x)-\varepsilon m^\ast \left(\frac{t}{\varepsilon}, \frac{y}{\varepsilon}, \frac{x}{\varepsilon}\right)\right|\leq C \varepsilon,
\end{equation}
 for any $y\in \mathbb{R}^n$ with $|y-x| \leq M_0t$.

For any such $\varepsilon$ and $t \geq \varepsilon$, there exists $y_{\varepsilon, t, x} \in \varepsilon \overline{\Omega}$ with $|y_{\varepsilon, t, x}-x| \leq M_0 t$ such that
\begin{equation}
u^\varepsilon \left(x,t\right)= \varepsilon m \left(\frac{t}{\varepsilon}, \frac{y_{\varepsilon, t, x}}{\varepsilon}, \frac{x}{\varepsilon}\right)+g(y_{\varepsilon, t, x}).
\end{equation}
Note that $y$ can potentially depend on $\varepsilon, t, x$. It follows from Proposition \ref{prp:mstarm} that
\begin{equation} \label{eqn:ueps-mstar}
    \left|u^\varepsilon \left(x,t\right) -\varepsilon m^\ast \left(\frac{t}{\varepsilon}, \frac{y_{\varepsilon, t, x}}{\varepsilon}, \frac{x}{\varepsilon}\right)-g(y_{\varepsilon, t, x})\right| \leq C\varepsilon.
\end{equation}
Therefore, combining \eqref{eqn:ueps-u}, \eqref{eqn:mbarstar-mstar}, and \eqref{eqn:ueps-mstar}, we have
\begin{equation}
    \begin{aligned}
        \overline{u}(x, t) - u(x, t) &\leq \overline{m}^\ast(t, y_{\varepsilon, t, x}, x) + g(y_{\varepsilon, t, x}) - u(x, t)\\
        & \leq \varepsilon m^\ast \left(\frac{t}{\varepsilon}, \frac{y_{\varepsilon, t, x}}{\varepsilon}, \frac{x}{\varepsilon}\right) + C\varepsilon + g\left(y_{\varepsilon, t, x}\right) - u(x, t)\\
        & \leq u^\varepsilon(x,t) + 2C\varepsilon -u(x, t)\\
        &\leq 2C\varepsilon +\delta.\\
    \end{aligned}
\end{equation}
Sending $\varepsilon$ to zero and then $\delta$ to zero, we obtain $ \overline{u}(x, t) - u(x, t) 
\leq 0$ for any $t>0$.

On the other hand, there exists $\overline{y}_{t, x, \delta} \in \mathbb{R}^n$ with $\left|\overline{y}_{t, x, \delta}-x\right| \leq M_0t$ such that 

\begin{equation}
\overline{u}(x, t) + \delta \geq \overline{m}^\ast\left(t, \overline{y}_{t, x, \delta}, x\right) + g\left(\overline{y}_{t, x, \delta}\right).
\end{equation}
Then for $\varepsilon$ small enough with $x, \overline{y}_{t, x, \delta} \in \varepsilon \overline{\Omega}$, we have 
\begin{equation}
    \begin{aligned}
    u(x, t) &\leq u^\varepsilon (x, t) +\delta    \\
    &\leq \varepsilon m^\ast \left(\frac{t}{\varepsilon}, \frac{\overline{y}_{t, x, \delta}}{\varepsilon}, \frac{x}{\varepsilon}\right)+g(\overline{y}_{t, x, \delta})+C\varepsilon +\delta\\
    & \leq \overline{m}^\ast(t, \overline{y}_{t, x, \delta}, x) + g(\overline{y}_{t, x, \delta}) +2C\varepsilon + \delta\\
    & \leq \overline{u}\left(x,t\right) + 2C \varepsilon +2\delta.
    \end{aligned}
\end{equation}
Sending $\varepsilon$ to zero and then $\delta$ to zero, we obtain $ u(x, t) - \overline{u}(x, t) \leq 0$ for any $t > 0$.
\end{proof}


Now, we are ready to prove our first main theorem.
\begin{proof}[Proof of Theorem \ref{thm:main1}]
    If $0 < t < \varepsilon$, by the comparison principle, we know
\[
\left|u^\varepsilon(x,t)-g(x)\right|\leq Ct,
\]
and 
\[
\left|u(x,t)-g(x)\right|\leq Ct,
\]
for some constant $C>0$ that only depends on $H$ and $\|Dg\|_{L^\infty(\R^n)}$.
Therefore,
\[
\left|u^\varepsilon(x, t)-u(x, t)\right| \leq Ct \leq C \varepsilon,
\]
for some constant $C>0$ that only depends on $H$ and $\|Dg\|_{L^\infty(\R^n)}$.

If $t \geq \varepsilon$, from Lemma \ref{lem:ubaru}, it suffices to prove
\[
\left|u^\varepsilon\left(x,t\right)-\overline{u}\left(x,t \right)\right| \leq C \varepsilon.
\]

On the one hand, we have
\begin{equation}
    \begin{aligned}
    \overline{u}\left(x,t\right) &= \inf \left\{ \overline{m}^\ast \left(t, y, x \right) +g(y): |x -y | \leq M{_0} t, y \in \mathbb{R}^n \right\}\\
    &\leq \inf \left\{ \varepsilon m^\ast \left(\frac{t}{\varepsilon}, \frac{y}{\varepsilon}, \frac{x}{\varepsilon} \right) +g(y)  + C \varepsilon: |x -y | \leq M{_0} t, y \in \mathbb{R}^n \right\}\\
    &\leq \inf \left\{ \varepsilon m^\ast \left(\frac{t}{\varepsilon}, \frac{y}{\varepsilon}, \frac{x}{\varepsilon} \right) +g(y) + C \varepsilon: |x -y | \leq M{_0} t, y \in \varepsilon \overline{\Omega}\right\}\\
    & \leq u^\varepsilon \left(x, t\right)+ C \varepsilon,
    \end{aligned}
\end{equation}
where the second line follows from Theorem \ref{thm:mbarmstarrate}.

On the other hand, it follows from Lemma \ref{lem:velocity bound} and Proposition \ref{prp:mstarm} that
\begin{equation} \label{eqn:ueps}
    \begin{aligned}
        u^\varepsilon \left(x, t\right) &= \inf \left\{ \varepsilon m \left(\frac{t}{\varepsilon}, \frac{y}{\varepsilon}, \frac{x}{\varepsilon}\right) + g\left(y\right): |x- y | \leq \left(M_0 +\sqrt{n} \varepsilon\right)t, y \in \varepsilon \overline{\Omega}\right\}\\
        &\leq \inf \left\{ \varepsilon m^\ast \left(\frac{t}{\varepsilon}, \frac{y}{\varepsilon}, \frac{x}{\varepsilon}\right) + g\left(y\right)+ C \varepsilon : |x- y | \leq \left(M_0 +\sqrt{n} \varepsilon\right) t, y \in \varepsilon \overline{\Omega}\right\}.
    \end{aligned}
\end{equation}
 We claim that
 \begin{equation} \label{eqn:restricty}
 \begin{aligned}
    &\inf \left\{ \varepsilon m^\ast \left(\frac{t}{\varepsilon}, \frac{y}{\varepsilon}, \frac{x}{\varepsilon}\right)   + g\left(y\right) : |x- y | \leq \left(M_0 +\sqrt{n} \varepsilon\right)t, y \in \varepsilon \overline{\Omega}\right\}\\ 
    \leq & \inf \left\{ \varepsilon m^\ast \left(\frac{t}{\varepsilon}, \frac{y}{\varepsilon}, \frac{x}{\varepsilon}\right) +  \|Dg\|_{L^\infty} \sqrt{n} \varepsilon+ g\left(y\right) : |x- y | \leq M_0 t, y \in \mathbb{R}^n\right\}.
 \end{aligned}
 \end{equation}

 Suppose $y \notin \varepsilon \overline{\Omega}$, then there exists $\displaystyle \frac{\tilde{y}}{\varepsilon}, \frac{\tilde{x}}{\varepsilon} \in \partial \Omega$ with $\displaystyle \frac{\tilde{y}}{\varepsilon}-\frac{y}{\varepsilon} \in Y$ and $\displaystyle \frac{\tilde{x}}{\varepsilon} -\frac{x}{\varepsilon} \in Y$ such that
 \begin{equation} \label{eqn:mstarmtild}
 m^\ast \left(\frac{t}{\varepsilon}, \frac{y}{\varepsilon}, \frac{x}{\varepsilon}\right)= m \left(\frac{t}{\varepsilon}, \frac{\tilde{y}}{\varepsilon}, \frac{\tilde{x}}{\varepsilon}\right).    
 \end{equation}
Now, by the definition of $m^\ast \left(\frac{t}{\varepsilon}, \frac{\tilde{y}}{\varepsilon}, \frac{x}{\varepsilon}\right)$, we have
 \begin{equation}\label{eqn:ytildx}
  m \left(\frac{t}{\varepsilon}, \frac{\tilde{y}}{\varepsilon}, \frac{\tilde{x}}{\varepsilon}\right) \geq m^\ast \left(\frac{t}{\varepsilon}, \frac{\tilde{y}}{\varepsilon}, \frac{x}{\varepsilon}\right).
 \end{equation}
Hence, combining \eqref{eqn:mstarmtild} and \eqref{eqn:ytildx}, we obtain
 \[
 \varepsilon m^\ast \left(\frac{t}{\varepsilon}, \frac{y}{\varepsilon}, \frac{x}{\varepsilon}\right) + g\left(y\right) +  \left\|Dg\right\|_{L^\infty}\sqrt{n} \varepsilon \geq \varepsilon m^\ast \left(\frac{t}{\varepsilon}, \frac{\tilde{y}}{\varepsilon}, \frac{x}{\varepsilon}\right) + g(\tilde{y}),
 \]
which gives us \eqref{eqn:restricty}.
Moreover, it follows from Theorem \ref{thm:mbarmstarrate} that
 \begin{equation} \label{eqn:rembar}
     \begin{aligned} 
     &\inf \left\{ \varepsilon m^\ast \left(\frac{t}{\varepsilon}, \frac{y}{\varepsilon}, \frac{x}{\varepsilon}\right) +  \|Dg\|_{L^\infty} \sqrt{n} \varepsilon+ g\left(y\right) : |x- y | \leq M_0 t, y \in \mathbb{R}^n\right\}\\
     \leq \ & \inf \left\{ \overline{m}^\ast \left(t, y, x\right)+C\varepsilon + g\left(y\right): |x- y | \leq M_0 t, y \in \mathbb{R}^n\right\}\\
     \leq\ & \overline{u}\left(x, t\right) + C\varepsilon.
     \end{aligned}
 \end{equation}
 Combining \eqref{eqn:ueps}, \eqref{eqn:restricty}, and \eqref{eqn:rembar}, we have 
 \[
 u^\varepsilon(x, t) \leq \overline{u}\left(x, t\right) +C \varepsilon.
 \]

\end{proof}

\section{Homogenization in a dilute setting}\label{sec:main2}
In this section we prove Theorem \ref{thm:main2}. The following setting is always assumed:
Let $D\Subset (-\frac{1}{2},\frac{1}{2})^n$ be an open connected set with $C^1$ boundary containing $0$ and $\eta:[0,\frac{1}{2})\to [0,\frac{1}{2})$ be such that $\lim_{\ep\to 0} \eta(\ep)=0$. For $\ep>0$,
\[
\Omega^{\eta(\ep)} =\R^n\setminus \bigcup_{m\in \Z^n}(m+\eta(\ep) \ol D),\qquad  \text{ and }  \qquad \Omega_\ep = \ep \Omega^{\eta(\ep)}.
\]

Recall the optimal control formulas for $\tilde{u}^\varepsilon$, for $(x,t)\in \mathbb{R}^n \times [0,\infty)$,
\begin{equation}
    \begin{aligned}
        &\tilde{u}^\varepsilon (x, t)\\
        =\ & \inf \left\lbrace \int_0^t L\left( \frac{\xi(s)}{\varepsilon} ,\dot{\xi}(s)\right) \,ds+g\left(\xi(0)\right): \xi\in \mathrm{AC}([0,t];\mathbb{R}^n), \xi(t) = x\right\rbrace\\
        =\ &\inf \left\lbrace \varepsilon \int_0^{\frac{t}{\varepsilon}} L\left(\gamma(s),\dot{\gamma}(s)\right) \,ds+g\left(\varepsilon \gamma(0)\right): \gamma\in \mathrm{AC}\left(\left[0,\frac{t}{\varepsilon}\right];\mathbb{R}^n\right), \gamma\left(\frac{t}{\varepsilon}\right) = \frac{x}{\varepsilon}\right\rbrace.
    \end{aligned}
\end{equation}

We give the proof of Theorem \ref{thm:main2}.
\begin{proof}[Proof of Theorem \ref{thm:main2}]
Let $\varepsilon \in \left(0, \frac{1}{2}\right)$, $t>0$, and $x \in \overline{\Omega}_\varepsilon$. 
By \cite[Theorem 1.1]{TY2022},
\[
\|\tilde u^\ep - \tilde u\|_{L^\infty(\R^n \times [0,\infty))} \leq C\ep,
\]
which yields the first and last inequalities in \eqref{eqn:main2}. 

By the definition of optimal control formulas for $\tilde{u}^\varepsilon$ and $u^\varepsilon$, we have
\[
\tilde{u}^\varepsilon(x, t) \leq u^\varepsilon\left(x, t\right),
\]
which proves the second inequality in \eqref{eqn:main2}. The rest is to show the third inequality in \eqref{eqn:main2}, i.e., 
\begin{equation}\label{eqn:main2-remaining}
u^\varepsilon\left(x, t\right) \leq \tilde{u}^\varepsilon \left(x, t\right) +C\left(\varepsilon + \eta\left(\varepsilon\right)t\right).
\end{equation}

Let $\gamma:\left[0,\frac{t}{\varepsilon}\right] \to \mathbb{R}^n$ be an optimal path for $\tilde{u}^\varepsilon(x, t)$, that is, 
\[
\tilde{u}^\varepsilon\left(x, t\right)= \varepsilon \int_0^{\frac{t}{\varepsilon}} L\left(\gamma(s),\dot{\gamma}(s)\right) \,ds+g\left(\varepsilon \gamma(0)\right),
\]
with $\gamma\left(\frac{t}{\varepsilon}\right) =\frac{x}{\varepsilon}$. 
Thanks to \cite{Tran}, we have $\left\|\dot{\gamma}\right\|_{L^\infty\left[0, \frac{t}{\varepsilon}\right]} \leq M_0$.

For $k \in \mathbb{Z}^n$, we denote $D_k^\eta:=k+\eta(\varepsilon) D$ and 
\[J: = \left\{k \in \mathbb{Z}^n: \gamma \left(\left[0, \frac{t}{\varepsilon}\right]\right) \cap D_k^\eta \neq \emptyset \right\}.
\]
Note that $J$ collects all the indices $k$ such that $\gamma$ intersects with $D_k^\eta$. 
Of course, if $J=\emptyset$, then \eqref{eqn:main2-remaining} holds immediately.
Let us assume that $J \neq \emptyset$ from now on.
We claim that $\left|J\right| \leq \frac{2M_0t}{\varepsilon}$. Indeed, for $k, j \in \mathbb{Z}^n$ with $k \neq j$, $\dist\left(D_k^\eta, D_j^\eta\right) > \frac{1}{2}$ as $\eta(\ep) \in (0,\frac{1}{2})$. Since $\left\|\dot{\gamma}\right\|_{L^\infty\left[0, \frac{t}{\varepsilon}\right]} \leq M_0$, it takes at least $\frac{1}{2M_0}$ in time for $\gamma$ to travel from $D_k^\eta$ to $D_j^\eta$. Therefore,
\[
\left|J\right| \leq \frac{\frac{t}{\varepsilon}}{\frac{1}{2M_0}} \leq \frac{2M_0t}{\varepsilon}.
\]
Without loss of generality, we can assume $D_{k_1}^\eta$ is the first hole that $\gamma$ enters, and $\gamma$ enters $D_{k_{i+1}}^\eta$ after the final exit from $D_{k_i}^\eta$ for $i \in \left\{1, 2, \cdots, |J|\right\}$. Define $t_0=0$ and
\begin{equation*}
    \begin{aligned}
        s_i: &= \inf \left\{s: \gamma(s)\in D_{k_i}^\eta, s\geq t_{i-1}\right\},\\
        t_i: &= \sup \left\{s: \gamma(s)\in D_{k_i}^\eta, s\geq s_i\right\},
    \end{aligned}
\end{equation*}
for $i \in \{1, 2, \cdots, |J|\}$. Intuitively, $s_i$ is the first time $\gamma$ intersects $D_{k_i}^\eta$ after its final exit from $D_{k_{i-1}}^\eta$, while $t_i$ is the last time point at which $\gamma$ exits $D_{k_i}^\eta$. Since $\left\|\dot{\gamma}\right\|_{L^\infty} \leq M_0$, we have 
\[
t_1-s_1 \geq \frac{\left|\gamma(t_1) - \gamma(s_1)\right|}{M_0}.
\]
We claim that for some constant $C > 0$ independent of $\ep,\eta(\ep)$ and determined later,
\[
\int_{s_1}^{t_1} L \left(\gamma(s), \dot{\gamma}(s)\right) \, ds \leq C\eta (\varepsilon),
\]
and there exists a path $\tilde{\gamma}_1:[s_1, t_1] \to \overline{\Omega}^{\eta(\varepsilon)} $ such that
\begin{equation}\label{eqn:tildecostbound}
 \int_{s_1}^{t_1} L \left(\tilde{\gamma}_1(s), \dot{\tilde{\gamma}}_1(s)\right) \, ds \leq C \eta(\varepsilon).   
\end{equation}

Indeed, consider a path $\xi:[s_1, t_1] \to \mathbb{R}^n$ defined by
\begin{equation}
  \xi (s) :=  \left\{\begin{aligned} & \frac{\gamma(t_1) -\gamma(s_1)}{|\gamma(t_1) -\gamma(s_1)|} M_0\left(s-s_1\right) + \gamma(s_1), \qquad  \text{if } s \in \left[s_1, s_1+\frac{\left|\gamma(t_1) - \gamma(s_1)\right|}{M_0}\right],\\
  &\gamma (t_1) \qquad \qquad  \qquad \qquad \qquad \quad \qquad \qquad \text{if } s \in \left[s_1+\frac{\left|\gamma(t_1) - \gamma(s_1)\right|}{M_0}, t_1\right].\\
    \end{aligned}
    \right.
\end{equation}
This path essentially represents the line segment from $\gamma(t_1)$ to $\gamma(s_1)$, traversed at a constant velocity $M_0$. After reaching $\gamma(t_1)$, the path remains at $\gamma(t_1)$ for the rest of the time. Then,
\[
\begin{aligned}
    \int_{s_1}^{t_1} L\left(\xi(s), \dot{\xi}(s)\right) \, ds =& \int_{s_1}^{s_1+\frac{\left|\gamma(t_1) - \gamma(s_1)\right|}{M_0}}L \left(\xi(s), \dot{\xi}(s)\right) \, ds\\
    \leq & \left(\frac{M_0^2}{2}+K_0\right) \frac{\left|\gamma(t_1) - \gamma(s_1)\right|}{M_0}\\
    \leq & \left(\frac{M_0}{2} +\frac{K_0}{M_0}\right) \eta(\varepsilon),
\end{aligned}
\]
where the first equality comes from (A5) and Lemma \ref{lem:A5}, and the second inequality follows from \eqref{eqn:K_0L}.
Since $\gamma$ is an optimal path for $\tilde{u}^\varepsilon(x, t)$, the cost functional with respect to $\gamma$ is optimal on any subinterval of $[0,\frac{t}{\ep}]$. 
As $\xi(s_1)=\gam(s_1)$ and $\xi(t_1)=\gam(t_1)$,
\[
\int_{s_1}^{t_1} L\left(\gamma(s), \dot{\gamma}(s)\right)ds \leq \int_{s_1}^{t_1} L\left(\xi(s), \dot{\xi}(s)\right) \, ds \leq C\eta(\varepsilon).
\]

If $\gamma(0) \in D_{k_1}^{\eta}$, then $s_1=0$ and $\gamma(s_1)=\gamma(0) \in D_{k_1}^{\eta}$. 
We can define $\tilde{\gamma}_1:[0, t_1] \to \overline{\Omega}^{\eta(\varepsilon)}$ by $\tilde{\gamma}_1(s) =\gamma(t_1)$ for $s\in [0,t_1]$.
Then,
\[
\int_0^{t_1} L(\gamma(t_1), 0)\, ds =0,
\]
in which case \eqref{eqn:tildecostbound} obviously holds. Note that $\left|\gamma(0)- \tilde{\gamma}_1(0)\right| \leq \eta(\varepsilon)$. 

If $\gamma(0) \in \overline{\Omega}^{\eta(\varepsilon)}$, we can revise $\xi$ and construct a path $\tilde{\gamma}_1: [s_1, t_1] \to \overline{\Omega}^{\eta(\varepsilon)}$, 
where $\tilde{\gamma}_1$ runs on $\partial D_{k_1}^{\eta}$ and connect $\gamma(s_1)$ to $\gamma(t_1)$ on $\left[s_1, s_1+\frac{\left|\gamma(t_1) - \gamma(s_1)\right|}{M_0}\right]$, and $\tilde{\gamma}_1$ stays at $\gamma(t_1)$ for the rest of the time. Then, we obtain
\[
\int_{s_1}^{t_1} L\left(\tilde{\gamma}_1(s), \dot{\tilde{\gamma}}_1(s)\right) \, ds \leq \left(\frac{C_b^2M_0}{2} +\frac{K_0}{M_0}\right) \eta(\varepsilon).
\]
Similarly, for each $i=1,\dots,|J|$, one can prove
\begin{equation}\label{eqn:gammacost}
    \int_{s_i}^{t_i} L\left(\gamma(s), \dot{\gamma}(s)\right)ds \leq  \left(\frac{M_0}{2} +\frac{K_0}{M_0}\right) \eta(\varepsilon),
\end{equation}
and construct $\tilde{\gamma}_i: [s_i, t_i] \to \overline{\Omega}^{\eta(\varepsilon)}$ with $\tilde{\gamma}_i (s_i) = \gamma(s_i)$, $\tilde{\gamma}_i(t_i)=\gamma(t_i)$, and
\begin{equation}\label{eqn:tildegammacost}
    \int_{s_i}^{t_i} L\left(\tilde{\gamma}_i(s), \dot{\tilde{\gamma}}_i(s)\right) \, ds \leq \left(\frac{C_b^2M_0}{2} +\frac{K_0}{M_0}\right) \eta(\varepsilon),
\end{equation}
Hence, when necessary we can revise $\gamma$ to a new path $\tilde{\gamma}: \left[0, \frac{t}{\varepsilon}\right] \to \overline{\Omega}^{\eta(\varepsilon)}$ by
\begin{equation}
  \tilde{\gamma} (s) :=  \left\{\begin{aligned} & \gamma(s), \qquad \quad  \text{if } s \in \left[0, \frac{t}{\varepsilon}\right] \setminus \left(\bigcup_{i=1}^{|J|} [s_i, t_i]\right),\\
  &\tilde{\gamma}_i(s), \qquad \quad \text{if } s \in \left[s_i, t_i\right], i = 1, 2, \cdots, |J|,\\
    \end{aligned}
    \right.
\end{equation}
which is an admissible path for $u^\varepsilon(x, t)$. Therefore,
\[
\begin{aligned}
    u^\varepsilon(x, t) &\leq  \varepsilon \int_0^\frac{t}{\varepsilon} L \left(\tilde{\gamma}(s), \dot{\tilde{\gamma}}(s)\right) \, ds + g\left(\varepsilon \tilde{\gamma}(0)\right) \\
    &\leq  \varepsilon \int_0^\frac{t}{\varepsilon} L \left(\gamma(s), \dot{\gamma}(s)\right) \, ds + g\left(\varepsilon \gamma(0)\right)+ C \varepsilon |\gamma(0)-\tilde{\gamma}(0)|\\
    &\qquad+\varepsilon |J| \left(\frac{M_0}{2} +\frac{K_0}{M_0}\right) \eta(\varepsilon)+ \varepsilon |J| \left(\frac{C_b^2M_0}{2} +\frac{K_0}{M_0}\right) \eta(\varepsilon)\\
    &\leq  \tilde{u}^\varepsilon(x, t) + C \varepsilon \eta(\varepsilon)+t \eta(\varepsilon)\left(4K_0+M_0^2 + C_b^2 M_0^2\right),
\end{aligned}
\]
where the second inequality follows from (A4), \eqref{eqn:gammacost}, and \eqref{eqn:tildegammacost}, and the last inequality comes from the fact that $|J| \leq \frac{2M_0t}{\varepsilon}$.
Thus, \eqref{eqn:main2-remaining} holds and the proof is complete.

\end{proof}

In the above proof, we obtained
\begin{equation}\label{eqn:main2-better}
     u^\varepsilon(x, t)\leq  \tilde{u}^\varepsilon(x, t) + C \varepsilon \eta(\varepsilon)+t \eta(\varepsilon)\left(4K_0+M_0^2 + C_b^2 M_0^2\right),
\end{equation}
which is a bit stronger than \eqref{eqn:main2-remaining}.

We show the obtained convergence rate in Theorem \ref{thm:main2} and \eqref{eqn:main2-better} are essentially optimal through the following lemma.
\begin{lem}\label{lem:thm2-optimal}
Consider $n=2$. Let $D: = B\left(0, \frac{1}{4}\right) \subset \mathbb{R}^2$ and $\eta:\left[0,\frac{1}{3}\right)\to \left[0,\frac{1}{3}\right)$ with $\lim_{\ep\to 0} \eta(\ep)=0$ such that
    $\Omega^{\eta(\ep)} =\R^2\setminus \bigcup_{m\in \Z^2}(m+\eta(\ep) \ol D)$ and $\Omega_\ep = \ep \Omega^{\eta(\ep)}$ for $\ep \in (0,\frac{1}{3})$.
Assume
    \[
    H(y,p):=\frac{a(y)|p|^2}{2}\qquad \text{ for } (y,p)\in  \R^2\times \R^2,
    \]
where $a\in \Lip(\R)$, which is $\Z$-periodic and satisfies $a(y): =1-|y_2|$ for $y_2 \in \left[-\frac{1}{2}, \frac{1}{2} \right]$. 
Let $g(x)= -x_1$ for $x \in \mathbb{R}^2$. 

For $\varepsilon \in \left(0, \frac{1}{3}\right)$, let $u^\varepsilon: \overline{\Omega}_\varepsilon \times [0, \infty) \to \mathbb{R}$ be the unique viscosity solution to \eqref{eqn:PDE_epsilon} and $\tilde{u}^\varepsilon : \mathbb{R}^2\times [0, \infty) \to \mathbb{R}$ be the unique viscosity solution to \eqref{eqn:whole_epsilon}, respectively. Let $\left\{e_1, e_2\right\}$ be the canonical basis for $\mathbb{R}^2$. Then, there exists a constant $C> 0$ independent of $\varepsilon\in (0,\frac{1}{3})$ such that 
\begin{equation}\label{eq:main2-rigorous}
u^\varepsilon\left(\frac{\varepsilon\eta(\varepsilon)}{4}e_2, 1\right) \geq \tilde{u}^\varepsilon\left(\frac{\varepsilon\eta(\varepsilon)}{4}e_2, 1\right) - C\varepsilon \eta(\varepsilon)-C \varepsilon+ C \eta(\varepsilon)^2.
\end{equation}
\end{lem}

\begin{proof}
For $y_2 \in \left[-\frac{1}{2}, \frac{1}{2} \right]$ and $v \in \mathbb{R}^2$, we have
\[
L(y, v) = \frac{|v|^2}{2a(y)}= \frac{|v|^2}{2\left(1-|y_2|\right)}.
\]
In particular, $L(se_1, v) = \frac{|v|^2}{2}$ for $s \in \mathbb{R}$.

\smallskip

We first compute $\tilde{u}^\varepsilon(0, 1)$ and find an optimal path for $\tilde{u}^\varepsilon(0, 1)$. Suppose that $\gamma: \left[0, \frac{1}{\varepsilon}\right] \to \mathbb{R}^2$ is an optimal path for $\tilde{u}^\varepsilon(0, 1)$, that is,
\[
\tilde{u}^\varepsilon(0, 1) = \varepsilon \int_0^\frac{1}{\varepsilon} L \left(\gamma(s), \dot{\gamma}(s)\right)\, ds +g (\varepsilon \gamma(0)),
\]
and $\gamma\left(\frac{1}{\varepsilon}\right)=0$. 
We compute
\[
\begin{aligned}
    \tilde{u}^\varepsilon(0, 1) &\geq  \varepsilon \int_0^\frac{1}{\varepsilon} \frac{\left|\dot{\gamma}(s)\right|^2}{2}\, ds + g (\varepsilon \gamma(0))\\
    &\geq  \frac{1}{2}\left| \varepsilon \int_0^\frac{1}{\varepsilon}\dot{\gamma}(s) \, ds \right|^2+g (\varepsilon \gamma(0))\\
    &\geq  \frac{1}{2} \left|\varepsilon \gamma(0)\right|^2 - \left|\varepsilon \gamma(0)\right|\\
    &\geq  -\frac{1}{2}.
\end{aligned}
\]
The second inequality above follows from Jensen's inequality.
Further, the minimum is attained at $|\varepsilon \gamma(0)|=1$. In fact, 
\begin{equation}\label{eq:tilde-u-ep-1-2-opt}
  \tilde{u}^\varepsilon(0, 1) = - \frac{1}{2}  
\end{equation}
and $\tilde{\gamma}: \left[0, \frac{1}{\varepsilon}\right] \to \mathbb{R}^2$ defined by 
\[
\tilde{\gamma}(s) : = \left(\frac{1}{\varepsilon} -s \right)e_1  \qquad \text{for $s \in \left[0, \frac{1}{\varepsilon}\right]$}
\]
is an optimal path for $ \tilde{u}^\varepsilon(0, 1)$, which is a straight line segment connecting $\left(\frac{1}{\varepsilon}, 0\right)$ and $(0, 0)$.
As $\tilde u^\ep$ is globally Lipschitz, we use \eqref{eq:tilde-u-ep-1-2-opt} to yield
\begin{equation}\label{eq:tilde-u-ep-near-by}
    \tilde{u}^\varepsilon\left(\frac{\varepsilon\eta(\varepsilon)}{4}e_2, 1\right) \leq -\frac{1}{2} + C\ep\eta(\ep).
\end{equation}

Next, we estimate the value of $u^\varepsilon\left(\frac{\varepsilon\eta(\varepsilon)}{4}e_2, 1\right)$. Suppose $\xi : \left[0, \frac{1}{\varepsilon}\right] \to \overline{\Omega}^{\eta(\varepsilon)}$ is an optimal path for $u^\varepsilon\left(\frac{\varepsilon\eta(\varepsilon)}{4}e_2, 1\right)$, that is,
\[
u^\varepsilon\left(\frac{\varepsilon\eta(\varepsilon)}{4}e_2, 1\right) = \varepsilon \int_0^\frac{1}{\varepsilon} L \left(\xi(s), \dot{\xi}(s)\right)\, ds +g (\varepsilon \xi(0)),
\]
and $\xi\left(\frac{1}{\varepsilon}\right)=\frac{\eta(\ep)}{4}e_2$. For $i \in \mathbb{Z}$, we define 
\[
S_i: = \left\{(y_1, y_2) \in \mathbb{R}^2: y_1 \in \left[i-\frac{\eta(\varepsilon)}{8}, i+\frac{\eta(\varepsilon)}{8}\right]\right\},
\]
which is a vertical strip region of width $\frac{\eta(\ep)}{4}$ centered at $ie_1$. We will estimate the running cost of $\xi$ based on whether it falls within or outside the strip regions. Let $\xi(0)=\left(\xi_1(0), \xi_2(0)\right) \in \mathbb{R}^2$. The initial condition $g(x) = -x_1$ implies that $\xi(0)$ falling in the negative part of the $e_1$ axis would incur higher costs. In fact, as long as $\xi_1(0)<3/2$, the control formula yields 
$$
u^\ep\left(\frac{\eta(\ep)}{4}e_2,1\right) > -\frac{3\ep}{2} \ge -\frac12 \ge  \tilde u^\ep\left(\frac{\eta(\ep)}{4}e_2,1\right) - C\ep \eta(\ep),
$$
which confirms \eqref{eq:main2-rigorous}. Hence we can, without loss of generality, assume that $\xi_1(0) \in \left[K+ \frac{1}{2}, K+\frac{3}{2}\right]$ for some positive integer $K \in \mathbb{N}$. Then, $\xi\left(\left[0, \frac{1}{\varepsilon}\right]\right)$ passes through at least $K$ strip regions, i.e., $S_1, S_2, \cdots, S_K$. Let
\[
I: = \left\{s \in \left[0, \frac{1}{\varepsilon}\right]: \xi(s) \in S_i \text{ for some } i \in \mathbb{N}\right\},
\]
which collects all the time that $\xi$ is in a vertical strip region. Since $\left\|\dot{\xi}\right\|_{L^\infty\left[0, \frac{1}{\varepsilon}\right]} \leq M_0$, we have $|I| \geq \frac{K \eta(\varepsilon)}{4M_0}$ where $|I|$ denotes the measure of $I$. Similarly, $\xi\left(\left[0, \frac{1}{\varepsilon}\right]\right)$ passes through at least $K$ gap regions between $S_i$'s, each of which is also a strip region of width $1-\frac{\eta(\varepsilon)}{4}$. Hence, for $I^c=\left[0, \frac{1}{\varepsilon}\right]\setminus I$, we have $|I^c| \geq \frac{K}{M_0}\left(1-\frac{\eta(\varepsilon)}{4}\right)$. Note that if $\xi(s) \in S_i \cap \overline{\Omega}^{\eta(\varepsilon)}$ for some $i \in \mathbb{Z}$, then
\begin{equation}\label{eq:main2-L-strip}
L\left(\xi(s), \dot{\xi}(s)\right) =\frac{\left|\dot{\xi}(s)\right|^2}{2\left(1 -\xi_2(s)\right)} \geq \frac{\left|\dot{\xi}(s)\right|^2}{2} \left(1+\frac{\eta(\varepsilon)}{8}\right).
\end{equation}
This is because $\xi$ has to avoid the hole of radius $\frac{\eta(\varepsilon)}{4}$ and the strip is centered on the same point as the hole and has a width of $\frac{\eta(\varepsilon)}{4}$.
Besides, for $s\in I^c=\left[0, \frac{1}{\varepsilon}\right] \setminus I$, we simply have
\begin{equation}\label{eq:main2-L-rough}
L\left(\xi(s), \dot{\xi}(s)\right) \geq  \frac{\left|\dot{\xi}(s)\right|^2}{2}.
\end{equation}
We use \eqref{eq:main2-L-strip} and \eqref{eq:main2-L-rough} to imply
\begin{equation}\label{eqn:ueplowerbound}
    \begin{aligned}
u^\varepsilon\left(\frac{\varepsilon\eta(\varepsilon)}{4}e_2, 1\right) &= \varepsilon \int_0^\frac{1}{\varepsilon} L \left(\xi(s), \dot{\xi}(s)\right)\, ds - \varepsilon \xi_1(0)\\
&\geq  \frac{\varepsilon}{2} \int_{I^c} \left|\dot{\xi}(s)\right|^2 \, ds + \frac{\varepsilon}{2} \left(1 +\frac{\eta(\varepsilon)}{8}\right) \int_I \left|\dot{\xi}(s)\right|^2 \, ds - \varepsilon \xi_1(0)\\
&\geq  \frac{\varepsilon}{2|I^c|} \left(K \left(1-\frac{\eta(\varepsilon)}{4}\right)\right)^2 + \frac{\varepsilon}{2|I|} \left(1 +\frac{\eta(\varepsilon)}{8}\right) \left(\frac{K\eta(\varepsilon)}{4}\right)^2 - \varepsilon \xi_1(0)\\
& = \frac{\varepsilon}{2} \varphi\left(\left|I^c\right|\right) - \varepsilon \xi_1(0)
    \end{aligned}
\end{equation}
where the third line follows from Jensen's inequality and 
\[
\varphi(z):= \frac{1}{z} \left(K \left(1-\frac{\eta(\varepsilon)}{4}\right)\right)^2 + \frac{1}{\frac{1}{\varepsilon}-z} \left(1 +\frac{\eta(\varepsilon)}{8}\right) \left(\frac{K\eta(\varepsilon)}{4}\right)^2.
\]
In order to obtain a good estimate of $u^\varepsilon\left(\frac{\varepsilon\eta(\varepsilon)}{4}e_2, 1\right)$, we need to minimize the last line in \eqref{eqn:ueplowerbound} over $|I^c| \in \left[0, \frac{1}{\varepsilon}\right]$. Essentially, we minimize the function $\varphi$ over $z=|I^c| \in \left[0, \frac{1}{\varepsilon}\right]$. By setting $\varphi^\prime(z)=0$ and solving for $z$, we get the critical point $z_0 = {\varepsilon^{-1}\left(\sqrt{1+\frac{\eta(\varepsilon)}{8}}\frac{\eta(\varepsilon)}{4}+1-\frac{\eta(\varepsilon)}{4}\right)^{-1}}\left(1-\frac{\eta(\varepsilon)}{4}\right)$ and 
\[
\min_{z \in \left[0, \frac{1}{\varepsilon}\right]}\varphi(z)= \varphi(z_0)= \varepsilon K^2 \left(1+\left(\sqrt{1+\frac{\eta(\varepsilon)}{8}}-1\right)\frac{\eta(\varepsilon)}{4}\right)^2.
\]
Plugging back in \eqref{eqn:ueplowerbound}, we obtain
\begin{equation}\label{eqn:uepxi1}
\begin{aligned}
u^\varepsilon\left(\frac{\varepsilon\eta(\varepsilon)}{4}e_2, 1\right) &\geq \frac{\varepsilon^2 K^2}{2}\left(1+\left(\sqrt{1+\frac{\eta(\varepsilon)}{8}}-1\right)\frac{\eta(\varepsilon)}{4}\right)^2 - \varepsilon \xi_1(0)\\
&\geq \frac{\varepsilon^2 K^2}{2}\left(1+2\left(\sqrt{1+\frac{\eta(\varepsilon)}{8}}-1\right)\frac{\eta(\varepsilon)}{4}\right) - \varepsilon \xi_1(0)\\
&\geq \left(\frac{1}{2}+\frac{\eta(\varepsilon)^2}{64}\right) \varepsilon^2 K^2 - \varepsilon\xi_1(0)\\
& \geq \left(\frac{1}{2}+\frac{\eta(\varepsilon)^2}{64}\right) \varepsilon^2 \left(\xi_1(0) - \frac{3}{2}\right)^2 - \varepsilon\xi_1(0),
\end{aligned}    
\end{equation}
where the fourth inequality comes from the fact that $ \xi_1(0) \leq K+\frac{3}{2}$. By further computations, we have
\begin{equation}\label{eqn:epxi1estimate}
    \begin{aligned}
    &\left(\frac{1}{2}+\frac{\eta(\varepsilon)^2}{64}\right) \varepsilon^2 \left(\xi_1(0) - \frac{3}{2}\right)^2 - \varepsilon\xi_1(0)\\
    \geq &\left(\frac{1}{2}+\frac{\eta(\varepsilon)^2}{64}\right) \left(\varepsilon\xi_1(0)\right)^2 -\left(1+ 3\varepsilon\left(\frac{1}{2}+\frac{\eta(\varepsilon)^2}{64}\right) \right) \varepsilon \xi_1(0) +\frac{9}{8} \varepsilon^2\\
    \geq & -\frac{\left(1+ 3\varepsilon\left(\frac{1}{2}+\frac{\eta(\varepsilon)^2}{64}\right)\right)^2}{2+ \frac{\eta(\varepsilon)^2}{16}}+\frac{9}{8} \varepsilon^2\\
    \geq &  -\frac{1}{2+ \frac{\eta(\varepsilon)^2}{16}} - \frac{3}{2} \varepsilon -\frac{9}{4} \varepsilon^2 + \frac{9}{8} \varepsilon^2\\
    \geq & -\frac{1}{2} +\frac{\eta\left(\varepsilon\right)^2}{100} -\frac{3}{2} \varepsilon -\frac{9}{8} \varepsilon^2\\
    \geq &  -\frac{1}{2} + \frac{\eta\left(\varepsilon\right)^2}{100} -3 \varepsilon.
    \end{aligned}
\end{equation}
Combining \eqref{eqn:uepxi1}, \eqref{eqn:epxi1estimate}, and \eqref{eq:tilde-u-ep-near-by}, we obtain
\[
u^\varepsilon\left(\frac{\varepsilon\eta(\varepsilon)}{4}e_2, 1\right) \geq \tilde{u}^\varepsilon\left(\frac{\varepsilon\eta(\varepsilon)}{4}e_2, 1\right) - C\varepsilon \eta(\varepsilon) - 3 \varepsilon+ \frac{\eta(\varepsilon)^2}{100}.
\]
\end{proof}

\begin{rem}
    In the inequality \eqref{eq:main2-rigorous}, we have the term $C\eta(\ep)^2$ instead of $C\eta(\ep)$ because we have only used the rough lower bound \eqref{eq:main2-L-rough} and have not utilized the minimizing property of the action functional on $I^c$.
    
    Nevertheless, as $\eta(\ep)$ can tend to $0$ as slow as it wants as $\ep \to 0$,  \eqref{eq:main2-rigorous} implies that \eqref{eqn:main2-better} is essentially optimal.
\end{rem}

We now use Theorem \ref{thm:main2} to deduce a convergence rate of $\ol H_{\Om^{\eta(\ep)}}$ to $\ol H_0$ under assumptions (A1)--(A3) and (A5).

\begin{cor}\label{cor:main2}
    Assume the settings of Theorem \ref{thm:main2}.
    For $R>0$, there exists $C=C(n,\partial D, H, R)>0$ such that, for $\ep\in (0,1)$,
    \[
    \left| \ol H_{\Om^{\eta(\ep)}}(p) - \ol H_0(p) \right| \leq C(\ep + \eta(\ep)) \qquad \text{ for } |p| \leq R.
    \]
\end{cor}

\begin{proof}
    Fix $p$ with $|p| \leq R$.
    Let $g(x)=p\cdot x$ for $x\in \R^n$.
    By Theorem \ref{thm:main2} and Remark \ref{rem:1}, for $x\in B(0,1) \cap \Om_\ep$,
    \[
    |u^\ep(x,1)-\tilde u(x,1)| \leq C(\ep +\eta(\ep)),
    \]
    and
    \[
    |u^\ep(x,1) - \tilde u^{\eta(\ep)}(x,1)| \leq C\ep.
    \]
    Combine the two inequalities to yield
    \begin{equation}\label{eq:cor21}
    |\tilde u(x,1)-\tilde u^{\eta(\ep)}(x,1)| \leq C(\ep +\eta(\ep)).
    \end{equation}
    
    On the other hand, as $g(x)=p\cdot x$, we have that, for $(x,t)\in \R^n \times [0,\infty)$,
    \begin{equation}\label{eq:cor22}
        \tilde u(x,t) = p\cdot x - \ol H_0(p) t, \qquad \text{ and } \qquad \tilde u^{\eta(\ep)}(x,t)= p\cdot x - \ol H_{\Om^{\eta(\ep)}}(p) t.
    \end{equation}
    Combine \eqref{eq:cor21} and \eqref{eq:cor22} to conclude.
\end{proof}

\begin{rem}
    Assume the settings of Theorem \ref{thm:main2}.
    Let $\eta(\ep)=\ep$ for $\ep>0$.
    For $R>0$, there exists $C=C(n,\partial D, H, R)>0$ such that, for $\ep \in (0,1)$,
    \begin{equation}\label{eq:rate-H-ep-H}
    \left| \ol H_{\Om^{\ep}}(p) - \ol H_0(p) \right| \leq C\ep  \qquad \text{ for } |p| \leq R.
    \end{equation}
    Note that $\Om^\ep \neq \Om_\ep$, and more precisely,
    \[
    \Om^\ep = \R^n \setminus \bigcup_{m\in \Z^n} (m+\ep \ol D).
    \]
    It seems that the rate in \eqref{eq:rate-H-ep-H} is the best one can hope for.
    However, this has not been studied in the literature and deserves some attention.
    Here, $\{\Om^\ep\}$ does not have the scaling property as those in \cite{KTT, Tu2, TuZhang}.
\end{rem}

\section{Homogenization with domain defects}\label{sec:main3}
In this section, we always assume the following setting.
Let $D\Subset (-\frac{1}{2},\frac{1}{2})^n$ be an open connected set  with $C^1$ boundary containing $0$.
Let $\Omega = \R^n\setminus \bigcup_{m\in \Z^n}(m+\ol D)$ and $\Omega_\ep = \ep \Omega$ for $\ep>0$.
Let $I \subsetneq \Z^n$ with $I \ne \emptyset$ be an index set that specifies the places where the holes are missing and
\[
W=\Omega \cup \bigcup_{m\in I} (m+\ol D)= \R^n\setminus \bigcup_{m\in \Z^n\setminus I}(m+\ol D), \qquad W_\ep = \ep W.
\]
Note that for $\Om_\ep$ to be connected, we need $n\geq 2$.
For $k\in \N$, $I_k = I \cap [-k,k]^n$ and condition \eqref{eqn:assumpI} reads
\begin{equation*}
    \frac{\left|I_k\right|}{k} = \omega_0\left(\frac{1}{k}\right) 
\end{equation*}
where $\omega_0$ is a given modulus of continuity.

\subsection{Proof of Theorem \ref{thm:main3}}
We always assume the settings of Theorem \ref{thm:main3} in this subsection.

Let $w^\varepsilon: \overline W_\varepsilon \times [0, \infty) \to \mathbb{R}$ be the viscosity solution to \eqref{eqn:w_epsilon}. The optimal control formula for $w^\varepsilon$ is as follows:
\begin{equation}
    \begin{aligned}
        &w^\varepsilon (x, t)\\
        =\ & \inf \left\lbrace \int_0^t L\left( \frac{\xi(s)}{\varepsilon} ,\dot{\xi}(s)\right)\, ds+g\left(\xi(0)\right): \xi\in \mathrm{AC}\left([0,t];\overline W_\varepsilon\right), \xi(t) = x\right\rbrace\\
        =\ &\inf \left\lbrace \varepsilon \int_0^{\frac{t}{\varepsilon}} L\left(\gamma(s),\dot{\gamma}(s)\right) \,ds+g\left(\varepsilon \gamma(0)\right): \gamma\in \mathrm{AC}\left(\left[0,\frac{t}{\varepsilon}\right];\overline{W} \right), \gamma\left(\frac{t}{\varepsilon}\right) = \frac{x}{\varepsilon}\right\rbrace.
    \end{aligned}
\end{equation}

\begin{lem}\label{lem:velocity bound-md}
Assume the settings of Theorem \ref{thm:main3}. 
Let $\varepsilon, t>0$ and $x\in \mathbb{R}^n$. Suppose that $\gamma:\left[0, \frac{t}{\varepsilon}\right] \to \overline{W}$ is a minimizing curve of $w^\varepsilon(x, t)$ in the sense that $\gamma$ is absolutely continuous, and
\begin{equation}
w^\varepsilon (x, t) = \varepsilon \int_0^{\frac{t}{\varepsilon}} L\left(\gamma(s),\dot{\gamma}(s)\right) \,ds+g\left(\varepsilon \gamma(0)\right)  
\end{equation}
with $\gamma\left(\frac{t}{\varepsilon}\right)=\frac{x}{\varepsilon}$. 
Then there exists a constant $M_0=M_0\left(n,\partial \Omega, H, \|Dg\|_{L^\infty(\mathbb{R}^n)}\right)>0$ such that 
\[
\left\|\dot{\gamma}\right\|_{L^\infty([0,\frac{t}{\varepsilon}])}  \leq M_0.
\]
\end{lem}
The proof of Lemma \ref{lem:velocity bound-md} is given in Appendix \ref{appendix}.
We can also define a metric function $\mdef$ similar to $m$ in Section \ref{sec:prelim}.

\begin{defn}\label{def:md}
Let $x, y\in \overline{W} $. Define 
\begin{equation}
    \mdef(t,x,y)= \inf \left\lbrace \int_0^{t} L\left(\gamma(s),\dot{\gamma}(s)\right)\, ds: \gamma\in \mathrm{AC}\left(\left[0,t\right];\overline{W} \right), \gamma(0)=x, \gamma\left(t\right)=y \right\rbrace.
\end{equation}
\end{defn}

With the metric $\mdef$ and Lemma \ref{lem:velocity bound-md}, we can rewrite the optimal control formula for $w^\varepsilon$ as

\begin{equation}
        w^\varepsilon (x, t)=\inf \left\lbrace \varepsilon \mdef\left(\frac{t}{\varepsilon}, \frac{y}{\varepsilon}, \frac{x}{\varepsilon}\right)+g\left(y\right): y \in \overline{W}_\varepsilon, \left|x-y\right|\leq M_0t\right\rbrace.
\end{equation}

\begin{lem}\label{lem:mstarmd}
    Let $t \geq 1$ and $x \in \overline{W}$. 
    Consider $y \in \overline{W}$ with $|y-x| \leq M_0t$. 
    Assume that there exists an optimal path $\eta:[0,t] \to \overline{W}$ for $\mdef\left(t, y, x\right)$, that is, 
     \[\mdef\left(t, y, x\right) = \int_0^t L\left(\eta(s), \dot{\eta}(s)\right),\] 
     with $\left\|\dot{\eta}\right\|_{L^\infty[0, t]} \leq M_0$.
    Then, there exists a constant $C=C\left(n, \partial \Omega, H, \|Dg\|_{L^\infty(\mathbb{R}^n)} \right)>0$ and a constant $\tilde{C} = \tilde{C}\left(n, \partial \Omega, H, \|Dg\|_{L^\infty(\mathbb{R}^n)} \right)>0$ such that
    \[
    m^\ast\left(t, y, x\right) \leq \mdef(t, y, x) + C \left|I_{\left\lceil M_0t+|x|\right\rceil}\right| + \tilde{C},
    \]
    where $\left|I_{\left\lceil M_0t+|x|\right\rceil}\right|$ is the cardinality of the set $I_{\left\lceil M_0t+|x|\right\rceil}$.
\end{lem}

\begin{proof}
    Let $t>1$ and $x\in \overline{W}$. 
    Let $\eta:[0,t] \to \overline{W}$ be an optimal path for $\mdef\left(t, y, x\right)$ with $\left\|\dot{\eta}\right\|_{L^\infty[0, t]} \leq M_0$.

    \smallskip
    
     In $[-\left\lceil M_0t+|x|\right\rceil,  \left\lceil M_0t+|x|\right\rceil]^n$, there are $|I_{\left\lceil M_0t+|x|\right\rceil}|$ unit cells that do not have holes in them. For $m_i \in I_{\left\lceil M_0t+|x|\right\rceil}$ and $i\in \left\{ 1, 2, \cdots, \left|I_{\left\lceil M_0t+|x|\right\rceil}\right|\right\}$, define
     \[
     \cH_i := m_i + \overline{D},
     \]
which are the holes in the defective unit cells. 
Note that $\cH_i\subset \overline{W}$ and hence $\eta$ can run into $\cH_i \subset \overline{W}$ for any $i\in \{1, 2, \cdots, \left|I_{\left\lceil M_0t+|x|\right\rceil}\right|\}$. 
Without loss of generality, we can assume $\cH_1$ is the first hole that $\eta$ enters and $\eta$ enters $\cH_i$ after the final exit from $\cH_{i-1}$. 
There are two cases to be considered.

\medskip

\noindent {\bf Case 1.} Suppose $x, y \in \overline{\Omega}$. 
Define $t_0=0$ and
\begin{equation*}
    \begin{aligned}
        s_i: &= \inf \left\{s: \eta(s)\in \partial \cH_i, s\geq t_{i-1}\right\},\\
        t_i: &= \sup \left\{s: \eta(s)\in \partial \cH_i, s\geq s_i\right\},
    \end{aligned}
\end{equation*}
for $i \in \{1, 2, \cdots, \left|I_{\left\lceil M_0t+|x|\right\rceil}\right|\}$. Intuitively, $s_i$ represents the initial time point at which $\eta$ enters $\cH_i$ after its final exit from $\cH_{i-1}$, while $t_i$ denotes the final time point at which $\eta$ exits $\cH_i$.

Consider $s \in [s_1, t_1]$ and define $y_1 := \eta(s_1)$ and $x_1:=\eta(t_1)$. We claim that there exists a path $\eta_1: [s_1, t_1] \to Y_{m_1} \cap \overline{\Omega}$ with $\eta_1(s_1) = y_1$, $\eta_1(t_1) = x_1$, and 
\begin{equation}\label{eq:rep-claim}
    \int_{s_1}^{t_1} L(\eta_1(s), \dot{\eta}_1(s))\,ds \leq C,
\end{equation}
for some constant $C=C\left(n, \partial \Omega, H, \|Dg\|_{L^\infty(\mathbb{R}^n)} \right) >0$.

\begin{itemize}
    \item[(a)] Suppose $t_1-s_1 \leq 1$. Since $\left\|\dot{\eta}\right\|_{L^\infty[0, t]} \leq M_0$, we have 
\[
\left|y_1-x_1\right| \leq \int_{s_1}^{t_1}\left|\dot{\eta}(s)\right| \leq M_0 (t_1 -s_1),
\]
which implies $\frac{|y_1-x_1|}{t_1-s_1} \leq M_0$. By a similar argument as in the proof of Proposition \ref{prop:mstarbound}, there exists a path $\eta_1:[s_1, t_1] \to Y_{m_1} \cap \overline{\Omega}$ such that $\eta_1(s_1) = y_1$, $\eta_1(t_1) = x_1$, and $\left\|\dot{\eta}_1\right\|_{L^\infty[s_1, t_1]} \leq C_b M_0$. Moreover, 

\begin{equation}\label{eqn:costtimel1}
 \begin{aligned}
 \int_{s_1}^{t_1} L(\eta_1(s), \dot{\eta}_1(s))\,ds &\leq \int_{s_1}^{t_1}\left(\frac{\left|\dot{\eta}_1(s)\right|^2}{2} + K_0\right)\,ds \\
 &\leq \left(\frac{C_b^2M_0^2}{2}+K_0\right)\left(t_1-s_1\right)\\
 & \leq \left(\frac{C_b^2M_0^2}{2}+K_0\right)
\end{aligned}
\end{equation}
since $t_1-s_1 \leq 1$.

\item[(b)] Suppose $t_1-s_1 > 1$. Consider the straight line segment connecting $y_1, x_1$ using time 1, that is,
\[
\gamma(s):= (x_1-y_1)s +y_1,
\]
for $s \in [0, 1]$, and $\left\|\dot{\gamma}\right\|_{L^\infty[0, 1]}=|x_1-y_1|$. By a similar argument as in the proof of Proposition \ref{prop:mstarbound}, we can revise this straight line segment into $\tilde{\gamma}: [0,1] \to Y_{m_1} \cap \overline{\Omega}$ with $\tilde{\gamma}(0)=y_1$, $\tilde{\gamma}(1)=x_1$, and $\|\dot{\gamma}\|_{L^\infty[0, 1]} \leq C_b |x_1-y_1|\leq C_b \sqrt{n}$. Define a new path $\tilde{\eta}_1:[s_1, t_1] \to Y_{m_1} \cap \overline{\Omega}$ by 
\[
\tilde{\eta}_1(s): =\left\{
\begin{aligned}
&\tilde{\gamma}(s-s_1), \quad \text{ for } s \in [s_1, s_1+1],\\    
& x_1, \qquad \qquad \text{ for } s \in [s_1+1, t_1].
\end{aligned}
\right.
\]
Now
\begin{equation}\label{eqn:costtimeg1}
 \begin{aligned}
 &\int_{s_1}^{t_1} L(\tilde{\eta}_1(s), \dot{\tilde{\eta}}_1(s))\,ds \\
 \leq\ &  \int_{s_1}^{s_1+1} L\left( \tilde{\gamma}(s-s_1), \dot{\tilde{\gamma}}(s-s_1)\right)\,ds + \int_{s_1+1}^{t_1} L(x_1, 0) \,ds\\
 \leq \ &\int_0^1 L\left(\tilde{\gamma}(s), \dot{\tilde{\gamma}}(s)\right)\,ds+0
 \leq \int_{0}^{1}\left(\frac{\left|\dot{\tilde{\gamma}}(s)\right|^2}{2} + K_0\right)\,ds \\
\leq \ & \frac{n C_b^2}{2}+K_0,
\end{aligned}
\end{equation}
where the second inequality follows from ({\rm A5}) and Lemma \ref{lem:A5}.
\end{itemize}

Combining \eqref{eqn:costtimel1} and \eqref{eqn:costtimeg1}, we prove the claim \eqref{eq:rep-claim}. 
In general, for $1 \leq i \leq  \left|I_{\left\lceil M_0t+|x|\right\rceil}\right|$, define $y_i:= \eta(s_i)$ and $x_i:= \eta(t_i)$. Then, there exists a path $\eta_i: [s_i, t_i] \to Y_{m_i} \cap \overline{\Omega}$ with $\eta_i(s_i) = y_i$, $\eta_i(t_i) = x_i$, and 
\begin{equation}\label{eqn:bdihole}
    \int_{s_i}^{t_i} L(\eta_i(s), \dot{\eta}_i(s))\,ds \leq C,
\end{equation}
for some constant $C=C\left(n, \partial \Omega, H, \|Dg\|_{L^\infty(\mathbb{R}^n)} \right) >0$.
Note that $x_i$ and $y_i$ may not exist, but this is perfectly acceptable as we do not need to consider them in that case.
Now, define a new path $\xi: [0, t] \to \overline{\Omega}$ defined by
\begin{equation}\label{eqn:ihrevise}
\xi(s):=\left\{\begin{aligned}
& \eta(s), \quad \,\text{ for } s \in [0,t]\setminus \left(\bigcup_{i=1}^{\left|I_{\left\lceil M_0t+|x|\right\rceil}\right|} [s_i, t_i]\right),\\
&\eta_i(s), \quad \text{ for } s \in [s_i, t_i], i\in \left\{ 1, 2, \cdots, \left|I_{\left\lceil M_0t+|x|\right\rceil}\right|\right\},
\end{aligned}
\right.
\end{equation}
and this is an admissible path for $m(t,y,x)$. Therefore,

\begin{equation}
\begin{aligned}
    &m(t, y, x)\\
    \leq\ &\int_0^t L\left(\xi(s), \dot\xi(s)\right)\, ds \\
    \leq\ & \int_{[0,t]\setminus \left(\bigcup_{i=1}^{\left|I_{\left\lceil M_0t+|x|\right\rceil}\right|} [s_i, t_i]\right)} L(\eta(s), \dot{\eta}(s))\,ds+
    \sum_{i=1}^{\left|I_{\left\lceil M_0t+|x|\right\rceil}\right|}\int_{[s_i, t_i]} L(\eta_i(s), \dot{\eta_i}(s))\,ds\\
    \leq \ &\int_0^t L(\eta(s), \dot{\eta}(s))\,ds +\sum_{i=1}^{\left|I_{\left\lceil M_0t+|x|\right\rceil}\right|}\int_{[s_i, t_i]} \left( L(\eta_i(s), \dot{\eta_i}(s))-L(\eta(s), \dot{\eta}(s))\right)\,ds\\
    \leq \ &\int_0^t L(\eta(s), \dot{\eta}(s))\, ds+\sum_{i=1}^{\left|I_{\left\lceil M_0t+|x|\right\rceil}\right|}\int_{[s_i, t_i]} L(\eta_i(s), \dot{\eta_i}(s))\,ds\\
    \leq \ & \mdef\left(t, y, x\right)+C\left|I_{\left\lceil M_0t+|x|\right\rceil}\right|,
\end{aligned}
\end{equation}
where the second to last inequality follows from ({\rm A5}) and Lemma \ref{lem:A5}, and $C=C\left(n, \partial \Omega, H, \|Dg\|_{L^\infty(\mathbb{R}^n)} \right)$ comes from \eqref{eqn:bdihole}. By Proposition \ref{prp:mstarm}, we have 
\[
m^\ast(t, y, x) \leq \tilde{C} + \mdef\left(t, y, x\right)+C\left|I_{\left\lceil M_0t+|x|\right\rceil}\right|
\]
for some constant $\tilde{C}=\tilde{C}\left(n, \partial \Omega, H, \|Dg\|_{L^\infty(\mathbb{R}^n)} \right)>0$.

\medskip

\noindent {\bf Case 2.} Suppose either $x$ or $y$ is in $\cH_i = m_i+ \overline{D}$ for some $i\in \left\{ 1, 2, \cdots, \left|I_{\left\lceil M_0t+|x|\right\rceil}\right|\right\}$, that is, a hole in one of the defective unit cells. Without loss of generality, we assume $y \in \cH_1=  m_1+ \overline{D}$ and $x \in \overline{\Omega}$. One can prove for the case where $x$ is in a hole of a defective unit cell and $y$ is not and the case both $x, y$ are in some holes similarly. Define
    \[
    t_1 : = \sup \left\{s: \eta(s) \in \partial \cH_1, s\geq 0 \right\}, 
    \]
and 
\begin{equation*}
    \begin{aligned}
        s_i: &= \inf \left\{s: \eta(s)\in \partial \cH_i, s\geq t_{i-1}\right\},\\
        t_i: &= \sup \left\{s: \eta(s)\in \partial \cH_i, s\geq s_i\right\},
    \end{aligned}
\end{equation*}
for $i \in \left\{2, \cdots, \left|I_{\left\lceil M_0t+|x|\right\rceil}\right|\right\}$ as in Case 1. 
Similar to \eqref{eqn:ihrevise} in Case 1, we can find a revised path $\xi: [0, t] \to \overline{\Omega}$ defined by
\begin{equation}
\xi(s):=\left\{\begin{aligned}
& \eta(s), \quad \,\text{ for } s \in [0,t]\setminus \left(\bigcup_{i=2}^{\left|I_{\left\lceil M_0t+|x|\right\rceil}\right|} [s_i, t_i]\right),\\
&\eta_i(s), \quad \text{ for } s \in [s_i, t_i], i\in \left\{2, \cdots, \left|I_{\left\lceil M_0t+|x|\right\rceil}\right|\right\},
\end{aligned}
\right.
\end{equation}
for some paths $\eta_i: [s_i, t_i] \to Y_{m_i} \cap \overline{\Omega}$ with $\eta_i(s_i) = y_i$, $\eta_i(t_i) = x_i$, and 
\begin{equation}
    \int_{s_i}^{t_i} L(\eta_i(s), \dot{\eta}_i(s))\,ds \leq C,
\end{equation}
with $i \in \left\{2, \cdots, \left|I_{\left\lceil M_0t+|x|\right\rceil}\right|\right\}$. Moreover,
\begin{equation} \label{eqn:xibound}
\left\|\dot{\xi}\right\|_{L^\infty[0, t]} \leq \tilde{C}: = \max\left\{ C_bM_0, C_b\sqrt{n}\right\}.
\end{equation}

By the proof in Case 1, we know
\begin{equation}\label{eqn:ximd}
\int_0^t L\left(\xi(s), \xi(s)\right) \,ds\leq \mdef(t, y, x) +C \left|I_{\left\lceil M_0t+|x|\right\rceil}\right|.
\end{equation}
\begin{itemize}
    \item[(a)]Suppose $t_1<\frac{1}{8}$. Let $\tilde{x}, \tilde{y} \in \partial \Omega$ with $\tilde{x}-x \in  Y, \tilde{y}-y\in Y$. From the proof of Proposition \ref{prop:mstarbound}, we know there exists a path $\alpha: \left[0, \frac{1}{8}\right] \to \overline{\Omega}$ such that $\alpha(0)=\tilde{y}$, $\alpha\left(\frac{1}{8}\right)=\xi\left(t_1\right)=\eta\left(t_1\right)$, and a path $\beta:\left[0, \frac{1}{8}\right] \to \overline{\Omega}$ such that $\beta(0)=x$, $\beta(\frac{1}{8})=\tilde{x}$. Now consider a new path $\eta:[0, t] \to \overline{\Omega}$ defined by
    \begin{equation}
    \tilde{\eta}(s) :=  \left\{\begin{aligned} &\alpha(s), \quad \qquad \qquad \qquad \qquad \qquad \text{if } s \in \left[0, \frac{1}{8}\right],\\
  &\xi\left(2\left(s-\frac{1}{8}\right)+t_1\right), \quad \qquad \quad \text{if } s \in \left[\frac{1}{8}, \frac{3}{8}\right],\\
  &\xi \left(s+\frac{1}{8}+ t_1\right), \, \, \qquad \qquad \qquad \text{if }  s \in \left[\frac{3}{8}, t-t_1-\frac{1}{8}\right],\\
  & x, \quad \qquad \qquad \qquad \,\, \quad \qquad \qquad \text{if } s \in \left[t-t_1-\frac{1}{8}, t-\frac{1}{8}\right],\\
&\beta \left(s-t+\frac{1}{8}\right),   \qquad \qquad   \qquad \, \, \, \text{if }  s \in \left[t-\frac{1}{8}, t\right],\\
    \end{aligned}
    \right.
    \end{equation}
which is an admissible path for $m^\ast\left(t, y, x\right)$. Therefore, 
\begin{equation}
\begin{aligned}
    &m^\ast\left(t, y, x\right)\\ 
    \leq\ &\int_0^t L\left(\tilde{\eta}(s), \dot{\tilde{\eta}}(s)\right)\,ds\\
    \leq\ & \int_0^\frac{1}{8} L\left(\alpha(s), \dot{\alpha}(s)\right)\,ds + \int_\frac{1}{8}^\frac{3}{8}L \left(\xi\left(2\left(s-\frac{1}{8}\right)+t_1\right), 2\dot{\xi} \left(2\left(s-\frac{1}{8}\right) + t_1\right)\right)\,ds\\
    & \ + \int_\frac{3}{8}^{t-t_1-\frac{1}{8}} L\left(\xi \left(s+\frac{1}{8}+ t_1\right), \dot{\xi} \left(s+\frac{1}{8}+ t_1\right)\right) + \int_{t-t_1-\frac{1}{8}}^{t-\frac{1}{8}} L(x, 0) \,ds\\ 
    &\ +\int_{t-\frac{1}{8}}^t L\left(\beta\left(s-t +\frac{1}{8}\right), \dot{\beta}\left(s- t +\frac{1}{8}\right)\right)\,ds\\
    \leq \ & \left(32C_b^2 n +K_0\right) \frac{1}{8} + \frac{1}{2} \int_{t_1}^{\frac{1}{2}+t_1} L \left(\xi(s), 2\dot{\xi}(s)\right)\,ds+\int_{\frac{1}{2}+t_1}^tL\left(\xi(s), \dot{\xi}(s)\right)\,ds\\
    &\ + 0 + \left(32C_b^2 n +K_0\right) \frac{1}{8}\\
    \leq \ & 8C_b^2n +\frac{K_0}{4} + \frac{1}{4} \left(\frac{4\tilde{C}^2}{2}+ K_0\right) + \int_0^tL\left(\xi(s), \dot{\xi}(s)\right)\,ds\\
    \leq \ & 8C_b^2n + \frac{\tilde{C}^2}{2}+\frac{K_0}{2} +  \mdef(t, y, x) +C \left|I_{\left\lceil M_0t+|x|\right\rceil}\right|,
\end{aligned}
\end{equation}
where the third inequality follows from \eqref{prop:mstarbound} and \eqref{eqn:xibound} and the last inequality follows from \eqref{eqn:ximd}.

    \item[(b)]Suppose $t_1 \geq \frac{1}{8}$. Let $\tilde{x}, \tilde{y} \in \partial \Omega$ with $\tilde{x}-x \in Y, \tilde{y}-y\in Y$. By the proof of Proposition \ref{prop:mstarbound}, there exists a path $\alpha: [0, \frac{1}{16}] \to \overline{\Omega}$ such that $\alpha(0)=\tilde y$, $\alpha\left(\frac{1}{16}\right)=\eta \left(t_1\right)=\xi \left(t_1\right)$, and $\left\| \dot{\alpha} \right\|_{L^\infty\left[0, \frac{1}{16}\right]} \leq 16C_b\sqrt{n}$. And there exists a path $\beta: \left[0, \frac{1}{16}\right] \to \overline{\Omega}$ with $\beta(0)=x$, $\beta\left(\frac{1}{16}\right)=\tilde{x}$, and $\left\| \dot{\beta} \right\|_{L^\infty\left[0, \frac{1}{16}\right]} \leq 16C_b\sqrt{n}$. Now consider a new path $\tilde{\eta}: [0, t] \to \overline{\Omega}$ defined by
   \begin{equation}
    \tilde{\eta}(s) :=  \left\{\begin{aligned} &\alpha(s), \qquad \qquad \qquad \qquad \quad \, \, \, \text{if } s \in \left[0, \frac{1}{16}\right],\\
  &\xi\left(s-\frac{1}{16}+t_1\right), \, \quad \quad \qquad  \text{if } s \in \left[\frac{1}{16}, t-t_1+\frac{1}{16}\right],\\
  & x, \qquad \quad \qquad \qquad \qquad \qquad \text{if } s \in \left[t-t_1+\frac{1}{16}, t-\frac{1}{16}\right], \\
  &\beta \left(s-t+\frac{1}{16}\right),  \, \qquad \, \, \qquad  \text{if } s \in \left[t-\frac{1}{16}, t\right],
    \end{aligned}
    \right.
    \end{equation}    
which is an admissible path for $m^\ast(t, y, x)$. Therefore, 
\[
\begin{aligned}
    &m^\ast(t, y, x) \\
    \leq\ &\int_0^\frac{1}{16} L\left(\alpha(s), \dot{\alpha}(s)\right)\,ds + \int_\frac{1}{16}^{t-t_1+\frac{1}{16}} L\left(\xi\left(s-\frac{1}{16}+t_1\right), \dot{\xi} \left(s-\frac{1}{16}+t_1\right)\right)\,ds \\
    &\ + \int_{t-t_1+\frac{1}{16}}^{t-\frac{1}{16}} L(x, 0)\,ds+ \int_{t-\frac{1}{16}}^tL\left(\beta \left(s-t+ \frac{1}{16}\right), \dot{\beta}\left(s-t+\frac{1}{16}\right) \right)\,ds\\
    \leq \ & \frac{1}{16} \left(\frac{16^2C_b^2n}{2}+K_0\right)+ \int_{t_1}^t L \left(\xi(s), \dot{\xi}(s)\right)\,ds +0 +\frac{1}{16} \left(\frac{16^2C_b^2n}{2}+K_0\right)\\
    \leq\  & 8 C_b^2n +\frac{K_0}{8}+ \mdef(t, y, x) +C \left|I_{\left\lceil M_0t+|x|\right\rceil}\right|,
\end{aligned}
\]
where the last inequality follows from \eqref{eqn:ximd}.
\end{itemize}
\end{proof}

We give the proof of Theorem \ref{thm:main3}.
\begin{proof}[Proof of Theorem \ref{thm:main3}]
 Fix $x\in \overline W_\varepsilon$. If $0 < t < \varepsilon$, by the comparison principle, we know
\[
\left|w^\varepsilon(x,t)-g(x)\right|\leq Ct,
\]
and 
\[
\left|u(x,t)-g(x)\right|\leq Ct,
\]
for some constant $C>0$ that only depends on $H$.
Therefore,
\[
\left|w^\varepsilon(x, t)-u(x, t)\right| \leq Ct \leq C \varepsilon,
\]
for some constant $C>0$ that only depends on $H$. 

If $t \geq \varepsilon$, suppose 
\[
w^\varepsilon(x, t) = \varepsilon \mdef \left(\frac{t}{\varepsilon}, \frac{\tilde{y}}{\varepsilon}, \frac{x}{\varepsilon} \right) +g(\tilde{y})
\]
for some $\tilde{y} \in \overline W_\varepsilon$ with $|\tilde{y}-x| \leq M_0 t$. Let $\gamma: \left[0, \frac{t}{\varepsilon}\right] \to \overline{W}$ be an optimal path for $\mdef\left(\frac{t}{\varepsilon}, \frac{\tilde{y}}{\varepsilon}, \frac{x}{\varepsilon}\right)$. Then $\left\|\dot{\gamma}\right\|_{L^\infty([0,\frac{t}{\varepsilon}])}  \leq M_0$ by Lemma \ref{lem:velocity bound-md}. From Lemma \ref{lem:mstarmd}, we have
\begin{equation}
\begin{aligned}
        \varepsilon m^\ast \left(\frac{t}{\varepsilon}, \frac{\tilde{y}}{\varepsilon}, \frac{x}{\varepsilon} \right) &\leq  \varepsilon \mdef \left(\frac{t}{\varepsilon}, \frac{\tilde{y}}{\varepsilon}, \frac{x}{\varepsilon} \right) +  \varepsilon C \left|I_{\left\lceil\frac{M_0t+|x|}{\varepsilon} \right\rceil}\right| + \varepsilon \tilde{C}\\
    &\leq  \varepsilon \mdef \left(\frac{t}{\varepsilon}, \frac{\tilde{y}}{\varepsilon}, \frac{x}{\varepsilon} \right) + \varepsilon C \frac{ \left|I_{\left\lceil\frac{M_0t+|x|}{\varepsilon} \right\rceil}\right| }{\left\lceil\frac{M_0t+|x|}{\varepsilon} \right\rceil} \left\lceil\frac{M_0t+|x|}{\varepsilon} \right\rceil + \varepsilon \tilde{C}\\
    &\leq  \varepsilon \mdef \left(\frac{t}{\varepsilon}, \frac{\tilde{y}}{\varepsilon}, \frac{x}{\varepsilon} \right) + C\omega_0 \left( \frac{1}{\left\lceil\frac{M_0t+|x|}{\varepsilon} \right\rceil}\right) \left(M_0t+|x| +\varepsilon \right) + \varepsilon \tilde{C}\\
    &\leq  \varepsilon \mdef \left(\frac{t}{\varepsilon}, \frac{\tilde{y}}{\varepsilon}, \frac{x}{\varepsilon} \right) + C\left(M_0t+|x| +1 \right) \omega_0 \left(\frac{\varepsilon}{ M_0t +|x|}\right) +\varepsilon \tilde{C}.
\end{aligned}
\end{equation}
Therefore, 
\begin{equation}
    \begin{aligned}
    u\left(x,t\right) &= \inf \left\{ \overline{m}^\ast \left(t, y, x \right) +g(y): |x -y | \leq M{_0} t, y \in \mathbb{R}^n \right\}\\
    &\leq \inf \left\{ \varepsilon m^\ast \left(\frac{t}{\varepsilon}, \frac{y}{\varepsilon}, \frac{x}{\varepsilon} \right) +g(y)  + C \varepsilon: |x -y | \leq M{_0} t, y \in \mathbb{R}^n \right\}\\
    &\leq \inf \left\{ \varepsilon m^\ast \left(\frac{t}{\varepsilon}, \frac{y}{\varepsilon}, \frac{x}{\varepsilon} \right) +g(y) + C \varepsilon: |x -y | \leq M{_0} t, y \in \overline{W}_\varepsilon\right\}\\
    & \leq \varepsilon m^\ast \left(\frac{t}{\varepsilon}, \frac{\tilde{y}}{\varepsilon}, \frac{x}{\varepsilon} \right) +g(\tilde{y}) + C \varepsilon\\
    &\leq  \varepsilon \mdef \left(\frac{t}{\varepsilon}, \frac{\tilde{y}}{\varepsilon}, \frac{x}{\varepsilon} \right) +g(\tilde{y}) +C\left(M_0t+|x| +1 \right) \omega_0 \left(\frac{\varepsilon}{ M_0t +|x|}\right) +\varepsilon \tilde{C}\\
     &= w^\varepsilon \left(x, t\right)+C\left(M_0t+|x| +1 \right) \omega_0 \left(\frac{\varepsilon}{ M_0t +|x|}\right) +\varepsilon \tilde{C}.
    \end{aligned}
\end{equation}
On the other hand, by the optimal control formulas for $u^\varepsilon$ and $w^\varepsilon$ and Theorem \ref{thm:main1}, we have
\[
w^\varepsilon(x, t) \leq  u^\varepsilon (x, t) \leq u(x, t) + C \varepsilon,
\]
and
\[
u(x, t) - C\varepsilon \leq u^\varepsilon(x, t).
\]
Therefore, \eqref{eqn:main3} holds and the proof is complete.
\end{proof}

\subsection{The optimality of the bound in Theorem \ref{thm:main3} and some nonconvergent results}
We now show the optimality of the bound in Theorem \ref{thm:main3}.

\begin{lem}\label{lem:W-op-1}
Assume that
\[
H(y,p)= \frac{|p|^2}{2} \qquad \text{ for } (y,p)\in\Om \times \R^n.
\]
Then, 
\begin{equation*}
    \ol H(p) \leq \frac{|p|^2}{2} \qquad \text{ for all $p\in \R^n$},
\end{equation*}
and
\begin{equation*}
        \ol H(-e_1) = \frac{1}{2}.
\end{equation*}
Here, $\{e_1,e_2,\ldots,e_n\}$ is the canonical basis of $\R^n$.
\end{lem}

\begin{proof}
    As $H(y,p)=\frac{|p|^2}{2}$ for $(y,p)\in\ol \Om \times \R^n$, by the inf-sup representation formula, we see that 
    \[
    \ol H(p)=\ol H_\Om(p)=\inf_{\varphi\in \Lip(\T^n)} \esssup_{y\in \Om}\frac{|p+D\varphi(y)|^2}{2}.
    \]
    By taking $\varphi=0$, we imply 
    \begin{equation}\label{eq:lemW20}
        \ol H(p) \leq \frac{|p|^2}{2} \qquad \text{ for all $p\in \R^n$.}
    \end{equation}
    We now show that
    \begin{equation}\label{eq:lemW21}
        \ol H(-e_1) = \frac{1}{2}.
    \end{equation}
    Thanks to \eqref{eq:lemW20}, we only need to prove $\ol H(-e_1) \geq \frac{1}{2}$.
    Indeed, fix a test function $\varphi \in \Lip(\T^n)$.
    For $y\in \R^n$, write $y=(y_1,y') \in \R \times \R^{n-1}$.
    There exists $\delta>0$ such that
    \[
    E=[0,1]\times \left[\frac{1}{2}-\delta,\frac{1}{2}+\delta\right] \times \cdots \times \left[\frac{1}{2}-\delta,\frac{1}{2}+\delta\right] =[0,1] \times E' \subset \Omega.
    \]
    As $\varphi$ is $\Z^n$-periodic, $\varphi(0,y')=\varphi(1,y')$ for $y'\in E'$.
    We have
    \begin{align*}
            \int_E \frac{|-e_1+D\varphi(y)|^2}{2}\,dy &\geq \int_E \frac{|1-\varphi_{y_1}(y)|^2}{2}\,dy\\
            &\geq \frac{1}{2}\int_{E'} \left(\int_0^1 \left(1-\varphi_{y_1}\right)\,dy_1  \right)^2\,dy'
            =\frac{1}{2}|E'| =    \frac{1}{2}|E|. 
\end{align*}
Hence, $\esssup_{y\in \Om} \frac{|-e_1+D\varphi(y)|^2}{2} \geq \frac{1}{2}$.
The proof is complete.
\end{proof}

\begin{rem}
    Assume the settings in Lemma \ref{lem:W-op-1}.
    We proved that $\ol H(p)\leq \frac{|p|^2}{2}$ for all $p\in \R^n$.
    Following the same argument as that for \eqref{eq:lemW21}, we have further that
    \[
    \ol H(se_i) = \frac{s^2}{2} \qquad \text{ for } 1\leq i \leq n,\, s\in \R.
    \]
    However, we do not yet have the explicit formula for $\ol H(p)$ for general values of $p$. 
    This is a complicated task as we need to consider the geometry of $\Om$.
    See \cite{Con} for some results along this line.
\end{rem}

\begin{lem}\label{lem:W-op-2}
    Assume $(-\frac{9}{20},\frac{9}{20})^n \Subset D\Subset (-\frac{1}{2},\frac{1}{2})^n$ and
    \[
    H(y,p)=\frac{a(y)|p|^2}{2}\qquad \text{ for } (y,p)\in  \R^n\times \R^n.
    \]
    Here, $a \in C(\R^n)$ is a given $\Z^n$-periodic function such that $a=1$ on $\ol \Om$, $a>1$ in $\R^n\setminus \ol \Om$, and
    \begin{equation}\label{eq:a-ass}
    \int_{-\frac{1}{2}}^{\frac{1}{2}} \frac{1}{a(\frac{e_2}{4}+se_1)}\,ds < \frac{1}{4}.
    \end{equation}
    Assume further that 
    \[
    I=\{m^2 e_1\,:\, m\in \N \}.
    \]
    Then, the modulus of continuity is $\om_0(\ep)=\ep^{\frac{1}{2}}$.
    
    Let $g(x)=-x_1$ for $x\in \R^n$.
    For $\ep>0$, let $w^\ep$ be the unique viscosity solution to \eqref{eqn:w_epsilon}.
    Let $u$ be the unique viscosity solution to \eqref{eqn:PDE_limit}.
    Then, there exists $\delta=\delta(a)>0$ such that, for $\ep \in (0,\frac{1}{4})$,
    \begin{equation}\label{eq:proof-op-w}
    w^\ep\left(\frac{\ep e_2}{2},1\right) \leq u\left(\frac{\ep e_2}{2},1\right) -\delta \ep^{\frac{1}{2}}+\ep =-\frac{1}{2} -\delta \ep^{\frac{1}{2}} +\ep.
    \end{equation}
\end{lem}

The inequality \eqref{eq:proof-op-w} confirms the optimality of the bound in Theorem \ref{thm:main3}.
The assumption that the hole $D$ is rather big, nearly filling the whole cell $(-\frac{1}{2},\frac{1}{2})^n$ is needed in order for \eqref{eq:a-ass} to satisfy.

\begin{proof}
    As $H(y,p)=\frac{|p|^2}{2}$ for $(y,p)\in\ol \Om \times \R^n$, by Lemma \ref{lem:W-op-1},
    \begin{equation*}
        \ol H(-e_1) = \frac{1}{2}.
    \end{equation*}
    We then use the assumption $g(x)=-x_1$ for $x\in \R^n$ to yield
    \[
    u(x,t) = -x_1 - t \ol H(-e_1) = -x_1 - \frac{t}{2}\qquad \text{ for } (x,t)\in \R^n \times [0,\infty).
    \]
    In particular, $u\left(\frac{\ep e_2}{2},1\right)=-\frac{1}{2}$.    
\smallskip

    By the hypothesis, we have
    \[
    L(y,v)=\frac{|v|^2}{2a(y)}\qquad \text{ for } (y,p)\in  \R^n\times \R^n.
    \]
    Denote by $\xi:[0,1] \to \R^n$ by
    \[
    \xi(s)=
    \begin{cases}
        \frac{-e_1+e_2}{2} - s e_2 \qquad &\text{ for } 0\leq s \leq \frac{1}{4},\\
        \frac{-e_1}{2}+\frac{e_2}{4}+2(s-\frac{1}{4})e_1\qquad &\text{ for } \frac{1}{4}\leq s \leq \frac{3}{4},\\
        \frac{e_1}{2}+\frac{e_2}{4}+(s-\frac{3}{4})e_2 \qquad &\text{ for } \frac{3}{4}\leq s \leq 1.
    \end{cases}
    \]
    In light of \eqref{eq:a-ass}, we have that
    \[
    \int_0^1 L(\xi(s),\dot \xi(s))\,ds=\int_0^1 \frac{|\dot \xi(s)|^2}{2a(\xi(s))}\,ds <\frac{1}{2}.
    \]
    Construct $\gam:[0,\infty) \to \ol W$ such that
    \[
    \gam(s)=
    \begin{cases}
        \frac{e_2}{2}+se_1 \qquad &\text{ for }s \notin [m^2-\frac{1}{2},m^2+\frac{1}{2}],\\
        m^2 e_1 + \xi(s-m^2-\frac{1}{2})\qquad &\text{ for }s \in [m^2-\frac{1}{2},m^2+\frac{1}{2}] \text{ with } m\in \N.
    \end{cases}
    \]
    Here, $\gam$ travels on the straight ray $\{\frac{e_2}{2}+se_1:s\geq 0\}$ when it sees the normal cell with a hole. And when it sees a cell with a missing hole, it detours to reduce the running cost.
    Each detour is a shift of $\xi$ by an integer vector.
    
    By the optimal control formula for $w^\ep$,
    \begin{align*}
        w^\ep\left(\frac{\ep e_2}{2},1\right) &\leq g\left(\ep\gam\left(\frac{1}{\ep}\right)\right) + \ep \int_0^{\frac{1}{\ep}} L(\gam(s),-\dot \gam(s))\,ds\\
        &\leq -1+\ep + \ep \int_0^{\frac{1}{\ep}}\frac{1}{2}\,ds + \ep \sum_{m=1}^{\lfloor \ep^{-\frac{1}{2}} \rfloor} \int_{m^2-\frac{1}{2}}^{m^2+\frac{1}{2}} \left( L(\gam(s),-\dot \gam(s))-\frac{1}{2}\right)\,ds\\
        &\leq-\frac{1}{2}+\ep-\frac{1}{2}\left(\ep^{\frac{1}{2}} -\ep \right) \int_{0}^{1} \left( 1-\frac{|\dot \xi(s)|^2}{a(\xi(s))}\right)\,ds.
    \end{align*}
    The proof is complete by setting
    \[
    \delta= \frac{1}{4} \int_{-\frac{1}{2}}^{\frac{1}{2}} \left( 1-\frac{|\dot \xi(s)|^2}{a(\xi(s))}\right)\,ds>0.
    \]
\end{proof}

Next, we demonstrate if the size of $I$ is bigger than that in Theorem \ref{thm:main3}, then $w^\ep$ does not converge to $u$ as $\ep \to 0$.

\begin{lem}\label{lem:W1}
    Assume 
   \[
    H(y,p)=\frac{a(y)|p|^2}{2}\qquad \text{ for } (y,p)\in  \R^n\times \R^n.
    \]
    Here, $a \in C(\R^n)$ is a given $\Z^n$-periodic function such that $a=1$ on $\ol \Om$, $a>1$ in $\R^n\setminus \ol \Om$.
    Assume further that 
    \[
    I=\{m e_1\,:\, m\in \N \cup \{0\}\}.
    \]
    
    Let $g(x)=-x_1$ for $x\in \R^n$.
    For $\ep>0$, let $w^\ep$ be the unique viscosity solution to \eqref{eqn:w_epsilon}.
    Let $u$ be the unique viscosity solution to \eqref{eqn:PDE_limit}.
    Then,
    \[
    \limsup_{\ep \to 0} w^\ep(0,1) < u(0,1)=-\frac{1}{2}.
    \]
\end{lem}

\begin{proof}
    As above, we have $\ol H(-e_1) = \frac{1}{2}$ and
    \[
    u(x,t) = -x_1 - t \ol H(-e_1) = -x_1 - \frac{t}{2}\qquad \text{ for } (x,t)\in \R^n \times [0,\infty).
    \]
    In particular, $u(0,1)=-\frac{1}{2}$.   
    
    By the assumptions, we have
 \[
    L(y,v)=\frac{|v|^2}{2a(y)}\qquad \text{ for } (y,p)\in  \R^n\times \R^n.
\]
     By the optimal control formula for $w^\ep$,
    \begin{align*}
        w^\ep(0,1) &\leq g(e_1) + \ep \int_0^{\frac{1}{\ep}} L(se_1,-e_1)\,ds\\
        &\leq -1 + \ep \int_0^{\frac{1}{\ep}}\frac{1}{2}\,ds + \ep \sum_{m=1}^{\lfloor \ep^{-1} \rfloor} \int_{m-\frac{1}{2}}^{m+\frac{1}{2}} \left( \frac{1}{2 a(s e_1)}-\frac{1}{2}\right)\,ds\\
        &\leq-\frac{1}{2}-\frac{1}{2}\left(1 -\ep \right) \int_{-\frac{1}{2}}^{\frac{1}{2}} \left( 1-\frac{1}{a(s e_1)}\right)\,ds\\
        &\leq -\frac{1}{2} -\theta,
    \end{align*}
    for $\ep \in (0,\frac{1}{2})$, where
    \[
    \theta= \frac{1}{4} \int_{-\frac{1}{2}}^{\frac{1}{2}} \left( 1-\frac{1}{a(s e_1)}\right)\,ds>0.
    \]
Here, $\theta>0$ as, for $s\in [-\frac{1}{2},\frac{1}{2}]$, $a(se_1)>1$ for $se_1\in D$ and $a(se_1)=1$ for $se_1 \notin D$.
\end{proof}

    A key point used in the above proof is the ray $\gam(s)=s e_1$ for $s\in [0,\infty)$ stays on $\ol W$ and is admissible for the optimal control formula of $w^\ep$.
    This ray does not stay on $\ol \Om$ and is not admissible for the optimal control formula of $u^\ep$.
\begin{rem}
    Assume the settings in Lemma \ref{lem:W1}.
    We have shown in Lemma \ref{lem:W1} that $w^\ep$ does not converge to $u$ in general because of the impact of the missing holes.
    We have not shown that $w^\ep$ is convergent as $\ep \to 0$ and have not quantified the convergence rate in this scenario.
\end{rem}

Finally, we show that (A5) is needed to obtain the convergence result in Theorem \ref{thm:main3}.

\begin{lem}\label{lem:no-A5}
    Assume
    \[
    H(y,p)=\frac{|p|^2}{2} + V(y) \qquad \text{ for } (y,p)\in \R^n\times \R^n.
    \]
    Here, $V\in C(\R^n)$ is a given potential energy which is $\Z^n$-periodic satisfying $V=0$ on $\ol \Om$ and $V>0$ in $\R^n \setminus \ol \Om$. 
    Assume that $I=\{0\}$, that is, only one hole at the origin is missing.

    Let $g=0$. 
    For $\ep>0$, let $w^\ep$ be the unique viscosity solution to \eqref{eqn:w_epsilon}.
    Let $u$ be the unique viscosity solution to \eqref{eqn:PDE_limit}.
    Then,
    \[
    \limsup_{\ep \to 0} w^\ep(0,1) \leq -V(0) < u(0,1)=0.
    \]
\end{lem}

\begin{proof}
     As $H(y,p)=\frac{|p|^2}{2}$ for $(y,p)\in\ol \Om \times \R^n$, by Lemma \ref{lem:W-op-1}, $\ol H(0)=0$.
    We then use the assumption $g=0$ to yield $u=0$.
    In particular, $u(0,1)=0$.

    In our classical mechanic setting,
        \[
    L(y,v)=\frac{|v|^2}{2} - V(y) \qquad \text{ for } (y,v)\in \R^n\times \R^n,
    \]
    and $V(0)>0$ by the hypothesis.
    Let $\xi:[0,\infty) \to \ol W$ be such that $\xi(s)=0$ for all $s\geq 0$.
    By the optimal control formula for $w^\ep$, we imply
    \[
    w^\ep(0,1) \leq g(0) + \ep \int_0^{\frac{1}{\ep}} L(\xi(s),\dot \xi(s))\,ds = -V(0)<0,
    \]
    which concludes the proof.
\end{proof}

\begin{rem}
The idea of the above proof is simple.
The missing hole $D$ is an attractor of minimizing curves as it is cheaper to stay inside $D$.
Specifically, the Aubry set in this case is
\[
\cA = \{ y\in D\,:\, V(y)=\max V \}.
\]
It is clear that, for $y_0\in \cA$,
\[
L(y_0,0) = - V(y_0) = \min L.
\]
By using the same idea as the proof above, we have further that, for $t>0$,
\[
\lim_{\ep \to 0} w^\ep(0,t) = - (\max V) t.
\]
\end{rem}

\appendix

\section{Proof of some auxiliary results} \label{appendix}

We first show that the metric on $\partial \Omega$ is comparable with the Euclidean metric. Recall that the unit-scale perforated domain $\Omega$ is assumed to be connected, and $\partial \Omega$ is $C^1$. The set $\Omext := \R^n\setminus \ol\Omega$ may have disconnected components. Similarly, $\partial \Omega$ may have disconnected components, and the boundary of each connected component of $\Omext$ is contained in $\partial \Omega$. For a connected component $D$ of $\Omext$, we say $D$ is \emph{extended} if $D$ is an unbounded set and call $D$ \emph{localized} if otherwise. Since we assume that $\Omega$ is connected, for $n=2$ each $D$ must be localized, and we can find a finite number of $D_i \subset 2Y$, $1\le i\le N_D$, so that 
\[
\mathcal{O} = \bigcup_{i=1}^{N_D}\bigcup_{m\in \Z^n} \left(m+D_i\right)
\]
and the union over $i$ is a disjoint one. For $n\ge 3$, $\Omext$ can have both extended and localized components, but the same structure above for $\Omext$ still holds, except now the sets $D_i$, $1\le i\le N_D$, can be divided into localized and extended ones. Note that for $i\ne j$, $\ol D_i \cap \ol D_j = \emptyset$, and $\partial D_i$ is a connected component of $\partial \Omega$. 
\begin{lem}\label{lem:pOmetric} 
Suppose $p$ and $q$ are two points in $\partial \Omega$ such that the line segment $[p,q] := \{p+t(q-p) : t\in [0,1]\}$ has intersections with $\ol\Omega$ only at $p$ and $q$. Then $p$ and $q$ belong to a connected component $M$ of $\partial \Omega$, and there exists a curve $\gamma: [0,1] \to M$ so that $\gamma(0)=p$, $\gamma(1)=q$, and
\begin{equation}
\label{eq:dM}
    \int_0^1 |\dot\gamma(t)|\,dt \le C_b |p-q|
\end{equation}
for some constant $C_b$ that depends only on $n$ and $\partial \Omega$.
\end{lem}

The above (after some generalization) shows
\begin{equation*}
    d_{\partial \Omega}(x,y) \le C_b |x-y|, \qquad \text{ if $x,y\in \partial \Omega$ and $x\sim y$}.
\end{equation*}
Here $x\sim y$ means they belong to a connected component of $\partial \Omega$, and the distance $d_{\partial \Omega}(x,y)$ is defined to be the minimal length of curves in $\partial \Omega$ joining $x$ and $y$. Note also that this estimate still holds with the same $C_b$ even after the $\ep$-rescaling. The result above should be compared with  Lemma 2.6 of \cite{HorieIshii1998} which says there is a constant $C>0$ so that $d_{\ol\Omega}(x,y) \le C|x-y|$.

\begin{proof}
We denote by $(p,q)$ the set $\{p+t(q-p) : t\in (0,1)\}$. Because $p,q\in \partial\Omega$ and $(p,q)\in \ol \Omega^c$, the segment $(p,q)$ must be in a connected component $D$ of $\Omext$. Let $M$ be the boundary of this component. Then $M$ must be connected. 

Without loss of generality, assume $D$ has a nonempty intersection with $Y$. Let $\ell$ be a fixed integer greater than $4\sqrt{n}$, say $\ell =3n$, then $M \cap \ell Y$ is compact and hence there exists a finite open cover $\{B_r (x_i) \subset \mathbb{R}^n: i = 1, 2, \cdots, K\}$ consisting of open balls of some radius $r>0$ centered at $\{x_i\}$, and each $x_i$ is in $M\cap \ell Y$, such that for each $i \in \{1, 2, \cdots, K\}$, $\Gamma_i := M \cap \ell Y \cap \ol B_{2r}\left(x_i\right)$ is the graph of a $C^1$ mapping in some coordinate system. Moreover, the projection of this graph to the domain of definition of the mapping is a compact domain of $\R^{n-1}$, and hence we can find a constant $C_1 >0$ (uniform for all $i$'s) so that any two points $x,y$ in the same graph $\Gamma_i$ can be connected by a $C^1$ path $\gamma_{x,y} : [0,1]\to \Gamma_i$, $\gamma(0)=x$, $\gamma(1)=y$, so that
$$
\mathrm{arc\,length}(\gamma_{x,y}) = \int_0^1 |\dot\gamma(t)|\,dt \le C_1|x-y|.
$$
Finally, for any pair of points $p, q\in M\cap \ell Y$, if $|p-q|<r$, then we can find some $x_i$ so that $p,q$ are in $\Gamma_i$. Then the above and hence \eqref{eq:dM} hold. If $|p-q|\ge r$ and $p,q$ do not belong to any same element of $\{B_{2r}(x_i)\}$, then we can find a piecewise curve $\gamma_{p,q}$ joining them so that each piece stays in an element of $\{\Gamma_i\}$. Clearly, we need at most $K$ such pieces. Hence the arc length of $\gamma_{p,q}$ can be bounded by
$$
KC_1(4r) \le C_2|p-q|,
$$
where $C_2 = 4KC_1$ is a constant depending only on $n$ and $\partial \Omega$ (and is uniform for all $p,q \in M\cap \ell Y$).

\medskip

Now for any segment $[p,q]$ that intersects with $M$ only at $p,q$, we aim to construct a curve $\gamma$ in $M$ joining them with controlled arc length. We consider two cases.
\smallskip

\emph{Case 1. $p$ and $q$ are relatively close: $|p-q|\le 2\sqrt{n}$}. By periodic translation, we may assume that $p \in Y$ and hence $[p,q] \subset \ell Y$. Then by the argument above, we find the desired path joining $p,q$, with arc length bounded by $C_2|p-q|$. 
\smallskip

\emph{Case 2. $p$ and $q$ are relatively far away: $|p-q|>2\sqrt{n}$}. Let $N$ be the integer such that $N\sqrt{n} <|p-q| \le (N+1)\sqrt{n}$; note that $N <|p-q|/\sqrt{n}$. Set
$$
x_i = p + \left(\frac{|p-q|-N\sqrt{n}}{2} + \left(i-\frac12\right)\sqrt{n}\right)\frac{q-p}{|q-p|}, \quad 1 \le i \le N.
$$ 
We observe that $\{x_i\}_{i=1}^N$ are evenly distributed points in $(p,q)$, and $|x_i-x_{i+1}|=\sqrt{n}$, for each $i \in \{1,\dots,N\}$. Also, we check
$$
\sqrt{n}/2 < |p-x_1| < \sqrt{n}, \quad  \sqrt{n}/2  < |q-x_N| < \sqrt{n}.
$$

For each $1\le i \le N$, the ball $\ol B_{\sqrt{n}/2}(x_i)$ has nonempty intersection with $M$ because the component $D$ of $\Omext$ enclosed by $M$ cannot contain a copy of the unit cube. So we can find $p_i \in M \cap \ol B_{\sqrt{n}/2}(x_i)$. Note that for different $i,j\in \{1,\dots,N\}$, the balls $\ol B_{\sqrt{n}/2}(x_i)$ and $\ol B_{\sqrt{n}/2}(x_j)$ can intersect only in $(p,q)$ which is in $M^c$; therefore, $p_i\ne p_j$. Similarly, $p_i$'s are different from $p,q$. Set $p_0=p$ and $p_{N+1}=q$, then $\{p_i\}_{i=0}^{N+1}\subset M$ satisfy
$$
|p_i-p_{i+1}| \le 2\sqrt{n}, \qquad \text{ for } 0\le i \le N.
$$
For each $0\le i\le N$, now that $p_i$ and $p_{i+1}$ are two points in $M$ that are within distance $2\sqrt{n}$, we can use the result of Case 1 and find a curve $\gamma_i$ in $M$ joining $p_i$ and $p_{i+1}$ and whose arc length is bounded by $ 2\sqrt{n}C_2$. Combine those curves we finally get a curve in $M$ joining $p$ to $q$ whose arc length is bounded by
\begin{equation*}
(N+1) 2\sqrt{n}C_2 \le 3C_2 |p-q|. 
\end{equation*}
Setting $C_b=3C_2$ we see that the proof is complete.
\end{proof}

We next give a proof of Lemma \ref{lem:velocity bound}.
\begin{proof}[Proof of Lemma \ref{lem:velocity bound}]
Define a new path $\tilde{\gamma}: [0, t] \to \overline{\Omega}_\varepsilon$ by
$\tilde{\gamma}(s): = \varepsilon\gamma\left(\frac{s}{\varepsilon}\right)$ for $s \in [0, t]$. Then we have $\tilde{\gamma}(t)=x$ and
\[
u^\varepsilon(x, t)= \int_{0}^t L\left(\frac{\tilde{\gamma}(s)}{\varepsilon}, \dot{\tilde{\gamma}}(s)\right)ds +g(\tilde{\gamma}(0)),
\]
and $\tilde{\gamma}$ is an optimal path for $u^\varepsilon(x, t)$ as
\[u^\varepsilon (x, t) = \inf \left\lbrace \int_0^t L\left( \frac{\xi(s)}{\varepsilon} ,\dot{\xi}(s)\right) \,ds+g\left(\xi(0)\right): \xi\in \mathrm{AC}([0,t];\overline{\Omega}_\varepsilon), \xi(t) = x\right\rbrace.\]
Note that $\dot{\tilde{\gamma}} (s)=\dot{\gamma}\left(\frac{s}{\varepsilon}\right)$ for $s \in [0, t]$. Therefore, it suffices to prove
\[
\left\|\dot{\tilde{\gamma}}\right\|_{L^\infty[0,t]} \leq M_0
\]
for some constant $M_0=M_0\left(n,\partial\Omega, H, \|Dg\|_{L^\infty(\mathbb{R}^n)}\right)>0$.

Suppose $\tilde{\gamma}$ is differentiable at $t_0 \in (0,t)$. For $t_1 \in (t_0, t)$, by the dynamic programming principle, we have
\[
u^\varepsilon(\tilde{\gamma}(t_1), t_1)= \int_{t_0}^{t_1} L \left( \frac{\tilde{\gamma}(s)}{\varepsilon}, \dot{\tilde{\gamma}}(s)\right) \, ds + u^\varepsilon\left(\tilde{\gamma}(t_0), t_0\right),
\]
which implies
\begin{equation}\label{eqn:udq_int}
\frac{u^\varepsilon\left(\tilde{\gamma}(t_1), t_1\right)-u^\varepsilon\left(\tilde{\gamma}(t_0), t_0\right)}{t_1-t_0} =\frac{1}{t_1-t_0} \int_{t_0}^{t_1} L \left( \frac{\tilde{\gamma}(s)}{\varepsilon}, \dot{\tilde{\gamma}}(s)\right) \, ds.
\end{equation}
On the other hand, for $t_1 \in (t_0, t)$ and sufficiently close to $t_0$, 
\begin{equation}\label{eqn:udq_estimate}
    \begin{aligned}
&\frac{u^\varepsilon\left(\tilde{\gamma}(t_1), t_1\right)-u^\varepsilon\left(\tilde{\gamma}(t_0), t_0\right)}{t_1-t_0} \\
= \ & \frac{u^\varepsilon\left(\tilde{\gamma}(t_1), t_1\right)-u^\varepsilon\left(\tilde{\gamma}(t_0), t_1\right) 
+ u^\varepsilon\left(\tilde{\gamma}(t_0), t_1\right)-u^\varepsilon\left(\tilde{\gamma}(t_0), t_0\right)}{t_1-t_0}\\
\leq \ & \frac{\left|u^\varepsilon\left(\tilde{\gamma}(t_1), t_1\right)-u^\varepsilon\left(\tilde{\gamma}(t_0), t_1\right)\right|}{\left|\tilde{\gamma}(t_1)-\tilde{\gamma}(t_0)\right|} \cdot\frac{\left|\tilde{\gamma}(t_1)-\tilde{\gamma}(t_0)\right|}{t_1-t_0}+ \frac{\left|u^\varepsilon\left(\tilde{\gamma}(t_0), t_1\right)-u^\varepsilon\left(\tilde{\gamma}(t_0), t_0\right)\right|}{t_1-t_0}\\
\leq\  & C \frac{\left|\tilde{\gamma}(t_1)-\tilde{\gamma}(t_0)\right|}{t_1-t_0}+C,
    \end{aligned}
\end{equation}
where the last inequality and the constant $C=C\left(n,\partial \Omega, H,\|Dg\|_{L^\infty(\R^n)}\right)>0$ comes from \eqref{eqn:u_prior} and Lemma \ref{lem:pOmetric}. Combining \eqref{eqn:udq_int} and \eqref{eqn:udq_estimate}, we obtain
\[
\frac{1}{t_1-t_0} \int_{t_0}^{t_1} L \left( \frac{\tilde{\gamma}(s)}{\varepsilon}, \dot{\tilde{\gamma}}(s)\right) \, ds \leq C \frac{\left|\tilde{\gamma}(t_1)-\tilde{\gamma}(t_0)\right|}{t_1-t_0}+C.
\]
Taking the limit as $t_1$ goes to $t_0$, we get
\begin{equation}\label{eqn:Lupper}
L \left( \frac{\tilde{\gamma}(t_0)}{\varepsilon}, \dot{\tilde{\gamma}}(t_0)\right) \leq C  \left|\dot{\tilde{\gamma}}(t_0) \right|+C.
\end{equation}
From \eqref{eqn:K_0L}, we also have 
\begin{equation}\label{eqn:Llower}
L \left( \frac{\tilde{\gamma}(t_0)}{\varepsilon}, \dot{\tilde{\gamma}}(t_0)\right)\geq \frac{\left|\dot{\tilde{\gamma}}(t_0)\right|^2}{2}-K_0.    
\end{equation}
From \eqref{eqn:Lupper} and \eqref{eqn:Llower}, it follows that there exists a constant $M_0>0$ dependent on $n,\partial\Omega,H,\|Dg\|_{L^\infty(\R^n)}$ such that
\[
\left|\dot{\tilde{\gamma}}(t_0)\right| \leq M_0.
\]
Since $t_0$ is an arbitrary point where $\tilde{\gamma}$ is differentiable, we conclude
\[
\left\|\dot{\tilde{\gamma}}\right\|_{L^\infty[0,t]} \leq M_0.
\]
  
\end{proof}

\begin{rem}
In fact, upon careful analysis, the constants $C$ in \eqref{eqn:udq_estimate} and $M_0$ can be made independent of both $n$ and $\partial \Omega$. However, as this lies outside our primary focus, we leave this to the reader.
\end{rem}

\begin{proof}[Proof of Lemma \ref{lem:velocity bound-md}]
    The proof of this lemma follows that of Lemma \ref{lem:velocity bound} verbatim and hence is omitted.
\end{proof}

We prove the inf-sup formula for $\ol H_\Om$.
\begin{lem}\label{lem:inf-sup}
Assume {\rm(A1)--(A3)}. 
We have, for $p\in \R^n$,
\[
 \ol H_\Om(p) = \inf_{\varphi \in \Lip(\T^n)} \esssup_{y\in \Om} H(y,p+D\varphi(y)). 
\]
\end{lem}

\begin{proof}
Fix $p\in \R^n$.
By using the vanishing discount procedure for the state constraint problem (see e.g., \cite{CDL,Mitake2008, IMT}), we can find 
$\left(\overline{H}_\Omega(p), v\right)\in\R\times\Lip(\overline{\Omega})$ satisfying 
\begin{equation}\label{eqn:cell-appendix}
    \left\{\begin{aligned}
    H\left(y, p+Dv(y)\right) &\leq \overline{H}_\Om(p) \quad \text{in }  \Omega, \\
    H\left(y, p+Dv(y)\right) &\geq \overline{H}_\Om(p) \quad \text{on }  \overline{\Omega}.
    \end{aligned}
    \right.
\end{equation}
in the sense of viscosity solutions. 
Here, $v$ is $\Z^n$-periodic.
Extend $v$ to a function $\tilde v\in \Lip(\R^n)$ which is $\Z^n$-periodic.
Then,
\[
\esssup_{y\in \Om} H(y,p+D\tilde v(y))
=\esssup_{y\in \Om} H(y,p+Dv(y)) = \ol H_\Om(p).
\]
Hence,
\[
 \ol H_\Om(p) \geq \inf_{\varphi \in \Lip(\T^n)} \esssup_{y\in \Om} H(y,p+D\varphi(y)). 
\]
We now prove the converse inequality.
Assume otherwise that there exist $\varphi\in \Lip(\T^n)$ and $\delta>0$ such that
\[
\esssup_{y\in \Om} H(y,p+D\varphi(y)) < \ol H_\Om(p)-2\delta.
\]
For $\lam>0$ sufficiently small, we see that
\begin{equation*}
    \left\{\begin{aligned}
   \lam \varphi(y)+ H\left(y, p+\varphi(y)\right) &\leq \overline{H}_\Om(p)-\delta \quad \text{in }  \Omega, \\
    \lam v(y)+ H\left(y, p+Dv(y)\right) &\geq \overline{H}_\Om(p)-\delta \quad \text{on }  \overline{\Omega}.
    \end{aligned}
    \right.
\end{equation*}
By the usual comparison principle, we yield $v\geq \varphi$.
By the same argument, $v-C\geq \varphi$ for any $C\in \R$, which leads to a contradiction. 
\end{proof}

\begin{rem}
    Although we have the inf-sup formula for $\ol H_\Om$, it is not too clear for us whether the following inf-max formula holds
    \[
    \ol H_\Omega(p) = \inf_{\varphi\in C^1(\T^n)} \max_{x\in \ol \Om} H(y,p+D\varphi(y))\ ?
    \]
    A technical difficulty here is that, for $v\in \Lip(\ol \Om)$, $\Z^n$-periodic solving \eqref{eqn:cell-appendix}, it is not easy to smooth it up using convolutions with a standard mollifier because of the appearance of $\partial \Omega$.
\end{rem}


\section*{Data availability}
Data sharing does not apply to this article as no datasets were generated or analyzed during the current study.

\section*{Conflict of interest}
There is no conflict of interest.

\section*{Acknowledgement}
HVT would like to thank Van Hai Van for her drawing of Figure \ref{figVan}.

\end{document}